\documentclass{compositio}
\usepackage{stmaryrd}
\usepackage{amsmath, amscd, amssymb, mathrsfs, accents, amsfonts}
\usepackage{url}
\usepackage[all]{xy}
\usepackage{longtable}
\usepackage{dsfont}
\usepackage{tikz}

\definecolor{Myblue}{rgb}{0,0,0.6}
\usepackage[a4paper,colorlinks,citecolor=Myblue,linkcolor=Myblue,urlcolor=Myblue,pdfpagemode=None]{hyperref}

\SelectTips{cm}{}

\newtheorem{theorem}{Theorem}[section]
\newtheorem{proposition}[theorem]{Proposition}
\newtheorem{lemma}[theorem]{Lemma}
\newtheorem{corollary}[theorem]{Corollary}

\theoremstyle{definition}
\newtheorem{definition}[theorem]{Definition}
\newtheorem{example}[theorem]{Example}
\newtheorem{remark}[theorem]{Remark}

\numberwithin{equation}{section}

\def\res{\operatorname{Res}}

\def\inc{\operatorname{inc}}

\def\Im{\operatorname{Im}}

\def\can{\operatorname{can}}

\def\K{\mathbf{K}}

\def\Hom{\operatorname{Hom}}

\def\Grmodd{\operatorname{GrMod}}

\DeclareMathOperator{\hmf}{hmf}
\DeclareMathOperator{\HMF}{HMF}
\DeclareMathOperator{\HF}{HF}

\DeclareMathOperator{\At}{At}

\begin{document}

\def\Res{\res\!}
\newcommand{\cat}[1]{\mathcal{#1}}
\newcommand{\lto}{\longrightarrow}
\newcommand{\xlto}[1]{\stackrel{#1}\lto}
\newcommand{\mf}[1]{\mathfrak{#1}}
\newcommand{\md}[1]{\mathscr{#1}}
\newcommand{\intvar}{\bs{x}_{\textup{int}}}
\newcommand{\extvar}{\bs{x}_{\textup{ext}}}
\newcommand{\qderu}[2]{\mathbf{D}^{#1}(#2)}
\newcommand{\ud}{\mathrm{d}}
\def\l{\,|\,}
\def\cf{\boldsymbol{cf}}
\def\bx{\boldsymbol{x}}
\def\by{\boldsymbol{y}}
\def\ba{\boldsymbol{a}}
\def\bb{\boldsymbol{b}}
\def\totimes{\otimes}
\def\di{Q}
\newcommand{\cotimes}[1]{\,\widehat{\otimes}_{#1}\,}
\def\QQ{\mathds{Q}}
\def\krc{C}
\def\diffm{d}
\def\diffh{d_{\chi}}
\def\redh{\overline{H}}
\def\ZZ{\mathds{Z}}
\def\bs{\boldsymbol}
\def\Ztwo{\mathds{Z}_2}
\def\mdual{^{\vee}}
\def\KR{\operatorname{KR}}
\def\I{\!\operatorname{i}\!}
\def\E{\operatorname{e}\!}
\def\sln{\mathfrak{sl}(N)}
\def\nN{\mathds{N}}
\def\nZ{\mathds{Z}}
\def\nQ{\mathds{Q}}
\def\nR{\mathds{R}}
\def\nC{\mathds{C}}
\def\idem{\epsilon}
\def\Xcirc{%
\begin{tikzpicture}[inner sep=0mm]
\node (X) at (0,0) {$X$};
\node (0) at (0,0) [circle,inner sep=0.99pt, thin,draw=black,fill= white] {};
\end{tikzpicture}%
}
\def\Xbul{%
\begin{tikzpicture}[inner sep=0mm]
\node (X) at (0,0) {$X$};
\node (0) at (0,0) [circle,inner sep=0.99pt, thin,draw=black,fill= black] {};
\end{tikzpicture}%
}

\title{Computing Khovanov-Rozansky homology \\ and defect fusion}
\author{Nils Carqueville}
\email{nils.carqueville@physik.uni-muenchen.de}
\address{Arnold Sommerfeld Center for Theoretical Physics, LMU M\"unchen \& Excellence Cluster Universe}

\author{Daniel Murfet}
\email{daniel.murfet@math.ucla.edu}
\address{Department of Mathematics, UCLA}


\begin{abstract}
We compute the categorified $\sln$ link invariants as defined by Khovanov and Rozansky, for various links and values of~$N$. This is made tractable by an algorithm for reducing tensor products of matrix factorisations to finite rank, which we implement in the computer algebra package Singular. 
\end{abstract}

\maketitle

\section{Introduction}

In this paper we present the first method to directly compute Khovanov and Rozansky's $\sln$ link homology for arbitrary links. This is made possible by our implementation in the computer algebra system Singular of a technique for explicitly fusing topological defects in Landau-Ginzburg models or, what is the same, reducing tensor products of matrix factorisations to finite-rank. 

Before explaining our approach, we recall the definition of Khovanov and Rozansky's $\sln$ link homology. The starting point is the link diagram $D$ of an oriented link $L$. If there are~$m$ strands, then one introduces ``potentials'' $x_{i}^{N+1}$, $i\in\{1,\ldots,m\}$, one for each strand. In the example of the Hopf link we can decorate the link diagram $D$ as follows:
$$
\begin{tikzpicture}[scale=1.0,baseline, inner sep=1mm, >=stealth]
\node (1) at (-2.9,0) [circle,inner sep=1.5pt, draw=white,fill= white] {{\normalsize $x^{N+1}_{1}$}};
\node (2) at (-1.5,0) [circle,inner sep=1.5pt, draw=white,fill= white] {{\normalsize $x^{N+1}_{2}$}};
\node (3) at (0.15,0) [circle,inner sep=1.5pt, draw=white,fill= white] {{\normalsize $x^{N+1}_{3}$}};
\node (4) at (1.55,0) [circle,inner sep=1.5pt, draw=white,fill= white] {{\normalsize $x^{N+1}_{4}$}};
\draw[ ->, thick, shift={(-1.4,0)} ]  (-41:1) arc (-41:310:1); 
\draw[ ->, thick ] (140:1) arc (-220:130:1); 
\end{tikzpicture}
$$
Next one ``resolves'' crossings by replacing
\begin{minipage}{0.48cm}
\begin{tikzpicture}[scale=0.23,baseline, inner sep=0.1mm, >=stealth]
\node (0) at (0,0) [circle,inner sep=1.5pt, draw=white,fill= white] {};
\node (bl) at (-1,-1) [circle,draw=white,fill= white] {};
\node (br) at (1,-1) [circle,draw=white,fill= white] {};
\node (tl) at (-1,1) [circle,draw=white,fill= white] {};
\node (tr) at (1,1) [circle,draw=white,fill= white] {};
\draw[-] (bl) -- (0); 
\draw[->] (br) -- (tl); 
\draw[->] (0) -- (tr); 
\end{tikzpicture}
\end{minipage}
 and 
\begin{minipage}{0.48cm}     
\begin{tikzpicture}[scale=0.23,baseline, inner sep=0.1mm, >=stealth]
\node (0) at (0,0) [circle,inner sep=1.5pt, draw=white,fill= white] {};
\node (bl) at (-1,-1) [circle,draw=white,fill= white] {};
\node (br) at (1,-1) [circle,draw=white,fill= white] {};
\node (tl) at (-1,1) [circle,draw=white,fill= white] {};
\node (tr) at (1,1) [circle,draw=white,fill= white] {};
\draw[->] (bl) -- (tr); 
\draw[-] (br) -- (0); 
\draw[->] (0) -- (tl); 
\end{tikzpicture}
\end{minipage} 
by certain combinations of the local diagrams 
\begin{minipage}{0.36cm}
\begin{tikzpicture}[scale=0.23,baseline, inner sep=0.1mm, >=stealth]
\draw[ ->, shift={(-2.5,0)} ]  (-50:1) arc (-50:50:1); 
\draw[ -> ] (230:1) arc (230:130:1); 
\end{tikzpicture}
\end{minipage}
 and 
 \begin{minipage}{0.48cm}
\begin{tikzpicture}[scale=0.23,baseline, inner sep=0.1mm, >=stealth]
\node (0) at (0,0) [circle,inner sep=0.7pt, thin,draw=black,fill= black] {};
\node (bl) at (-1,-1) [circle] {};
\node (br) at (1,-1) [circle] {};
\node (tl) at (-1,1) [circle] {};
\node (tr) at (1,1) [circle] {};
\draw[-] (bl) -- (0); 
\draw[-] (br) -- (0); 
\draw[->] (0) -- (tl); 
\draw[->] (0) -- (tr); 
\end{tikzpicture}
\end{minipage}. \footnote{We do not use the (representation theoretically more appropriate) ``wide edge'' depiction of~\cite{kr0401268} for the second diagram. This is done to facilitate the interpretation in terms of Landau-Ginzburg models below.}
More precisely, one associates complexes of matrix factorisations to each crossing: the construction of~\cite{kr0401268} provides an assignment
\begin{equation}
\label{DeltaX}
\begin{tikzpicture}[scale=0.7,baseline, inner sep=1mm, >=stealth]
\node (1) at (-2.5,1.0) [circle,inner sep=1.5pt, draw=white,fill= white] {$x^{N+1}_{i}$};
\node (2) at (0.1,1.0) [circle,inner sep=1.5pt, draw=white,fill= white] {$x^{N+1}_{j}$};
\node (3) at (0.1,-1.1) [circle,inner sep=1.5pt, draw=white,fill= white] {$x^{N+1}_{k}$};
\node (3) at (-2.5,-1.1) [circle,inner sep=1.5pt, draw=white,fill= white] {$x^{N+1}_{l}$};
\draw[ ->, thick, shift={(-2.5,0)} ]  (-50:1) arc (-50:50:1); 
\draw[ ->, thick ] (230:1) arc (230:130:1); 
\end{tikzpicture}
\widehat = \,
\Xcirc \, ,
\qquad 
\begin{tikzpicture}[scale=0.7,baseline, inner sep=1mm, >=stealth]
\node (1) at (-1.1,1.0) [circle,inner sep=1.5pt, draw=white,fill= white] {$x^{N+1}_{i}$};
\node (2) at (1.5,1.0) [circle,inner sep=1.5pt, draw=white,fill= white] {$x^{N+1}_{j}$};
\node (3) at (1.5,-1.2) [circle,inner sep=1.5pt, draw=white,fill= white] {$x^{N+1}_{k}$};
\node (3) at (-1.1,-1.2) [circle,inner sep=1.5pt, draw=white,fill= white] {$x^{N+1}_{l}$};
\node (0) at (0,0) [circle,inner sep=1.5pt, thin,draw=black,fill= black] {};
\node (bl) at (-1,-1) [circle] {};
\node (br) at (1,-1) [circle] {};
\node (tl) at (-1,1) [circle] {};
\node (tr) at (1,1) [circle] {};
\draw[-,  thick] (bl) -- (0); 
\draw[-,  thick] (br) -- (0); 
\draw[->,  thick] (0) -- (tl); 
\draw[->,  thick] (0) -- (tr); 
\end{tikzpicture}
\widehat = \,
\Xbul
\end{equation}
where $\Xcirc$ and $\Xbul$ are certain $(4\times 4)$-matrices with entries in $\nQ[x_{i},x_{j},x_{k},x_{l}]$ such that $\Xcirc^2=\Xbul^2 = (x^{N+1}_{i}+x^{N+1}_{j}-x^{N+1}_{k}-x^{N+1}_{l})\cdot\operatorname{id}_{4\times 4}$. In other words, $\Xcirc$ and~$\Xbul$ are matrix factorisations of the sum of outgoing potentials minus the sum of incoming potentials.

The ``resolution'' assigns two-term complexes of matrix factorisations to over- and under-crossings: 
\begin{equation}
\label{resolutions}
\begin{tikzpicture}[scale=0.5,baseline, inner sep=1mm, >=stealth]
\node (0) at (0,0) [circle,inner sep=1.5pt, draw=white,fill= white] {};
\node (bl) at (-1,-1) [circle,draw=white,fill= white] {};
\node (br) at (1,-1) [circle,draw=white,fill= white] {};
\node (tl) at (-1,1) [circle,draw=white,fill= white] {};
\node (tr) at (1,1) [circle,draw=white,fill= white] {};
\draw[-, thick] (bl) -- (0); 
\draw[->, thick] (br) -- (tl); 
\draw[->, thick] (0) -- (tr); 
\end{tikzpicture}
\,\widehat =\,
\Big(
\xymatrix{%
0 \vphantom{\underline{X}} \ar[r] & \underline{\Xcirc} \ar[r] & \Xbul \vphantom{\underline{X}} \ar[r] & 0 \vphantom{\underline{X}} \Big) 
}
,  \qquad
\begin{tikzpicture}[scale=0.5,baseline, inner sep=1mm, >=stealth]
\node (0) at (0,0) [circle,inner sep=1.5pt, draw=white,fill= white] {};
\node (bl) at (-1,-1) [circle,draw=white,fill= white] {};
\node (br) at (1,-1) [circle,draw=white,fill= white] {};
\node (tl) at (-1,1) [circle,draw=white,fill= white] {};
\node (tr) at (1,1) [circle,draw=white,fill= white] {};
\draw[->, thick] (bl) -- (tr); 
\draw[-, thick] (br) -- (0); 
\draw[->, thick] (0) -- (tl); 
\end{tikzpicture}
\,\widehat =\,
\Big(
\xymatrix{%
0 \vphantom{\underline{X}} \ar[r] & \Xbul \vphantom{\underline{X}} \ar[r] & \underline{\Xcirc} \ar[r] & 0 \vphantom{\underline{X}} \Big) 
}
.
\end{equation}
where we suppress certain shifts with respect to an internal grading of~$\Xcirc$ and~$\Xbul$ and the underlined component has cohomological degree zero. For each crossing the matrices $\Xcirc,\Xbul$ depend only on the variables $x_{i},x_{j},x_{k},x_{l}$ occurring in the incoming and outgoing potentials. 

For two general matrix factorisations $X,X'$ of polynomials $W,W'$, respectively, there is a tensor product matrix factorisation $X\otimes X'$ of $W+W'$. To each crossing in a link diagram~$D$ of a link~$L$ one associates the appropriate complex in~\eqref{resolutions}, and then one takes the tensor product of all these complexes to produce the total complex $C(D)$. It has a bigrading which is induced by the cohomological grading of the two-term complexes~\eqref{resolutions} and the internal grading of the matrix factorisations. The main result of~\cite{kr0401268} can now be summarised by the fact that the cohomology $H(L)$ of $C(D)$ is an invariant of~$L$ which is conveniently encoded in the Poincar\'e polynomial
$$
\KR_{N}(L) = \sum_{i,j\in\nZ} t^{i} q^j \dim_{\nQ}(H^{i,j}(L)) \, . 
$$
This is indeed a more refined invariant than the uncategorified $\sln$ link invariant of~\cite{RT1990} which can be recovered as the graded Euler characteristic $\KR_{N}(L)|_{t=-1}$. There is also a ``reduced'' version of the above construction which leads to link invariants $\overline{\KR}_{N}(L,K)$ that \textsl{a priori} depend on the choice of a component~$K$ of~$L$. 

\medskip

Since the definition of Khovanov-Rozansky homology is very explicit, one might expect that these invariants can be computed in a straightforward way. However, there is a technical impediment that has prevented direct computations until now.

In a link diagram any oriented strand is incoming and outgoing with respect to some crossing, so $C(D)$ is a complex of matrix factorisations of zero, that is, it is a complex in the category of $\mathbb{Z}_2$-graded complexes. Each of these $\mathbb{Z}_2$-graded complexes has a differential given by a matrix which involves the variables~$x_i$ labelling edges in the link diagram, so determining $H(L)$ means computing the cohomology of large square-zero matrices in many variables. A computer algebra package such as Singular computes cohomology by finding Gr\"obner bases, and in the case of Khovanov-Rozansky homology the number and the degrees of the generators of the ideals involved are such that naive approaches exhaust memory and computing power very quickly, making it impossible to determine invariants in all but the simplest examples this way. We mitigate this problem by inserting an intermediate step, which we call \emph{web compilation}, based on techniques developed in \cite{dm1102.2957}.

To see the problem more clearly, let us consider the general situation of two matrix factorisations $X,X'$ of $W_{1}(\boldsymbol{x})-W_{2}(\boldsymbol{y})$ and $W_{2}(\boldsymbol{y})-W_{3}(\boldsymbol{z})$, respectively, where~$W_{2}$ has an isolated singularity at the origin. Then $X\otimes X'$ is a matrix factorisation of $W_{1}(\boldsymbol{x})-W_{3}(\boldsymbol{z})$, but it is of infinite rank over $\nQ[\boldsymbol{x},\boldsymbol{z}]$ as it depends also on the variables~$\boldsymbol{y}$. It is a basic fact that $X\otimes X'$ is isomorphic to an object~$F$ of finite rank in the triangulated category of matrix factorisations of $W_{1}(\boldsymbol{x})-W_{3}(\boldsymbol{z})$ over $\nQ[\boldsymbol{x},\boldsymbol{z}]$. However, only recently a general construction of~$F$ as well as the isomorphism $X\otimes X'\longrightarrow F$ was given in~\cite{dm1102.2957}. We reformulate the details of the construction in Section~\ref{compilewebs}, but the main idea is that first one ``inflates'' the matrix of the differential of $X\otimes X'$ by replacing each $\bs{y}$-monomial by the matrix that represents its multiplication on the Jacobian $\nQ[\boldsymbol{y}]/(\partial_{y_{i}} W_{2})$ in some chosen basis. The inflated matrix $B$ is the differential of a finite rank matrix factorisation, and is endowed with an idempotent $\idem=\vartheta\circ\psi$. The finite rank matrix factorisation $F$ homotopy equivalent to $X \otimes X'$ is given by the splitting of~$\idem$ (up to shifts). 

We have implemented the above construction using the computer algebra system Singular~\cite{cmWebCompileCode}. This effectively computes tensor products of finite-rank matrix factorisations and, in particular, Khovanov-Rozansky homology. We spell out the details of the construction in Section~\ref{compilewebs} and in Section~\ref{compres} we present the results for link invariants obtained in this way. In particular we compute unreduced and reduced Khovanov-Rozansky homology for all prime links with up to 6 crossings for various values of~$N$. 

Our method could be straightforwardly applied to compute the coloured homological $\sln$ link invariants of~\cite{w0907.0695}, to the categorification of the Kauffman polynomial and to virtual links~\cite{kr0701333}. Unfortunately our approach is not applicable to the categorification~\cite{kr0505056} of the Homfly polynomial, as the potentials that occur in this theory are not isolated singularities. 

\medskip

Let us mention another motivation: the fusion of topological defects in Landau-Ginzburg models. The latter are two-dimensional topological field theories which govern the ``physics'' of domains on a worldsheet. The boundary conditions at the line separating two neighbouring domains are called defects, see e.\,g.~\cite{k1004.2307, dkr1107.0495}. If two such domains are governed by Landau-Ginzburg models with potentials $W_{1}$ and $W_{2}$, then the defects are described by matrix factorisations~$X$ of $W_{1}-W_{2}$~\cite{br0707.0922}.

One may also consider more than two domains and defects between them. For example, there can be a defect~$X$ between theories $W_{1}$ and $W_{2}$, and another defect~$X'$ between $W_{2}$ and $W_{3}$. Because of their topological nature (and in the absence of field insertions in the domain with potential~$W_{2}$) the defects may be moved arbitrarily close to one another. The limit of this process is called the fusion of the two defects, and it is described by the matrix factorisation $X\otimes X'$~\cite{br0707.0922}. Computing a finite-rank representative of this tensor product in the homotopy category of matrix factorisations represents the process of ``integrating out'' the ``unphysical'' variables on which $W_2$ depends.
In this way fusion endows the set of all Landau-Ginzburg models with the structure of a bicategory~\cite{Calinetal2, McNameethesis, cr0909.4381, cr1006.5609}. 

With this interpretation in mind let us now briefly revisit the resolutions~\eqref{resolutions} in terms of the matrix factorisations~\eqref{DeltaX} that are at the heart of the Khovanov-Rozansky construction. It is natural to interpret~$\Xcirc$ and~$\Xbul$ as defects between Landau-Ginzburg models with potentials $x^{N+1}_{i}+x^{N+1}_{j}$ and $x^{N+1}_{k}+x^{N+1}_{l}$. In this sense computing Khovanov-Rozansky invariants boils down to repeatedly fusing defects in Landau-Ginzburg models. 

From this point of view, once one has replaced each crossing in a link diagram with either of the two resolutions 
\begin{minipage}{0.36cm}
\begin{tikzpicture}[scale=0.23,baseline, inner sep=0.1mm, >=stealth]
\draw[ ->, shift={(-2.5,0)} ]  (-50:1) arc (-50:50:1); 
\draw[ -> ] (230:1) arc (230:130:1); 
\end{tikzpicture}
\end{minipage}
 or 
 \begin{minipage}{0.36cm}
\begin{tikzpicture}[scale=0.23,baseline, inner sep=0.1mm, >=stealth]
\node (0) at (0,0) [circle,inner sep=0.7pt, thin,draw=black,fill= black] {};
\node (bl) at (-1,-1) [circle] {};
\node (br) at (1,-1) [circle] {};
\node (tl) at (-1,1) [circle] {};
\node (tr) at (1,1) [circle] {};
\draw[-] (bl) -- (0); 
\draw[-] (br) -- (0); 
\draw[->] (0) -- (tl); 
\draw[->] (0) -- (tr); 
\end{tikzpicture}
\end{minipage}, we can interpret the resulting graph as a slice of a worldsheet foam~\cite{kr0404189, msv0708.2228} with defect lines: to every edge~$i$ we associate a Landau-Ginzburg model with potential $x_{i}^{N+1}$, and to every point we associate a defect. To a point on an edge~$i$ we assign the invisible defect~$I$, and to a point that is a vertex we associate the matrix factorisation~$\Xbul$ as before. This perspective agrees with the construction of~\cite{kr0401268} as the matrix factorisation~$\Xcirc$ associated to 
\begin{minipage}{0.3cm}
\begin{tikzpicture}[scale=0.23,baseline, inner sep=0.1mm, >=stealth]
\draw[ ->, shift={(-2.5,0)} ]  (-50:1) arc (-50:50:1); 
\draw[ -> ] (230:1) arc (230:130:1); 
\end{tikzpicture}
\end{minipage}
is given by the (external) tensor product $I\otimes I$.

\medskip

The paper is organised as follows. In Section~\ref{preliminaries} we review basic facts about graded matrix factorisations and fix our notation. In Section~\ref{compilewebs} we present our method for computing Khovanov-Rozansky homology, taking for granted the matrix factorisations $\Xcirc$ and $\Xbul$. In Appendix \ref{section:constrkr} we provide an alternative construction of this basic data in terms of autoequivalences of triangulated categories arising from adjunctions of symmetric polynomials.
In Section~\ref{compres} we
present a discussion of our results for Khovanov-Rozansky homology of links with up to 6 crossings and three components.

\begin{acknowledgements}
We thank N.~Behr, S.~Cautis, S.~Gukov, M.~Khovanov, T.~Licata, A.~Morozov, J.~Rasmussen, R.~Rouquier, I.~Runkel, C.~Stroppel, and G.~Watts. 
\end{acknowledgements}

\section{Preliminaries}
\label{preliminaries}


\subsection{Graded matrix factorisations}

Let $R = \bigoplus_{i \ge 0} R_i$ be a graded ring. A $(\ZZ \times \Ztwo)$-graded $R$-module is a graded $R$-module $X$ together with a decomposition $X = X^0 \oplus X^1$ as graded submodules. An element of $(X^i)_j$ has \textsl{bidegree} $(i,j)$ and if $\varphi: X \lto Y$ has bidegree $(a,b)$ if $\varphi(X^i) \subseteq X^{i+a}$ and $\varphi(X_j) \subseteq X_{j+b}$ for all $i,j$.

Let $W \in R_{2c}$ be a homogeneous element of even degree that we call a \textsl{potential}. A \textsl{graded linear factorisation} of $W$ is a pair $X^0, X^1$ of graded $R$-modules together with a pair of $R$-linear maps $(d_X^0, d_X^1)$ of degree $c$, as in the diagram
\[
\xymatrix{
X^0 \ar[r]^-{d_X^0} & X^1 \ar[r]^-{d_X^1} & X^0\,,
}
\]
such that $d_X^0 \circ d_X^1 = W \cdot 1_{X^1}$ and $d_X^1 \circ d_X^0  = W \cdot 1_{X^0}$. It is equivalent to say that $X$ is a $(\ZZ \times \Ztwo)$-graded $R$-module together with an $R$-linear endomorphism $d_X$ of bidegree $(1,c)$, the \textsl{differential}, such that $d_X^2 = W \cdot 1_X$. A graded $R$-module is a graded linear factorisation of zero concentrated in $\nZ_{2}$-degree zero. 

A \textsl{morphism} $\varphi: X \lto Y$ of degree $e$ between graded linear factorisations is an $R$-linear map of bigdegree $(0, e)$ satisfying $d_Y \circ \varphi = \varphi \circ d_X$. 
A graded linear factorisation $X$ is a \textsl{matrix factorisation} if each $X^i$ is a free graded $R$-module, and a \textsl{finite-rank matrix factorisation} if each $X^i$ is a free graded module of finite rank. In this case, if we write each $X^i$ as a finite direct sum of grading shifted copies $R\{ b_i \}$ of $R$, the differential $d_X$ is represented by an odd matrix with homogeneous entries (given an integer $m$, $M\{ m \}$ denotes the $R$-module $M$ with the shifted grading $M\{ m \}_i = M_{i - m}$).

Given morphisms $\varphi, \psi: X \lto Y$ of graded linear factorisations of $W$, a \textsl{homotopy} from $\varphi$ to $\psi$ is an $R$-linear bidegree $(-1,-c)$ map $\lambda: X \lto Y$ such that $d_Y \circ \lambda + \lambda \circ d_X = \psi - \varphi$. We define the \textsl{homotopy category of graded linear factorisations} $\HF(R,W)$ to be the category of graded linear factorisations modulo the homotopy relation. We write $\HMF(R,W)$, respectively $\hmf(R,W)$, for the full subcategories of matrix factorisations and finite-rank matrix factorisations in $\HF(R,W)$. 

If $X$ is a graded linear factorisation of $W$ and $m \in \mathds{Z}$ then the grading shifted module $X\{ m \}$ is a graded linear factorisation with the differential of $X$ and decomposition $X\{ m \} = X^0\{ m \} \oplus X^1\{ m \}$, and so a degree $e$ morphism $X \lto Y$ is the same as a degree zero morphism $X \lto Y\{-e\}$. The \textsl{suspension} of $X$, denoted $X\langle 1 \rangle$, is the graded linear factorisation $(-d_X^1, d_X^0)$.

Let $X,Y$ be graded linear factorisations of $W,W' \in R_{2c}$, respectively. The graded $R$-module $\Hom_{\mathrm{gr}}(X,Y) = \bigoplus_{i \in \mathds{Z}} \Hom_R(X,Y)_i$ has an obvious $\Ztwo$-decomposition into even and odd maps, and equipped with the differential $d(\alpha) = d_Y \circ \alpha - (-1)^{|\alpha|} \alpha \circ d_X$ 
this is a graded linear factorisation of $W' - W$. In particular if $W = W'$ we have a $(\ZZ \times \Ztwo)$-graded complex. 


Given graded linear factorisations $X,Y$ of $W,W' \in R_{2c}$ respectively the tensor product $X \otimes Y$ (all tensor products in this section are $R$-linear) is a graded $R$-module with a natural $\Ztwo$-decomposition and differential $d_{X\otimes Y} = d_X \otimes 1 + 1 \otimes d_Y$ making it into a graded linear factorisation of $W + W'$. The notation here implicitly involves Koszul signs.


\subsection{Cyclic Koszul complexes}\label{prelim:cyclic_koszul}

Let us recall the definition of the cyclic Koszul complex from \cite[Section~2]{kr0401268}. Let $R = \bigoplus_{i \ge 0} R_i$ be a graded ring. If $a,b\in R$ are homogeneous with $\deg(a) + \deg(b) = 2c$ then $\{ a, b \}$ denotes the graded matrix factorisation of $ab$ given by
\[
\xymatrix{
R \ar[r]^-{a} & R\{ c - \deg(a) \} \ar[r]^-{b} & R \, .
}
\]
For two sequences $\bs{a} = (a_1,\ldots,a_n)$ and $\bs{b} = (b_1, \ldots, b_n)$ of homogeneous elements in $R$ satisfying $\deg(a_i) + \deg(b_i) = 2c$ for each $i$, we define
\[
\{ \bs{a}, \bs{b} \} := \{a_1,b_1\} \otimes \ldots \otimes \{a_n, b_n\}
\]
which is a graded matrix factorisation of $\sum_i a_i b_i \in R_{2c}$. A better way to present the differential on this tensor product is to introduce formal symbols $\theta_i$ of bidegree $(-1, \deg(a_i) - c)$, so that when we write $R \theta_i$ we mean $R\{ \deg(a_i) - c \}$ placed in cohomological degree $-1$. The graded free module $F = \bigoplus_i R\theta_i$ and its exterior algebra $\bigwedge F$ acquire both a $\ZZ$-grading from these cohomological degrees, and a grading in the usual sense. We define differentials $\delta_{\pm}$ on $\bigwedge F$ of bidegree $(\pm 1,c)$ by
\[
\delta_+ = \Big(\sum_i b_i \theta_i^*\Big) \neg\, (-) \, , \qquad \delta_{-} = \Big(\sum_i a_i \theta_i\Big) \wedge (-)\,.
\]

On the $\Ztwo$-folding of $\bigwedge F$ the map $\delta = \delta_+ + \delta_-$ has bidegree $(1,c)$ and squares to multiplication by $\sum_i a_i b_i$, so the pair $(\bigwedge F, \delta)$ is a graded matrix factorisation of this potential, canonically isomorphic to $\{ \bs{a}, \bs{b} \}$. We will identify these two factorisations.

\section{Compiling decorated webs}
\label{compilewebs}

Khovanov-Rozansky homology $\KR_{N}$ is a categorification of the polynomial link invariant $P_{N}$ which was constructed from the representation theory of the $\sln$ quantum group in~\cite{RT1990}. In this section we recall the definition of $\KR_{N}$ and explain our technique for computing it in terms of webs of matrix factorisations and their compilation. 

\medskip

The construction of $P_{N}(L)$ for a link~$L$ is a two-step process. The first step starts from the link diagram of~$L$ and replaces each crossing 
\begin{minipage}{0.48cm}
\begin{tikzpicture}[scale=0.23,baseline, inner sep=0.1mm, >=stealth]
\node (0) at (0,0) [circle,inner sep=1.5pt, draw=white,fill= white] {};
\node (bl) at (-1,-1) [circle,draw=white,fill= white] {};
\node (br) at (1,-1) [circle,draw=white,fill= white] {};
\node (tl) at (-1,1) [circle,draw=white,fill= white] {};
\node (tr) at (1,1) [circle,draw=white,fill= white] {};
\draw[-] (bl) -- (0); 
\draw[->] (br) -- (tl); 
\draw[->] (0) -- (tr); 
\end{tikzpicture}
\end{minipage}
 or 
\begin{minipage}{0.48cm}
\begin{tikzpicture}[scale=0.23,baseline, inner sep=0.1mm, >=stealth]
\node (0) at (0,0) [circle,inner sep=1.5pt, draw=white,fill= white] {};
\node (bl) at (-1,-1) [circle,draw=white,fill= white] {};
\node (br) at (1,-1) [circle,draw=white,fill= white] {};
\node (tl) at (-1,1) [circle,draw=white,fill= white] {};
\node (tr) at (1,1) [circle,draw=white,fill= white] {};
\draw[->] (bl) -- (tr); 
\draw[-] (br) -- (0); 
\draw[->] (0) -- (tl); 
\end{tikzpicture}
\end{minipage} 
 by either the \textsl{singular crossing}
 \begin{minipage}{0.48cm}
\begin{tikzpicture}[scale=0.23,baseline, inner sep=0.1mm, >=stealth]
\node (0) at (0,0) [circle,inner sep=0.7pt, thin,draw=black,fill= black] {};
\node (bl) at (-1,-1) [circle] {};
\node (br) at (1,-1) [circle] {};
\node (tl) at (-1,1) [circle] {};
\node (tr) at (1,1) [circle] {};
\draw[-] (bl) -- (0); 
\draw[-] (br) -- (0); 
\draw[->] (0) -- (tl); 
\draw[->] (0) -- (tr); 
\end{tikzpicture}
\end{minipage}
 or the \textsl{smoothing}
\begin{minipage}{0.36cm}
\begin{tikzpicture}[scale=0.23,baseline, inner sep=0.1mm, >=stealth]
\draw[ ->, shift={(-2.5,0)} ]  (-50:1) arc (-50:50:1); 
\draw[ -> ] (230:1) arc (230:130:1); 
\end{tikzpicture}
\end{minipage}, 
where both strands of a smoothing are actually divided by a bivalent vertex called a \textsl{mark} that we do not explicitly show in our graphs.\footnote{As mentioned in the introduction, instead of introducing labelled edges already at this point as in~\cite{moy1998}, or wide edges with three-valent vertices as in~\cite{kr0401268}, we prefer to formulate the construction in terms of the vertex  \begin{minipage}{0.38cm}
\begin{tikzpicture}[scale=0.18,baseline, inner sep=0.1mm, >=stealth]
\node (0) at (0,0) [circle,inner sep=0.7pt, thin,draw=black,fill= black] {};
\node (bl) at (-1,-1) [circle] {};
\node (br) at (1,-1) [circle] {};
\node (tl) at (-1,1) [circle] {};
\node (tr) at (1,1) [circle] {};
\draw[-] (bl) -- (0); 
\draw[-] (br) -- (0); 
\draw[->] (0) -- (tl); 
\draw[->] (0) -- (tr); 
\end{tikzpicture}
\end{minipage}.} 
From a link diagram with~$s$ crossings this produces $2^s$ oriented planar graphs~$\Gamma$ whose vertices are either bivalent with one incoming and one outgoing edge, or four-valent of type 
 \begin{minipage}{0.48cm}
\begin{tikzpicture}[scale=0.23,baseline, inner sep=0.1mm, >=stealth]
\node (0) at (0,0) [circle,inner sep=0.7pt, thin,draw=black,fill= black] {};
\node (bl) at (-1,-1) [circle] {};
\node (br) at (1,-1) [circle] {};
\node (tl) at (-1,1) [circle] {};
\node (tr) at (1,1) [circle] {};
\draw[-] (bl) -- (0); 
\draw[-] (br) -- (0); 
\draw[->] (0) -- (tl); 
\draw[->] (0) -- (tr); 
\end{tikzpicture}
\end{minipage}. If~$\Gamma$ has~$r$ edges and we label each edge by distinct integers $1,\ldots,r$, then we call it a \textsl{state graph}. To each state graph~$\Gamma$ the construction of~\cite{moy1998} associates a Laurent polynomial $p_{N}(\Gamma)\in\nZ[q^{\pm 1}]$ that is uniquely determined by its value $[N]=(q^N-q^{-N})(q-q^{-1})^{-1}$ for an oriented circle, and the \textsl{MOY relations} \cite[p.$4$]{kr0401268}.

In the second step one produces from the link diagram $D$ a linear combination of these $2^s$ graphs $\Gamma_{j}$ by resolving each crossing in the link diagram according to
\begin{equation}
\label{graphreduction}
\begin{tikzpicture}[scale=0.5,baseline, inner sep=1mm, >=stealth]
\node (0) at (0,0) [circle,inner sep=1.5pt, draw=white,fill= white] {};
\node (bl) at (-1,-1) [circle,draw=white,fill= white] {};
\node (br) at (1,-1) [circle,draw=white,fill= white] {};
\node (tl) at (-1,1) [circle,draw=white,fill= white] {};
\node (tr) at (1,1) [circle,draw=white,fill= white] {};
\draw[-, thick] (bl) -- (0); 
\draw[->, thick] (br) -- (tl); 
\draw[->, thick] (0) -- (tr); 
\end{tikzpicture}
\widehat = \;
q^{1-N} \,
\begin{tikzpicture}[scale=0.5,baseline, inner sep=1mm, >=stealth]
\draw[ ->, thick, shift={(-2.5,0)} ]  (-50:1) arc (-50:50:1); 
\draw[ ->, thick ] (230:1) arc (230:130:1); 
\end{tikzpicture}
- q^{-N} 
\begin{tikzpicture}[scale=0.5,baseline, inner sep=1mm, >=stealth]
\node (bl) at (-1,-1) [circle] {};
\node (br) at (1,-1) [circle] {};
\node (tl) at (-1,1) [circle] {};
\node (tr) at (1,1) [circle] {};
\node (0) at (0,0) [circle,inner sep=1.5pt, thin,draw=black,fill= black] {};
\draw[-,  thick] (bl) -- (0); 
\draw[-,  thick] (br) -- (0); 
\draw[->,  thick] (0) -- (tl); 
\draw[->,  thick] (0) -- (tr); 
\end{tikzpicture}
\, , \qquad 
\begin{tikzpicture}[scale=0.5,baseline, inner sep=1mm, >=stealth]
\node (0) at (0,0) [circle,inner sep=1.5pt, draw=white,fill= white] {};
\node (bl) at (-1,-1) [circle,draw=white,fill= white] {};
\node (br) at (1,-1) [circle,draw=white,fill= white] {};
\node (tl) at (-1,1) [circle,draw=white,fill= white] {};
\node (tr) at (1,1) [circle,draw=white,fill= white] {};
\draw[->, thick] (bl) -- (tr); 
\draw[-, thick] (br) -- (0); 
\draw[->, thick] (0) -- (tl); 
\end{tikzpicture}
\widehat = \;
q^{N-1} \,
\begin{tikzpicture}[scale=0.5,baseline, inner sep=1mm, >=stealth]
\draw[ ->, thick, shift={(-2.5,0)} ]  (-50:1) arc (-50:50:1); 
\draw[ ->, thick ] (230:1) arc (230:130:1); 
\end{tikzpicture}
- q^{N} 
\begin{tikzpicture}[scale=0.5,baseline, inner sep=1mm, >=stealth]
\node (bl) at (-1,-1) [circle] {};
\node (br) at (1,-1) [circle] {};
\node (tl) at (-1,1) [circle] {};
\node (tr) at (1,1) [circle] {};
\node (0) at (0,0) [circle,inner sep=1.5pt, thin,draw=black,fill= black] {};
\draw[-,  thick] (bl) -- (0); 
\draw[-,  thick] (br) -- (0); 
\draw[->,  thick] (0) -- (tl); 
\draw[->,  thick] (0) -- (tr); 
\end{tikzpicture} \, .
\end{equation}
Denoting by $q^{\alpha_j}$ the product of the~$s$ $q$-monomials associated to the graph $\Gamma_{j}$ by~\eqref{graphreduction},
$$
P_{N}(L) = \sum_{j=1}^{2^s} q^{\alpha_j} p_{N}(\Gamma_{j}) \, .
$$
Khovanov and Rozansky categorify both steps in the above construction. 

To present their categorification we consider the general notion of a web of matrix factorisations and its compilation. To compute $\KR_{N}$ we will need the special case that involves only the two matrix factorisations associated to 
\begin{minipage}{0.36cm}
\begin{tikzpicture}[scale=0.23,baseline, inner sep=0.1mm, >=stealth]
\draw[ ->, shift={(-2.5,0)} ]  (-50:1) arc (-50:50:1); 
\draw[ -> ] (230:1) arc (230:130:1); 
\end{tikzpicture}
\end{minipage}
 and 
 \begin{minipage}{0.48cm}
\begin{tikzpicture}[scale=0.23,baseline, inner sep=0.1mm, >=stealth]
\node (0) at (0,0) [circle,inner sep=0.7pt, thin,draw=black,fill= black] {};
\node (bl) at (-1,-1) [circle] {};
\node (br) at (1,-1) [circle] {};
\node (tl) at (-1,1) [circle] {};
\node (tr) at (1,1) [circle] {};
\draw[-] (bl) -- (0); 
\draw[-] (br) -- (0); 
\draw[->] (0) -- (tl); 
\draw[->] (0) -- (tr); 
\end{tikzpicture}
\end{minipage}. 
However, the general case is also of interest, e.\,g.~for the multiple fusion of topological defects in Landau-Ginzburg models. 

\begin{definition}\label{defwebs}
A \textsl{web} $\cat{F}$ is an oriented graph with a nonempty set of vertices $V$ and a nonempty set of oriented edges $E$. Edges are allowed to begin or end (but not both) on an auxiliary ``vertex'' which we call the \textsl{boundary}, but we stress that whenever we refer to vertices we never mean the boundary. An edge is \textsl{internal} if it does not begin or end on the boundary, and \textsl{external} otherwise. 

Let $k$ be a field of characteristic zero and $c \ge 0$. A \textsl{decoration} of $\cat{F}$ is an assignment as follows:
\begin{itemize}
\item For each edge $e$ a set of variables $\boldsymbol{x}_e$ and homogeneous potential $W_e \in k[\boldsymbol{x}_e]$ of degree $2c$, where we give the ring variables degree two, such that the potential has an isolated singularity, 
\[
\dim_k\left( k[\bs{x}_e]/(\partial_{\bs{x}_e} W) \right) < \infty \, .
\]
Here $(\partial_{\bs{x}_e} W)$ denotes the ideal generated by all the partial derivatives of $W_e$. For distinct edges $e,e'$ the sets $\boldsymbol{x}_e$ and $\boldsymbol{x}_{e'}$ are disjoint. We say that the variables in $\boldsymbol{x}_e$ \textsl{lie on} the edge $e$.
\item For each vertex $v$ a matrix factorisation $X_v$ of the difference between the incoming and outgoing potentials at $v$. More precisely, let $I_v$ denote the set of incoming edges at $v$ and $O_v$ the set of outgoing edges (either of which may be empty) and set
\[
W_v = \sum_{e \in O_v} W_e - \sum_{e \in I_v} W_e \, .
\]
Let $\boldsymbol{x}_v$ denote the set of variables lying on edges incident to $v$, so that $W_v \in k[\boldsymbol{x}_v]$. Then $X_v$ is a finite-rank graded matrix factorisation of $W_v$ over $k[\bs{x}_v]$, with a chosen homogeneous basis.
\end{itemize}
To any decoration we associate the total potential and total matrix factorisation, as follows. Let $\bs{x}$ be the set of all variables in the decoration. For each vertex $v$ there is the inclusion $\varphi_v: k[\bs{x}_v] \lto k[\bs{x}]$ and we define the \textsl{total matrix factorisation} by fixing once and for all an ordering $V = \{ v_1, \ldots, v_s \}$ on the vertices, and thereby defining
\[
X := \bigotimes_{v\in V} \varphi_v^*(X_v) = X_{v_1} \otimes \ldots \otimes X_{v_s}\,.
\]
Let $O$ denote the set of edges ending on the boundary, and $I$ the edges beginning on the boundary. Since the internal potentials cancel, $X$ is a finite-rank graded matrix factorisation over $k[\bs{x}]$ of the \textsl{total potential}
\[
W := \sum_{v\in V} W_v = \sum_{e \in O} W_e - \sum_{e \in I} W_e\,.
\]
The chosen homogeneous bases of the $X_v$ determine a homogeneous basis of $X$. Notice that if the original web is \textsl{closed}, meaning that $I$ and $O$ are both empty, then $W = 0$ and $X$ is a $(\ZZ \times \Ztwo)$-graded complex.
\end{definition}

\begin{example}
\label{MOYremark}
If we begin with a state graph $\Gamma$ of a link diagram and replace each smoothing
\begin{minipage}{0.30cm}
\begin{tikzpicture}[scale=0.23,baseline, inner sep=0.1mm, >=stealth]
\draw[ ->, shift={(-2.5,0)} ]  (-50:1) arc (-50:50:1); 
\draw[ -> ] (230:1) arc (230:130:1); 
\end{tikzpicture}
\end{minipage}
 by a \textsl{smooth vertex} 
\begin{minipage}{0.48cm}
\begin{tikzpicture}[scale=0.23,baseline, inner sep=0.1mm, >=stealth]
\node (bl) at (-1,-1) [circle] {};
\node (br) at (1,-1) [circle] {};
\node (tl) at (-1,1) [circle] {};
\node (tr) at (1,1) [circle] {};
\draw[-] (bl) -- (0); 
\draw[-] (br) -- (0); 
\draw[->] (0) -- (tl); 
\draw[->] (0) -- (tr); 
\node (0) at (0,0) [circle,inner sep=0.7pt, thin,draw=black,fill= white] {};
\end{tikzpicture}
\end{minipage}
then we obtain a closed oriented planar graph whose vertices are of types
\begin{minipage}{0.48cm}
\begin{tikzpicture}[scale=0.23,baseline, inner sep=0.1mm, >=stealth]
\node (bl) at (-1,-1) [circle] {};
\node (br) at (1,-1) [circle] {};
\node (tl) at (-1,1) [circle] {};
\node (tr) at (1,1) [circle] {};
\draw[-] (bl) -- (0); 
\draw[-] (br) -- (0); 
\draw[->] (0) -- (tl); 
\draw[->] (0) -- (tr); 
\node (0) at (0,0) [circle,inner sep=0.7pt, thin,draw=black,fill= white] {};
\end{tikzpicture}
\end{minipage} 
 and 
\begin{minipage}{0.48cm}
\begin{tikzpicture}[scale=0.23,baseline, inner sep=0.1mm, >=stealth]
\node (0) at (0,0) [circle,inner sep=0.7pt, thin,draw=black,fill= black] {};
\node (bl) at (-1,-1) [circle] {};
\node (br) at (1,-1) [circle] {};
\node (tl) at (-1,1) [circle] {};
\node (tr) at (1,1) [circle] {};
\draw[-] (bl) -- (0); 
\draw[-] (br) -- (0); 
\draw[->] (0) -- (tl); 
\draw[->] (0) -- (tr); 
\end{tikzpicture}
\end{minipage}. We define a decoration of this \textsl{state web} by fixing an integer $N$ and assigning to each edge $i$ a single variable $x_{i}$ and the potential $x_{i}^{N+1}$. Choose two finite-rank graded matrix factorisations
\[
\Xcirc, \Xbul \in \hmf( k[u,x,y,z], u^{N+1} +x^{N+1} - y^{N+1} - z^{N+1} )
\]
(whose dependence on the variables we do not usually display) and make the following assignments of matrix factorisations to vertices: 
\begin{equation}
\label{MFcrossings}
\begin{tikzpicture}[scale=0.7,baseline, inner sep=1mm, >=stealth]
\node (1) at (-1.1,1.0) [circle,inner sep=1.5pt, draw=white,fill= white] {$x^{N+1}_{i}$};
\node (2) at (1.5,1.0) [circle,inner sep=1.5pt, draw=white,fill= white] {$x^{N+1}_{j}$};
\node (3) at (1.5,-1.2) [circle,inner sep=1.5pt, draw=white,fill= white] {$x^{N+1}_{k}$};
\node (3) at (-1.1,-1.2) [circle,inner sep=1.5pt, draw=white,fill= white] {$x^{N+1}_{l}$};
\node (0) at (0,0) [circle,inner sep=1.5pt, thin,draw=black,fill= white] {};
\node (bl) at (-1,-1) [circle] {};
\node (br) at (1,-1) [circle] {};
\node (tl) at (-1,1) [circle] {};
\node (tr) at (1,1) [circle] {};
\draw[-,  thick] (bl) -- (0); 
\draw[-,  thick] (br) -- (0); 
\draw[->,  thick] (0) -- (tl); 
\draw[->,  thick] (0) -- (tr); 
\end{tikzpicture}
\widehat = \,
\Xcirc(x_{i},x_{j},x_{k},x_{l}) \, ,
\qquad 
\begin{tikzpicture}[scale=0.7,baseline, inner sep=1mm, >=stealth]
\node (1) at (-1.1,1.0) [circle,inner sep=1.5pt, draw=white,fill= white] {$x^{N+1}_{i}$};
\node (2) at (1.5,1.0) [circle,inner sep=1.5pt, draw=white,fill= white] {$x^{N+1}_{j}$};
\node (3) at (1.5,-1.2) [circle,inner sep=1.5pt, draw=white,fill= white] {$x^{N+1}_{k}$};
\node (3) at (-1.1,-1.2) [circle,inner sep=1.5pt, draw=white,fill= white] {$x^{N+1}_{l}$};
\node (0) at (0,0) [circle,inner sep=1.5pt, thin,draw=black,fill= black] {};
\node (bl) at (-1,-1) [circle] {};
\node (br) at (1,-1) [circle] {};
\node (tl) at (-1,1) [circle] {};
\node (tr) at (1,1) [circle] {};
\draw[-,  thick] (bl) -- (0); 
\draw[-,  thick] (br) -- (0); 
\draw[->,  thick] (0) -- (tl); 
\draw[->,  thick] (0) -- (tr); 
\end{tikzpicture}
\widehat = \,
\Xbul (x_{i},x_{j},x_{k},x_{l}) \, .
\end{equation}
Since the web is closed, the total matrix factorisation will be a finite-rank $(\ZZ \times \Ztwo)$-graded complex $C(\Gamma)$ over the polynomial ring of all edge variables. As we will see below, the cohomology $H(\Gamma)$ of this complex is finite-dimensional for general reasons.
\end{example}

\begin{example}\label{example:defectonbrane} 
Consider the web
$$
\begin{tikzpicture}[scale=1.0,baseline=-0.1cm, inner sep=1mm, >=stealth]
\node (b) at (-2,0)  {};
\node (D) at (0,0) [circle,draw=black,fill= white, thick] {$u$};
\node (Q) at (2,0) [circle,draw=black,fill= white, thick] {$v$};
\draw[->,  thick] (Q) -- (D) node[midway,above] {$e'$}; 
\draw[->,  thick] (D) -- (b) node[midway,above] {$e$}; 
\end{tikzpicture} 
$$
with vertices $u,v$ and edges $e,e'$ which are external and internal, respectively, and we imagine the boundary to the left. A decoration would consist of $e$-variables $\bs{x}$, $e'$-variables $\bs{x'}$ and homogeneous potentials $W \in k[\bs{x}], W' \in k[\bs{x'}]$ with isolated singularities together with a matrix factorisation $X$ of $W - W'$ and $Q$ of $W'$. 

The total matrix factorisation $X \otimes Q$ is an object of $\hmf(k[\bs{x},\bs{x'}], W)$ which can be viewed as an infinite-rank matrix factorisation over $k[\bs{x}]$ of $W$. In this way, if we were to fix $X$ and let $Q$ vary, we can view $X$ as a functor $\HMF(k[\bs{x'}], W') \lto \HMF(k[\bs{x}], W)$.
\end{example}

The relevant problem for Khovanov-Rozansky homology is: given a decorated web $\cat{F}$ how do we \textsl{compute} a finite-rank representative for the total factorisation? To avoid triviality we assume that there is at least one internal edge.

Let $\intvar$ denote the set of all variables lying on internal edges (the \textsl{internal} variables) and $\extvar$ the set of all variables lying on external edges (the \textsl{external} variables). Clearly $W$ only depends on the external variables, so the canonical inclusion $\varphi: k[\extvar] \lto k[\bs{x}]$ induces a functor
\[
\varphi_*: \hmf(k[\bs{x}], W) \lto \HMF(k[\extvar], W)
\]
defined by taking a graded matrix factorisation of $W$ over $k[\bs{x}]$ and viewing it as a factorisation of $W$ over the smaller ring. We refer to this functor as the \textsl{pushforward}. There are two things to note: firstly, that a finite-rank free module over $k[\bs{x}]$ will necessarily be of infinite-rank over $k[\extvar]$ so the pushforward produces infinite-rank matrix factorisations, and secondly that the case where $\extvar$ is empty and $W = 0$, so $k[\extvar] = k$, is of primary interest for our later discussion of link homology.

It is well-known that because all involved singularities are isolated the pushforward of the total factorisation $\varphi_*(X)$, while infinite-rank, is still homotopy equivalent to a finite-rank factorisation. Finding this equivalent finite-rank factorisation is what we mean by computing the pushforward. The reader might like to consider the total factorisation $X = C(\Gamma)$ from Example \ref{MOYremark}, where $k[\extvar] = k$ and an example of a finite-rank matrix factorisation equivalent to $\varphi_*(X)$ is the cohomology $H(\Gamma)$ viewed as a $(\ZZ \times \Ztwo)$-graded complex with zero differential. Since Khovanov-Rozansky homology will be built out of the $H(\Gamma)$'s, such computations are the central motivation of this paper.

In practice this computation is hard to do, so we introduce an intermediate notion:

\begin{definition}
\label{def:compilation}
Consider a pair $(C, \idem)$ consisting of a finite-rank graded matrix factorisation $C$ of $W$ over $k[\extvar]$ together with an idempotent endomorphism $\idem$ of $C$ (that is, $\idem^2$ is homotopic to $\idem$). We say that the decorated web $\mathcal F$ above \textsl{compiles} to $(C,\idem)$ if there exist morphisms
\[
\psi: C \lto \varphi_*(X) \, , \qquad \vartheta: \varphi_*(X) \lto C
\]
such that $\psi \circ \vartheta = 1_{\varphi_*(X)}$ and $\vartheta \circ \psi = \idem$ in $\operatorname{HMF}(k[\extvar],W)$.
\end{definition}

It is a short step from a compilation to a finite-rank representative for the pushforward:

\begin{remark} 
\label{compileAndSplit}
Suppose that $(C, \idem)$ compiles $\cat{F}$. The category $\hmf(k[\extvar],W)$ has split idempotents, and moreover idempotents in this category can be split algorithmically. So it is possible in practice to find a finite-rank graded matrix factorisation $F$ and morphisms $f,g$ as in the diagram
\[
\xymatrix@C+2pc{
F \ar@<-1ex>[r]_-{g} & C \ar@<-1ex>[l]_-{f}\ar@<-1ex>[r]_-{\psi} & \varphi_*(X) \ar@<-1ex>[l]_-{\vartheta}
}
\]
such that $f \circ g = 1_F$ and $g \circ f = \idem$ in $\hmf(k[\extvar], W)$. Then $f \circ \vartheta$ is a homotopy equivalence and $F$ is a finite-rank model of $\varphi_*(X)$.
\end{remark}

By reformulating a result of~\cite{dm1102.2957} we can explicitly construct a pair $(C, \idem)$ compiling $\cat{F}$. We will use the phrase ``compiling a web'' synonymously for ``computing the pushforward of the total matrix factorisation of a web''. This explicit construction is the key to our calculations, but in order to preserve the flow of the narrative we delay the details to Section \ref{section:explicitidempotents} and return now to the construction of Khovanov-Rozansky homology.

\medskip

In \cite{kr0401268} there are defined two basic matrix factorisations $\Xcirc$ and $\Xbul$ (our notation is different, but the explicit definitions are given in~\eqref{XbulXcirc} below). Let~$D$ be a link diagram for a link~$L$, and let $\mathcal C$ be the set of crossings in~$D$. If we replace each crossing by either 
\begin{minipage}{0.48cm}
\begin{tikzpicture}[scale=0.23,baseline, inner sep=0.1mm, >=stealth]
\node (bl) at (-1,-1) [circle] {};
\node (br) at (1,-1) [circle] {};
\node (tl) at (-1,1) [circle] {};
\node (tr) at (1,1) [circle] {};
\node (0) at (0,0) [circle,inner sep=0.7pt, thin,draw=black,fill= white] {};
\draw[-] (bl) -- (0); 
\draw[-] (br) -- (0); 
\draw[->] (0) -- (tl); 
\draw[->] (0) -- (tr); 
\node (0) at (0,0) [circle,inner sep=0.7pt, thin,draw=black,fill= white] {};
\end{tikzpicture}
\end{minipage} 
 or 
\begin{minipage}{0.48cm}
\begin{tikzpicture}[scale=0.23,baseline, inner sep=0.1mm, >=stealth]
\node (0) at (0,0) [circle,inner sep=0.7pt, thin,draw=black,fill= black] {};
\node (bl) at (-1,-1) [circle] {};
\node (br) at (1,-1) [circle] {};
\node (tl) at (-1,1) [circle] {};
\node (tr) at (1,1) [circle] {};
\draw[-] (bl) -- (0); 
\draw[-] (br) -- (0); 
\draw[->] (0) -- (tl); 
\draw[->] (0) -- (tr); 
\end{tikzpicture}
\end{minipage} 
we obtain a resolution~$\Gamma$ of~$D$, and if we associate matrix factorisations to both types of four-valent vertices by assigning~$\Xcirc$ to
\begin{minipage}{0.48cm}
\begin{tikzpicture}[scale=0.23,baseline, inner sep=0.1mm, >=stealth]
\node (bl) at (-1,-1) [circle] {};
\node (br) at (1,-1) [circle] {};
\node (tl) at (-1,1) [circle] {};
\node (tr) at (1,1) [circle] {};
\node (0) at (0,0) [circle,inner sep=0.7pt, thin,draw=black,fill= white] {};
\draw[-] (bl) -- (0); 
\draw[-] (br) -- (0); 
\draw[->] (0) -- (tl); 
\draw[->] (0) -- (tr); 
\node (0) at (0,0) [circle,inner sep=0.7pt, thin,draw=black,fill= white] {};
\end{tikzpicture}
\end{minipage} 
 and~$\Xbul\{-1\}$ to 
\begin{minipage}{0.48cm}
\begin{tikzpicture}[scale=0.23,baseline, inner sep=0.1mm, >=stealth]
\node (0) at (0,0) [circle,inner sep=0.7pt, thin,draw=black,fill= black] {};
\node (bl) at (-1,-1) [circle] {};
\node (br) at (1,-1) [circle] {};
\node (tl) at (-1,1) [circle] {};
\node (tr) at (1,1) [circle] {};
\draw[-] (bl) -- (0); 
\draw[-] (br) -- (0); 
\draw[->] (0) -- (tl); 
\draw[->] (0) -- (tr); 
\end{tikzpicture}
\end{minipage} 
then~$\Gamma$ becomes a state web as explained in Example~\ref{MOYremark}. Promoting every resolution $\Gamma$ to a state web using these local assignments, and writing $C(\Gamma)$ for the total matrix factorisation, it is shown in \cite{kr0401268} that there is a series of homotopy equivalences of matrix factorisations which together give a categorified version of the MOY relations: 
\begin{align}
\label{MOYcatDecomp1}&
C \bigg( 
\begin{tikzpicture}[scale=0.5,baseline=-2pt, inner sep=1mm, >=stealth]
\node (x2) at (-1,1) [circle,inner sep=1.5pt, draw=white,fill= white] {{\tiny $x_{2}$}};
\node (x3) at (-1,-1) [circle,inner sep=1.5pt, draw=white,fill= white] {{\tiny $x_{3}$}};
\node (centre) at (0,0.75) [circle] {};
\node (bl) at (-1,-1) [circle] {};
\node (br) at (1,-1) [circle] {};
\node (tl) at (-1,1) [circle] {};
\node (r) at (1,0) [circle] {};
\node (0) at (0,0) [circle,inner sep=1.5pt, thin,draw=black,fill= black] {};
\draw[->,  thick] (bl) -- (0); 
\draw[->,  thick] (0) -- (tl); 
\draw[ ->, thick, shift={(0.6,0)} ]  (175:0.6) arc (175:0:0.6); 
\draw[ -, thick, shift={(-0.6,0)} ]  (0:0.6) arc (-175:0:0.6); 
\end{tikzpicture} 
\bigg)
\langle 1\rangle
\cong
\bigoplus_{i=1}^{N-2} \,
C \bigg( 
\begin{tikzpicture}[scale=0.5,baseline=-2pt, inner sep=1mm, >=stealth]
\node (x2) at (0.5,1) [circle,inner sep=1.5pt, draw=white,fill= white] {{\tiny $x_{2}$}};
\node (x3) at (0.5,-1) [circle,inner sep=1.5pt, draw=white,fill= white] {{\tiny $x_{3}$}};
\draw[ ->, thick ]  (-50:1) arc (-50:50:1); 
\end{tikzpicture}
\bigg)
\{ 2-N-2i\} \, ,
\\
\label{MOYcatDecomp2}&
C \Bigg( 
\begin{tikzpicture}[scale=0.5,baseline=(centre.base), inner sep=1mm, >=stealth]
\node (x1) at (-1,2.5) [circle,inner sep=1.5pt, draw=white,fill= white] {{\tiny $x_{1}$}};
\node (x2) at (1,2.5) [circle,inner sep=1.5pt, draw=white,fill= white] {{\tiny $x_{2}$}};
\node (x3) at (1,-1) [circle,inner sep=1.5pt, draw=white,fill= white] {{\tiny $x_{3}$}};
\node (x4) at (-1,-1) [circle,inner sep=1.5pt, draw=white,fill= white] {{\tiny $x_{4}$}};
\node (centre) at (0,0.75) [circle] {};
\node (bl) at (-1,-1) [circle] {};
\node (br) at (1,-1) [circle] {};
\node (tl) at (-1,2.5) [circle] {};
\node (tr) at (1,2.5) [circle] {};
\node (0) at (0,0) [circle,inner sep=1.5pt, thin,draw=black,fill= black] {};
\node (02) at (0,1.5) [circle,inner sep=1.5pt, thin,draw=black,fill= black] {};
\draw[->,  thick] (bl) -- (0); 
\draw[->,  thick] (br) -- (0); 
\draw[->,  thick] (0) to [out=135,in=225] (02); 
\draw[->,  thick] (0) to [out=45,in=-45] (02); 
\draw[->,  thick] (02) -- (tl); 
\draw[->,  thick] (02) -- (tr); 
\end{tikzpicture} 
\Bigg)
\cong
C \bigg( 
\begin{tikzpicture}[scale=0.5,baseline=-2pt, inner sep=1mm, >=stealth]
\node (x1) at (-1,1) [circle,inner sep=1.5pt, draw=white,fill= white] {{\tiny $x_{1}$}};
\node (x2) at (1,1) [circle,inner sep=1.5pt, draw=white,fill= white] {{\tiny $x_{2}$}};
\node (x3) at (1,-1) [circle,inner sep=1.5pt, draw=white,fill= white] {{\tiny $x_{3}$}};
\node (x4) at (-1,-1) [circle,inner sep=1.5pt, draw=white,fill= white] {{\tiny $x_{4}$}};
\node (bl) at (-1,-1) [circle] {};
\node (br) at (1,-1) [circle] {};
\node (tl) at (-1,1) [circle] {};
\node (tr) at (1,1) [circle] {};
\node (0) at (0,0) [circle,inner sep=1.5pt, thin,draw=black,fill= black] {};
\draw[-,  thick] (bl) -- (0); 
\draw[-,  thick] (br) -- (0); 
\draw[->,  thick] (0) -- (tl); 
\draw[->,  thick] (0) -- (tr); 
\end{tikzpicture} 
\bigg)
\{ 1\}
\oplus
C \bigg( 
\begin{tikzpicture}[scale=0.5,baseline=-2pt, inner sep=1mm, >=stealth]
\node (x1) at (-1,1) [circle,inner sep=1.5pt, draw=white,fill= white] {{\tiny $x_{1}$}};
\node (x2) at (1,1) [circle,inner sep=1.5pt, draw=white,fill= white] {{\tiny $x_{2}$}};
\node (x3) at (1,-1) [circle,inner sep=1.5pt, draw=white,fill= white] {{\tiny $x_{3}$}};
\node (x4) at (-1,-1) [circle,inner sep=1.5pt, draw=white,fill= white] {{\tiny $x_{4}$}};
\node (bl) at (-1,-1) [circle] {};
\node (br) at (1,-1) [circle] {};
\node (tl) at (-1,1) [circle] {};
\node (tr) at (1,1) [circle] {};
\node (0) at (0,0) [circle,inner sep=1.5pt, thin,draw=black,fill= black] {};
\draw[-,  thick] (bl) -- (0); 
\draw[-,  thick] (br) -- (0); 
\draw[->,  thick] (0) -- (tl); 
\draw[->,  thick] (0) -- (tr); 
\end{tikzpicture} 
\bigg)
\{ -1\}\, ,
\\
\label{MOYcatDecomp3}&
C \bigg( 
\begin{tikzpicture}[scale=0.5,baseline=-2pt, inner sep=1mm, >=stealth]
\node (x1) at (-1,1) [circle,inner sep=1.5pt, draw=white,fill= white] {{\tiny $x_{1}$}};
\node (x2) at (2.5,1) [circle,inner sep=1.5pt, draw=white,fill= white] {{\tiny $x_{2}$}};
\node (x3) at (2.5,-1) [circle,inner sep=1.5pt, draw=white,fill= white] {{\tiny $x_{3}$}};
\node (x4) at (-1,-1) [circle,inner sep=1.5pt, draw=white,fill= white] {{\tiny $x_{4}$}};
\node (centre) at (0,0.75) [circle] {};
\node (bl) at (-1,-1) [circle] {};
\node (br) at (2.5,-1) [circle] {};
\node (tl) at (-1,1) [circle] {};
\node (tr) at (2.5,1) [circle] {};
\node (0) at (0,0) [circle,inner sep=1.5pt, thin,draw=black,fill= black] {};
\node (02) at (1.5,0) [circle,inner sep=1.5pt, thin,draw=black,fill= black] {};
\draw[->,  thick] (bl) -- (0); 
\draw[->,  thick] (0) -- (tl); 
\draw[->,  thick] (02) -- (br); 
\draw[->,  thick] (tr) -- (02); 
\draw[->,  thick] (0) to [out=45,in=135] (02); 
\draw[->,  thick] (02) to [out=-135,in=-45] (0); 
\end{tikzpicture} 
\bigg)
\cong
C \bigg( 
\begin{tikzpicture}[scale=0.5,baseline=-2pt, inner sep=1mm, >=stealth]
\node (x1) at (-1,1) [circle,inner sep=1.5pt, draw=white,fill= white] {{\tiny $x_{1}$}};
\node (x2) at (1,1) [circle,inner sep=1.5pt, draw=white,fill= white] {{\tiny $x_{2}$}};
\node (x3) at (1,-1) [circle,inner sep=1.5pt, draw=white,fill= white] {{\tiny $x_{3}$}};
\node (x4) at (-1,-1) [circle,inner sep=1.5pt, draw=white,fill= white] {{\tiny $x_{4}$}};
\draw[ ->, thick, shift={(-0,-1.3)} ] (140:1) arc (140:40:1); 
\draw[ ->, thick, shift={(-0,1.3)} ] (-40:1) arc (-40:-140:1); 
\end{tikzpicture}
\bigg)
\oplus
\Bigg[
\bigoplus_{i=0}^{N-3}
C \bigg(  
\begin{tikzpicture}[scale=0.5,baseline=-2pt, inner sep=1mm, >=stealth]
\node (x1) at (-2,1) [circle,inner sep=1.5pt, draw=white,fill= white] {{\tiny $x_{1}$}};
\node (x2) at (-0.4,1) [circle,inner sep=1.5pt, draw=white,fill= white] {{\tiny $x_{2}$}};
\node (x3) at (-0.4,-1) [circle,inner sep=1.5pt, draw=white,fill= white] {{\tiny $x_{3}$}};
\node (x4) at (-2,-1) [circle,inner sep=1.5pt, draw=white,fill= white] {{\tiny $x_{4}$}};
\draw[ ->, thick, shift={(-2.5,0)} ]  (-50:1) arc (-50:50:1); 
\draw[ ->, thick ] (130:1) arc (130:230:1); 
\end{tikzpicture}
\bigg)
\langle 1 \rangle
\{ 3-N+2i\}
\Bigg]\, ,
\\
\label{MOYcatDecomp4}& 
C \Bigg( \!
\begin{tikzpicture}[scale=0.5,baseline=(02.base), inner sep=1mm, >=stealth]
\node (centre) at (0,0.75) [circle] {};
\node (x1) at (-1,3) [circle,inner sep=1.5pt, draw=white,fill= white] {{\tiny $x_{1}$}};
\node (x2) at (0.7,3) [circle,inner sep=1.5pt, draw=white,fill= white] {{\tiny $x_{2}$}};
\node (x3) at (1.4,3) [circle,inner sep=1.5pt, draw=white,fill= white] {{\tiny $x_{3}$}};
\node (x4) at (-1,-1) [circle,inner sep=1.5pt, draw=white,fill= white] {{\tiny $x_{4}$}};
\node (x5) at (0.7,-1) [circle,inner sep=1.5pt, draw=white,fill= white] {{\tiny $x_{5}$}};
\node (x6) at (1.4,-1) [circle,inner sep=1.5pt, draw=white,fill= white] {{\tiny $x_{6}$}};
\node (bl) at (-1,-1) [circle] {};
\node (bm) at (1,-1) [circle] {};
\node (br) at (1,-1) [circle] {};
\node (tl) at (-1,3) [circle] {};
\node (tm) at (1,3) [circle] {};
\node (tr) at (1,3) [circle] {};
\node (0) at (0,0) [circle,inner sep=1.5pt, thin,draw=black,fill= black] {};
\node (02) at (1,1) [circle,inner sep=1.5pt, thin,draw=black,fill= black] {};
\node (03) at (0,2) [circle,inner sep=1.5pt, thin,draw=black,fill= black] {};
\draw[->,  thick] (bl) -- (0); 
\draw[->,  thick] (bm) -- (0); 
\draw[->,  thick] (br) to [out=45,in=-45] (02); 
\draw[->,  thick] (0) to [out=45,in=-135] (02); 
\draw[->,  thick] (02) to [out=135,in=-45] (03); 
\draw[->,  thick] (0) to [out=135,in=-135] (03); 
\draw[->,  thick] (02) to [out=45,in=-45] (tr); 
\draw[->,  thick] (03) -- (tl); 
\draw[->,  thick] (03) -- (tm); 
\end{tikzpicture} 
\;\Bigg)
\oplus C \bigg( 
\begin{tikzpicture}[scale=0.5,baseline=(02.base), inner sep=1mm, >=stealth]
\node (x1) at (0,3) [circle,inner sep=1.5pt, draw=white,fill= white] {{\tiny $x_{1}$}};
\node (x2) at (1,3) [circle,inner sep=1.5pt, draw=white,fill= white] {{\tiny $x_{2}$}};
\node (x3) at (2,3) [circle,inner sep=1.5pt, draw=white,fill= white] {{\tiny $x_{3}$}};
\node (x4) at (0,0) [circle,inner sep=1.5pt, draw=white,fill= white] {{\tiny $x_{4}$}};
\node (x5) at (1,0) [circle,inner sep=1.5pt, draw=white,fill= white] {{\tiny $x_{5}$}};
\node (x6) at (2,0) [circle,inner sep=1.5pt, draw=white,fill= white] {{\tiny $x_{6}$}};
\node (bl) at (0,0) [circle] {};
\node (bm) at (1,0) [circle] {};
\node (br) at (2,-0) [circle] {};
\node (tl) at (0,3) [circle] {};
\node (tm) at (1,3) [circle] {};
\node (tr) at (2,3) [circle] {};
\node (02) at (1.5,1.5) [circle,inner sep=1.5pt, thin,draw=black,fill= black] {};
\draw[ ->, thick ] (bl) -- (tl); 
\draw[->,  thick] (02) to [out=135,in=-90] (tm); 
\draw[->,  thick] (02) to [out=45,in=-90] (tr); 
\draw[->,  thick] (bm) to [out=90,in=-135] (02); 
\draw[->,  thick] (br) to [out=90,in=-45] (02); 
\end{tikzpicture}
\bigg)
\cong
C \Bigg( \;
\begin{tikzpicture}[scale=0.5,baseline=(02.base), inner sep=1mm, >=stealth]
\node (x1) at (-0.3,3) [circle,inner sep=1.5pt, draw=white,fill= white] {{\tiny $x_{1}$}};
\node (x2) at (0.4,3) [circle,inner sep=1.5pt, draw=white,fill= white] {{\tiny $x_{2}$}};
\node (x3) at (2,3) [circle,inner sep=1.5pt, draw=white,fill= white] {{\tiny $x_{3}$}};
\node (x4) at (-0.3,-1) [circle,inner sep=1.5pt, draw=white,fill= white] {{\tiny $x_{4}$}};
\node (x5) at (0.4,-1) [circle,inner sep=1.5pt, draw=white,fill= white] {{\tiny $x_{5}$}};
\node (x6) at (2,-1) [circle,inner sep=1.5pt, draw=white,fill= white] {{\tiny $x_{6}$}};
\node (centre) at (0,0.75) [circle] {};
\node (bl) at (0,-1) [circle] {};
\node (bm) at (0,-1) [circle] {};
\node (br) at (2,-1) [circle] {};
\node (tl) at (0,3) [circle] {};
\node (tm) at (0,3) [circle] {};
\node (tr) at (2,3) [circle] {};
\node (0) at (1,0) [circle,inner sep=1.5pt, thin,draw=black,fill= black] {};
\node (02) at (0,1) [circle,inner sep=1.5pt, thin,draw=black,fill= black] {};
\node (03) at (1,2) [circle,inner sep=1.5pt, thin,draw=black,fill= black] {};
\draw[->,  thick] (bl) to [out=135,in=-135] (02); 
\draw[->,  thick] (bm) -- (0); 
\draw[->,  thick] (br) -- (0); 
\draw[->,  thick] (0) to [out=135,in=-45] (02); 
\draw[->,  thick] (0) to [out=45,in=-45] (03); 
\draw[->,  thick] (02) to [out=135,in=-135] (tl); 
\draw[->,  thick] (02) -- (03); 
\draw[->,  thick] (03) -- (tr); 
\draw[->,  thick] (03) -- (tm); 
\end{tikzpicture} 
\! \Bigg)
\oplus C \bigg( 
\begin{tikzpicture}[scale=0.5,baseline=(02.base), inner sep=1mm, >=stealth]
\node (x1) at (1,3) [circle,inner sep=1.5pt, draw=white,fill= white] {{\tiny $x_{1}$}};
\node (x2) at (2,3) [circle,inner sep=1.5pt, draw=white,fill= white] {{\tiny $x_{2}$}};
\node (x3) at (3,3) [circle,inner sep=1.5pt, draw=white,fill= white] {{\tiny $x_{3}$}};
\node (x4) at (1,0) [circle,inner sep=1.5pt, draw=white,fill= white] {{\tiny $x_{4}$}};
\node (x5) at (2,0) [circle,inner sep=1.5pt, draw=white,fill= white] {{\tiny $x_{5}$}};
\node (x6) at (3,0) [circle,inner sep=1.5pt, draw=white,fill= white] {{\tiny $x_{6}$}};
\node (bl) at (3,0) [circle] {};
\node (bm) at (1,0) [circle] {};
\node (br) at (2,-0) [circle] {};
\node (tl) at (3,3) [circle] {};
\node (tm) at (1,3) [circle] {};
\node (tr) at (2,3) [circle] {};
\node (02) at (1.5,1.5) [circle,inner sep=1.5pt, thin,draw=black,fill= black] {};
\draw[ ->, thick ] (bl) -- (tl); 
\draw[->,  thick] (02) to [out=135,in=-90] (tm); 
\draw[->,  thick] (02) to [out=45,in=-90] (tr); 
\draw[->,  thick] (bm) to [out=90,in=-135] (02); 
\draw[->,  thick] (br) to [out=90,in=-45] (02); 
\end{tikzpicture}
\bigg) \, .
\end{align}
The \textsl{parity} of a resolution~$\Gamma$ is the number of components of the graph obtained from~$\Gamma$ by replacing each vertex 
\begin{minipage}{0.48cm}
\begin{tikzpicture}[scale=0.23,baseline, inner sep=0.1mm, >=stealth]
\node (bl) at (-1,-1) [circle] {};
\node (br) at (1,-1) [circle] {};
\node (tl) at (-1,1) [circle] {};
\node (tr) at (1,1) [circle] {};
\node (0) at (0,0) [circle,inner sep=0.7pt, thin,draw=black,fill= white] {};
\draw[-] (bl) -- (0); 
\draw[-] (br) -- (0); 
\draw[->] (0) -- (tl); 
\draw[->] (0) -- (tr); 
\node (0) at (0,0) [circle,inner sep=0.7pt, thin,draw=black,fill= white] {};
\end{tikzpicture}
\end{minipage} 
 or 
\begin{minipage}{0.48cm}
\begin{tikzpicture}[scale=0.23,baseline, inner sep=0.1mm, >=stealth]
\node (0) at (0,0) [circle,inner sep=0.7pt, thin,draw=black,fill= black] {};
\node (bl) at (-1,-1) [circle] {};
\node (br) at (1,-1) [circle] {};
\node (tl) at (-1,1) [circle] {};
\node (tr) at (1,1) [circle] {};
\draw[-] (bl) -- (0); 
\draw[-] (br) -- (0); 
\draw[->] (0) -- (tl); 
\draw[->] (0) -- (tr); 
\end{tikzpicture}
\end{minipage} 
 by the smoothing 
\begin{minipage}{0.30cm}
\begin{tikzpicture}[scale=0.23,baseline, inner sep=0.1mm, >=stealth]
\draw[ ->, shift={(-2.5,0)} ]  (-50:1) arc (-50:50:1); 
\draw[ -> ] (230:1) arc (230:130:1); 
\end{tikzpicture}
\end{minipage}  
modulo 2. The parity of a link is the parity of any of its resolutions~$\Gamma$.

Next Khovanov and Rozansky define morphisms $\chi_0: \Xcirc \lto \Xbul$ and $\chi_1: \Xbul \lto \Xcirc$ (the explicit matrices are recalled in Appendix \ref{section:stabilisation} below). These morphisms define two-term complexes of matrix factorisations 
\begin{align*}
C[\begin{minipage}{0.48cm}
\begin{tikzpicture}[scale=0.23,baseline, inner sep=0.1mm, >=stealth]
\node (0) at (0,0) [circle,inner sep=1.5pt, draw=white,fill= white] {};
\node (bl) at (-1,-1) [circle,draw=white,fill= white] {};
\node (br) at (1,-1) [circle,draw=white,fill= white] {};
\node (tl) at (-1,1) [circle,draw=white,fill= white] {};
\node (tr) at (1,1) [circle,draw=white,fill= white] {};
\draw[-] (bl) -- (0); 
\draw[->] (br) -- (tl); 
\draw[->] (0) -- (tr); 
\end{tikzpicture}
\end{minipage}] & =\!\!
\xymatrix{%
\Big( 0 \vphantom{\underline{X}} \ar[r] & \underline{\Xcirc}\{1-N\} \ar[r]^-{\chi_{0}} & \Xbul\{-N-1\} \vphantom{\underline{X}} \ar[r] & 0 \vphantom{\underline{X}} \Big) 
}
,  \\
C[\begin{minipage}{0.48cm}
\begin{tikzpicture}[scale=0.23,baseline, inner sep=0.1mm, >=stealth]
\node (0) at (0,0) [circle,inner sep=1.5pt, draw=white,fill= white] {};
\node (bl) at (-1,-1) [circle,draw=white,fill= white] {};
\node (br) at (1,-1) [circle,draw=white,fill= white] {};
\node (tl) at (-1,1) [circle,draw=white,fill= white] {};
\node (tr) at (1,1) [circle,draw=white,fill= white] {};
\draw[->] (bl) -- (tr); 
\draw[-] (br) -- (0); 
\draw[->] (0) -- (tl); 
\end{tikzpicture}
\end{minipage}] & =\!\!
\xymatrix{%
\Big(0 \vphantom{\underline{X}} \ar[r] & \Xbul \vphantom{\underline{X}}\{N-1\} \ar[r]^-{\chi_{1}} & \underline{\Xcirc} \{N-1\}\ar[r] & 0 \vphantom{\underline{X}} \Big) 
}
\end{align*}
where the underlines denote cohomological degree $0$ and grading shifts are inserted to match (\ref{graphreduction}). The \textsl{Khovanov-Rozansky complex} is
$$
C(D) = \bigotimes_{c\in\mathcal C} C[c] \, .
$$
Here the tensor product is over the polynomial ring~$R$ of all internal and external variables. By construction $C(D)$ is a $(\nZ\times\nZ\times\nZ_{2})$-graded complex of matrix factorisations of zero, and we denote its components of cohomological degree~$i$ and internal $\nZ$-degree~$j$ by $C^{i,j}(D)$. It is shown in~\cite{kr0401268} that the cohomology of $C(D)$ is concentrated in only one $\nZ_{2}$-degree which is given by the parity of~$L$ and is independent of the choice of link diagram~$D$ of~$L$. This is the \textsl{Khovanov-Rozansky homology $H(L)$}. Its Poincar\'e polynomial is the \textsl{Khovanov-Rozansky invariant}
$$
\KR_{N}(L) = \sum_{i,j\in\nZ} t^{i} q^j \dim_{\nQ}(H^{i,j}(L)) \, . 
$$

There is also a reduced version of Khovanov-Rozansky homology which depends on the choice of a component~$K$ of the link~$L$. If~$i$ is a label assigned to an edge lying on~$K$ in the link diagram~$D$, then we set $\overline{C}(D,K)=\bigotimes_{c\in\mathcal C} C[c] \otimes_{R}(R/(x_{i}))$. The cohomology of $\overline{C}(D,K)$ is another link invariant, and we call it the \textsl{reduced Khovanov-Rozansky homology $\overline{H}(L,K)$}. It is non-trivial in both $\nZ_{2}$-degrees, but there is a simple relation between them, as we will prove in Appendix~\ref{relationReducedUnreduced}: 
\begin{lemma}
\label{Z2gradingReduced}
Let $\overline{H}^k(L,K)$ denote the component of $\overline{H}(L,K)$ in $\nZ_{2}$-degree~$k$. If $p$ denotes the parity of $L$ then there is an isomorphism of $(\nZ\times\nZ)$-graded $\nQ$-vector spaces 
$$
\overline{H}^{p+1}(L,K) \cong \overline{H}^p(L,K)\{N-1\} \, .
$$
\end{lemma}
In view of this we define the \textsl{reduced Khovanov-Rozansky invariant $\overline{\KR}_{N}(L,K)$} as the Poincar\'e polynomial of $\overline{H}^p(L,K)$.

\medskip

Finally, let us show that we can compute Khovanov-Rozansky homology by compiling appropriate webs. By definition, in each cohomological degree the components of $C(D)$ are direct sums of the matrix factorisations $C(\Gamma)$ of zero, where~$\Gamma$ is a resolution of the link diagram~$D$. As explained earlier we promote~$\Gamma$ to a state web by decorating it with potentials $x_{i}^{N+1}$ and matrix factorisations $\Xcirc, \Xbul$. In order to compute Khovanov-Rozansky homology the first step is to compile these state webs, i.\,e.~to replace each cohomological component of $C(D)$ by the finite-dimensional $(\nZ\times\nZ)$-graded $\nQ$-vector spaces which are the direct sums of the compilations of the $C(\Gamma)$. 

In the second step we have to compile the differentials (which we denote $d_{\chi}$) of $C(D)$, i.\,e.~replace them by the maps that they induce on the compilations of the state webs; at this point we have a $\ZZ$-graded complex of $\ZZ$-graded $\QQ$-vector spaces, and we simply take cohomology to obtain the link invariant $H(L)$. Since by construction the differentials of $C(D)$ are direct sums of maps of the form $1\otimes\ldots\otimes 1\otimes\chi_{i}\otimes 1\otimes\ldots\otimes 1$, we can compile the $d_{\chi}$ as in the following general situation. 

Let us consider two decorated webs with identical underlying graphs and edge assignments, and whose matrix factorisations $X_{v}$ at vertex~$v$ are the same for all but one vertex~$v_{i}$ to which the second web assigns $X'_{v_{i}}$. We further assume to be given a morphism $\phi:X_{v_{i}}\longrightarrow X'_{v_{i}}$ which induces a map $\Phi=1\otimes\ldots\otimes 1\otimes\phi\otimes 1\otimes\ldots\otimes 1$ between the total matrix factorisations
\begin{align*}
X & =X_{v_{1}}\otimes\ldots\otimes X_{v_{i-1}} \otimes X_{v_{i}} \otimes X_{v_{i+1}}\otimes\ldots\otimes X_{v_{s}} \, , \\
X' & =X_{v_{1}}\otimes\ldots\otimes X_{v_{i-1}} \otimes X'_{v_{i}} \otimes X_{v_{i+1}}\otimes\ldots\otimes X_{v_{s}} \, .
\end{align*}
In the notation of Remark \ref{compileAndSplit} we have idempotents $\idem=\vartheta\circ\psi$ and $\idem'=\vartheta'\circ\psi'$ which can be split in $\hmf(k[\extvar],W)$ by morphisms $f,g, f',g'$ and matrix factorisations~$F, F'$ such that $f\circ g= 1_{F}, f' \circ g' = 1_{F'}$ and $g\circ f=\idem, g' \circ f' = \idem'$. Since $\psi$ is natural (see \cite[Lemma 7.6]{dm1102.2957}) there is a commutative diagram (where we omit grading shifts)
$$
\xymatrix@C+2pc{%
X \ar[d]_{\vartheta}\ar[r]^{\Phi} & X'\ar[d]^{\vartheta'}\\
X \otimes_{k[\intvar]} k[\intvar]/(\intvar^a)\langle n \rangle \ar[d]_{f}\ar[r]^{\Phi \otimes 1} & X' \otimes_{k[\intvar]} k[\intvar]/(\intvar^a)\langle n \rangle \ar[d]^{f'}\\
F \ar[r]^{f' \circ (\Phi \otimes 1) \circ g} & F'
}
$$
and thus the \textsl{compiled morphism} $f'\circ(\Phi\otimes 1)\circ g$ is what~$\Phi$ induces between the web compilations~$F$ and~$F'$. Note that we need no knowledge of the maps $\vartheta,\psi$ from Theorem~\ref{theorem:compilation_main} in this computation. 

In this way we compile every differential $d_{\chi}$ in $C(D)$ and compute the $\sln$ link homology $H(L)$. Its reduced variant is obtained in the same fashion after one sets $x_{i}=0$ everywhere in $C(D)$. We have implemented this method to automatically compute the invariants $\KR_{N}$ and $\overline{\KR}_{N}$ in Singular, see~\cite{cmWebCompileCode}. Another option to compute $H(L)$ would have been to use the MOY relations~\eqref{MOYcatDecomp1}--\eqref{MOYcatDecomp4} algorithmically and thus reduce every link diagram to unlinked circles, to which~\cite{kr0401268} associate a shift of the Jacobian of the potential $x^{N+1}$. We prefer our approach as it allows us to compile arbitrary webs (and in particular the invariant associated to arbitrary tangles in \cite{kr0401268}) and may be straightforwardly applied to compute other categorified link invariants such as those of~\cite{kr0701333, w0907.0695}.

\subsection{The explicit idempotents}\label{section:explicitidempotents}

We continue with the earlier notation, so $\cat{F}$ is a decorated web with total factorisation $X$. In this section we present a pair $(C, \idem)$ compiling $\cat{F}$.

The starting point is a list of polynomials acting null-homotopically on $X$, and a corresponding list of null-homotopies. Let $v$ be a vertex, $e$ an edge beginning at $v$ and $x$ a variable lying on $e$. Then $(X_v, d_v)$ denotes the matrix factorisation assigned to $v$. Certainly none of the potentials associated to the other edges incident at $v$ involve $x$, so from the identity $d_v^2 = (\sum_{f \in O_v} W_f - \sum_{f \in I_v} W_f) \cdot 1_{X_v}$ we deduce that
\[
\partial_x(d_v) \cdot d_v + d_v \cdot \partial_x(d_v) = \partial_x(W_e) \cdot 1_{X_v}\,.
\]
That is to say, $\partial_x(W_e)$ acts null-homotopically on $X_v$ with null-homotopy $\partial_x(d_v)$. But then the same is true of $X$, of which $X_v$ is a tensor factor, provided we interpret $\partial_x(d_v)$ as $1 \otimes \ldots \otimes 1\otimes \partial_x(d_v) \otimes 1\otimes \ldots \otimes 1$ on $X$ with suitable factors of the identity inserted, and Koszul signs.

Let $I$ denote the ideal in $k[\intvar]$ generated by the set of polynomials $\bs{\delta} = \{\partial_x(W_e)\}_{x,e}$ where $e$ ranges over all \textsl{internal} edges in the web and $x$ over each variable lying on $e$. We refer to $J = k[\intvar]/I$ as the \textsl{internal Jacobi algebra} of $\cat{F}$. It is the tensor product of the edge Jacobi algebras, 
\[
J = \bigotimes_{e} k[\bs{x}_e]/(\partial_{\bs{x}_e} W_e) \, ,
\]
where $e$ ranges over all internal edges, and so $J$ is finite-dimensional by hypothesis. Note that
\[
k[\bs{x}]/I = k[\extvar] \otimes_k J
\]
is a graded free $k[\extvar]$-module of finite rank, and consequently the quotient $X/\bs{\delta} X = X \otimes_{k[\bs{x}]} k[\bs{x}]/I$ is a finite-rank graded matrix factorisation of $W$ over $k[\extvar]$. At this point we can apply the main result of \cite{dm1102.2957} to give an explicit idempotent
\[
\bar\idem = \frac{1}{n!}(-1)^{\binom{n+1}{2}} \Big(\prod_{x,v} \partial_x(d_v)\Big) \circ \At(X)^n\,,
\]
such that the pair $(X/(\bs{\delta} X)\langle n \rangle, \bar\idem)$ compiles the web.\footnote{Strictly speaking this is not correct, as the pair $(X/(\bs{\delta} X)\langle n \rangle, \bar\idem)$ will split in the category of graded matrix factorisations not to $X$ but to $X\{ p \}$ for some nonzero integer $p$. We choose not to address this point since this is not the version of the idempotent that we actually compute with. In our approach below such subtleties will be taken into account.} Here $n$ is the total number of internal variables, the product is over all internal variables $x$ lying on an edge with origin $v$, and $\At(X)$ is the Atiyah class, see~\eqref{Atiyahclass} below. For theoretical purposes (for example if one wanted to pursue a residue formula for the Euler characteristic of $X$, or the correlators of \cite{kr0404189}) this compilation is the right one. But for implementation on a computer it is better to choose the sequence $\bs{\delta}$ differently, so that the connections in the construction (see below) become simple operations on polynomials.
\\

Let $x$ be an internal variable lying on an edge $e$. Since $k[\bs{x}_e]/(\partial_{\bs{x}_e} W_e)$ is finite-dimensional, there is an integer $a > 0$ such that $x^a \in (\partial_{\bs{x}_e} W_e)$. For simplicity, let us assume that this $a$ has been chosen large enough so that it works for every internal variable,\footnote{For $\sln$ homology, the integer $N$ will work: the partial derivatives $\partial_x(W_e)$ will be $N$-th powers of ring variables.} and write
\[
x^a = \sum_{y \in \bs{x}_e} b_{xy} \cdot \partial_y(W_e)
\]
for some polynomials $b_{xy}$. If $e$ has origin $v$ and we define $\lambda_x = \sum_y b_{xy} \cdot \partial_y(d_v)$ then
\[
\lambda_x \cdot d_v + d_v \cdot \lambda_x = x^a \cdot 1_{X_v}\,.
\]
We also write $\lambda_x$ for the map $1 \otimes \ldots \otimes 1 \otimes\lambda_x \otimes 1\otimes \ldots \otimes 1$ on $X$, which gives a null-homotopy for $x^a \cdot 1_X$. We can now apply the main theorem of \cite{dm1102.2957} to the sequence of polynomials $\intvar^a = \{ x^a \}_{x \in \intvar}$ and null-homotopies $\{ \lambda_x \}_{x \in \intvar}$. Again the formula for the idempotent will involve an Atiyah class, and to make this explicit we choose a connection on $X$. The algebra $k[\intvar]$ is a free module over the subring $k[\intvar^a] = k[x_1^a, x_2^a, \ldots]$ with a basis given by monomials in the internal variables $\bs{x}^\beta$ where the degree of each variable is smaller than~$a$, that is, $\beta_i < a$ for each $i$. From now on, let $\beta$ stand for such tuples of integers.

The K\"ahler differential on $k[\intvar^a]$ extends using this basis to a $k$-linear flat connection
\[
\nabla: k[\intvar] \lto k[\intvar] \otimes_{k[\intvar^a]} \Omega^1_{k[\intvar^a]/k}
\]
defined for a polynomial $f(\bs{x})$ in the internal variables by
\[
\nabla(f) = \sum_{x, \beta} \frac{\partial f_\beta}{\partial x^a} \bs{x}^\beta \ud (x^a)\,,
\]
where we write $f = \sum_\beta f_\beta \bs{x}^\beta$ for polynomials $f_\beta = f_\beta(\bs{x}^a)$ in $k[\intvar^a]$ and monomials $\bs{x}^\beta$ as above, restricted to degree smaller than~$a$ in each variable, and the sum is over all internal variables $x$. For example, if $x$ is an internal variable then
\begin{align*}
\nabla( 1 + x^2 + x^a + 2x^{a+2} ) &= \nabla( (1 + x^a) \cdot 1 + (1 + 2x^a) \cdot x^2 )\\
&= (1 + 2x^2 ) \cdot \ud(x^a)\,.
\end{align*}
Tensoring with $k[\extvar]$ we get a flat $k$-linear connection (writing $\Omega^1$ for $\Omega^1_{k[\extvar,\intvar^a]/k[\extvar]}$)
\[
\nabla: k[\bs{x}] \lto k[\bs{x}] \otimes_{k[\extvar,\intvar^a]} \Omega^1\,,
\]
which is ``standard'' in the sense of \cite[Section 8.1]{dm1102.2957} (essentially, a hypothesis of zero characteristic). The upshot is that for each internal variable $x$ we have defined a $k[\extvar]$-linear operator $\partial_{x^a}$ on $k[\bs{x}]$. It is obvious that $\partial_{x^a}$ is, modulo $x^a$, just division by $x^a$ without remainder:

\begin{lemma}\label{lemma:division_woremainder} Given $f \in k[\bs{x}]$ write $f = x^a g + h$ for a polynomial $h$ of $x$-degree smaller than~$a$. Then
\[
\partial_{x^a}(f) = x^a \cdot \partial_{x^a}(g) + g\,.
\]
\end{lemma}

Using our fixed homogeneous $k[\bs{x}]$-basis for $X$ we extend $\nabla$ to a flat $k$-linear connection
\[
\nabla: X \lto X \otimes_{k[\extvar,\intvar^a]} \Omega^1\,.
\]
The \textsl{Atiyah class} 
\begin{equation}
\label{Atiyahclass}
\At(X) := [d_X, \nabla]
\end{equation}
is a $k[\extvar,\intvar^a]$-linear morphism of matrix factorisations, which up to homotopy does not depend on the choice of connection. Composing the Atiyah class with itself $n$~times, where $n$ is the number of internal variables, gives a map $X \lto X \otimes \Omega^n \cong X$ where we implicitly choose an ordering $x_1, x_2, \ldots,x_n$ of the internal variables and identify $\ud(x_1^a) \wedge \ud(x_2^a) \wedge \ldots \wedge \ud(x_n^a)$ with $1$. This map $\At(X)^n$ induces a $k[\extvar]$-linear map on the quotient
\[
X/(\intvar^a X) = X \otimes_{k[\intvar]} k[\intvar]/(\intvar^a)\,
\]
which is a finite-rank graded matrix factorisation of $W$ over $k[\extvar]$. The homotopies $\lambda_x$ also give $k[\extvar]$-linear maps on this quotient, and we define a $k[\extvar]$-linear endomorphism of $X/(\intvar^a X)$ by
\begin{equation}
\idem = \frac{1}{n!}(-1)^{\binom{n+1}{2}} \Big(\prod_x \lambda_x \Big) \circ \At(X)^n\,.
\end{equation}
With this notation:

\begin{theorem}\label{theorem:compilation_main} There is a diagram in $\HMF(k[\extvar], W)$
\begin{equation}
\label{eq:compilation_main1}
\xymatrix@C+3pc{
X/(\intvar^a X)\{ n(c - 2a) \}\langle n \rangle \ar@<-1ex>[r]_-{\psi} & \varphi_*(X) \ar@<-1ex>[l]_-{\vartheta}
}
\end{equation}
with $\vartheta \circ \psi = \idem$ and $\psi \circ \vartheta = 1_{\varphi_*(X)}$. That is, the pair
\[
\big(X/(\intvar^a X)\{ n(c - 2a) \}\langle n \rangle, \idem\big)
\]
compiles the decorated web $\cat{F}$.
\end{theorem}
\begin{proof}
This follows from the main theorem of \cite{dm1102.2957}, with the caveat that the theorem just cited only discusses ungraded matrix factorisations. Let us justify why $\idem$ has degree zero; the same argument will show that the specific maps $\vartheta$ and $\psi$ given in \textsl{ibid.} both have degree zero.

For each internal variable $x$, $\lambda_x$ has bidegree $(-1, 2a - c)$ on $X$, so the product over $x$ has bidegree $(-n, n(2a - c))$. The degree of $\At(X)^n = [d_X, \nabla]^n$ has a contribution of $(n, nc)$ from the $n$ copies of $d_X$ plus a contribution from the operator $\partial_{x^a}$ for each internal variable; each such operator has degree $-2a$, so $\At(X)^n$ has bidegree $(n, n(c-2a))$ and the product $\idem$ has bidegree $(0,0)$.
\end{proof}

\begin{remark} A computer understands only matrices, so let us spell out how we write $\idem$ as a matrix. Let $\{ \xi_j \}_j$ denote our homogeneous basis for $X$ as a $k[\bs{x}]$-module and let $Q$ be the matrix of $d_X$ in this basis. Then with $\bs{x}^\beta$ denoting a monomial in the internal variables with all exponents $\beta_i < a$, we see that $\{\bs{x}^\beta\}_\beta$ gives a $k$-basis of $k[\intvar]/(\intvar^a)$ and therefore $\{ \bs{x}^\beta \xi_j \}_{\beta, j}$ is a $k[\extvar]$-basis for $X/(\intvar^a X)$.

The matrix $Q'$ of the differential $d_X$ in this basis is obtained from $Q$ by a process we call~\textsl{inflation}: $Q$ is a matrix of finite size with entries in $k[\bs{x}] = k[\extvar] \otimes_k k[\intvar]$. For any monomial $\bs{x}^\alpha$ in the internal variables let $M_{\bs{x}^\alpha}$ be the scalar matrix representing the $k$-linear map of multiplication with $\bs{x}^\alpha$ on~$k[\intvar]/(\intvar^a)$. We can write every entry of~$Q$ in the form $\sum_{\alpha} p_{\alpha} \bs{x}^\alpha$ with $p_{\alpha}\in k[\extvar]$. Then $Q'$ is obtained from $D$ by replacing every such entry by the ``inflated'', matrix-valued entry $\sum_{\alpha} p_{\alpha} M_{\bs{x}^\alpha}$. 
The same kind of inflation gives us the matrix of $\lambda_x$ in the new basis.

Next we index the variables $\intvar = \{x_1,x_2,\ldots,x_n\}$ and use \cite[Corollary 10.4]{dm1102.2957} to write
\[
\idem = \frac{1}{n!}(-1)^{\binom{n}{2}} \sum_{\tau \in S_n} \textup{sgn}(\tau)\Big( \prod_{i=1}^n \lambda_{x_i} \Big) [\partial_{\tau(1)}, d_X] \ldots [\partial_{\tau(n)}, d_X]
\]
where $\partial_j = \partial_{x_j}$. Let $\sigma: X/(\intvar^a X) \lto X$ be the $k[\extvar]$-linear map defined by $\sigma( \bs{x}^\beta \xi_j ) = \bs{x}^\beta \xi_j$ and let $\pi: X \lto X/(\intvar^a X)$ be the projection. The commutator $[\partial_j, d_X]$ on $X/(\intvar^a X)$ is equal to the composite $\pi \circ \partial_j \circ d_X \circ \sigma$. Let $\textup{div}_j(f)$ denote the division of $f$ by $x_j^a$ without remainder and for any monomial $\bs{x}^\alpha$ in the internal variables let $L_{j,\bs{x}^\alpha}$ be the scalar matrix representing the $k$-linear map on $k[\intvar]/(\intvar^a)$ which is multiplication by $\bs{x}^\alpha$ followed by $\textup{div}_j$. Since $\pi \circ \partial_j = \pi \circ \textup{div}_j$ (Lemma~\ref{lemma:division_woremainder}) the matrix of $[\partial_j, d_X]$ in the new basis is obtained from $Q$ by writing every entry of $Q$ in the form $\sum_{\alpha} p_{\alpha} \bs{x}^\alpha$ with $p_{\alpha}\in k[\extvar]$ and then replacing such an entry by the matrix $\sum_{\alpha} p_{\alpha} L_{j,\bs{x}^\alpha}$.

Multiplying together the matrices for the $\lambda_{x_j}$ and $[\partial_j, d_X]$ we have the matrix for $\idem$.
\end{remark}

Note that in general~$\idem$ the identity $\idem^2 = \idem$ only holds up to homotopy. However, we can use the method of lifting modulo a nilpotent ideal~\cite[Section~3.6]{LambekRingsModules} to find an explicit representative~$\idem'$ of the class of~$\idem$ that is strictly idempotent. Then one can compute a splitting of~$\idem'$ using standard methods. 

The conclusion is that we can algorithmically compile any decorated web, i.\,e.~for arbitrary tensor products of matrix factoriations one can explicitly compute a homotopy equivalent finite-rank matrix factorisation. In practice it is more efficient to compile a web by repeatedly computing and splitting idempotents locally instead of doing it ``all at once'' for the total matrix factorisation. This minimises the size of the involved ``inflated'' matrices and hence the computing time, and can be done as follows: we choose a path through the underlying graph that connects all vertices, and begin by reducing the tensor product of the first two matrix factorisations to finite rank (where we can apply Theorem~\ref{theorem:compilation_main} to reduce the pushforward one variable at a time). Then we compute the tensor product of the result with the next matrix factorisation in the path along the same lines, and iterate the procedure until the whole web is compiled. 

We have implemented this construction in the computer algebra system Singular, and we refer to~\cite{cmWebCompileCode} for a detailed description as well as several commented examples.

\section{Computational results}
\label{compres}

We now present our results for the $\sln$ Khovanov-Rozansky homology of all prime links with up to~6 crossings. The explicit Poincar\'e polynomials are listed in Tables~\ref{unreducedKR} and~\ref{reducedKR}. 

As discussed in the introduction, our construction via the splitting of idempotents provides the first method to compute Khovanov-Rozansky invariants directly from the original definition in~\cite{kr0401268}. However, there have been previous results and proposals for Khovanov-Rozansky invariants for certain links which relied on indirect arguments. Before we will discuss our results for links whose Khovanov-Rozansky invariants had not been studied before, we shall first recall the known results and proposals, and compare them with our own direct computations. 

So far the most extensive results were obtained in~\cite{r0508510, r0607544}. There it was shown that the reduced Khovanov-Rozansky invariants $\overline{\KR}_{N}$ of 2-bridge knots can be expressed in terms of their Homfly polynomial and signature (see e.\,g.~\cite{kawauchibook} for the definitions), and a similar result is true for 2-bridge links. Furthermore, explicit expressions were obtained for the unreduced invariants $\KR_{N}$ with $N>4$ for torus knots $T_{2,m}$ and the figure-eight knot $4_{1}$. For links with up to 6 crossings and small values of~$N$ we can verify these results and extend them to the cases $N=3$ and $N=4$ where necessary. We refer to Tables~\ref{unreducedKR} and~\ref{reducedKR} for further details. 

The expressions for $\KR_{N}(T_{2,3})$ and $\KR_{N}(T_{2,5})$ were already given in~\cite{gsv0412243}. Explicit proposals were also presented for reduced Khovanov-Rozansky invariants of further torus knots~\cite{dgr0505662, as1105.5117, dbmmss1106.4305} and all prime knots with up to ten crossings~\cite{dgr0505662}; unreduced invariants for the Hopf link were proposed in~\cite{gikv0705.1368}. For small numbers of crossings and small values of~$N$ our results verify all these proposals.

\medskip

Let us now move on to a discussion of Khovanov-Rozansky invariants that have not been previously studied in the literature. Our method allows to compute reduced and unreduced homologies alike, and it does not discriminate against any types of links. The main limiting factors for the computations to finish on the ordinary PCs we had access to are the number of crossings and the value of~$N$. In the language of Section~\ref{compilewebs}, these determine the number of vertices in the webs to be compiled and the size of ``inflated'' matrices, respectively. Currently this allows us to calculate Khovanov-Rozansky invariants for links with up to 6 crossings and 3 components; future computer generations will push these limits further. 

The unreduced and reduced Khovanov-Rozansky invariants that we obtained are collected in the tables below. For links with more than one component it matters how the components are oriented \textsl{relatively} to one another. We indicate this by referring to different ``versions'' of all the links with at least two components in the tables. For example, $4_{1}^2\,\text{{\footnotesize (v1)}}$ and $4_{1}^2\,\text{{\footnotesize (v2)}}$ denote the two versions of Solomon's link. Note also that we do not list their mirrors separately as the mirror has the same Khovanov-Rozansky invariant after the substitutions $q\longmapsto q^{-1}$, $t\longmapsto t^{-1}$. 

We conclude by listing some observations and comments on our results; we let the Tables~\ref{unreducedKR} and~\ref{reducedKR} speak for themselves otherwise. 

\begin{itemize}
\item 
In all our examples the choice of marked component for reduced Khovanov-Rozansky homology does not affect the invariants, i.\,e.~$\overline{\KR}_{N}(L,K) = \overline{\KR}_{N}(L,K')$ for all components $K,K'$ of the link~$L$. 
\item 
Unreduced Khovanov-Rozansky homology is non-trivial only in one $\nZ_{2}$-degree, but reduced homology is non-trivial in both $\nZ_{2}$-degrees. In fact the even and odd components are the same up to a grading shift by $N-1$, see Lemma~\ref{Z2gradingReduced}. For this reason we defined and list the Poincar\'e polynomials $\overline{\KR}_{N}$ only for the $\nZ_{2}$-degree given by the parity of the link. The invariant for the opposite $\nZ_{2}$-degree is then obtained by multiplication with $q^{N-1}$. 
\item 
The choice of relative orientation of components, called the version of a link above, can be detected by Khovanov-Rozansky homology for some but not all links. More precisely, both the reduced and unreduced invariants depend on the version in the case of the links $4_{1}^2, 6_{1}^2, 
6_{3}^2, 6_{1}^3, 6_{2}^3, 6_{3}^3$ in our examples. 
\item 
Also for some examples that are not covered by the work of~\cite{r0508510, r0607544} do we offer expressions for their Khovanov-Rozansky invariants for arbitrary~$N$. In some cases these are simply guesses that are true for all values of~$N$ that we considered. 

However, for the 2-component links $4_{1}^2\,\text{{\footnotesize (v1)}}$, $4_{1}^2\,\text{{\footnotesize (v2)}}$, $5_{1}^2\,\text{{\footnotesize (v1)}}$, $5_{1}^2\,\text{{\footnotesize (v2)}}$, $6_{3}^2\,\text{{\footnotesize (v1)}}$ we took advantage of~\cite[Eq.\,(5.5)]{gsv0412243} which proposes for any 2-component link~$L$ with components $K_{1}$ and $K_{2}$ a relation between $\KR_{N}(L)$ on the one hand and $\KR_N(K_{1})$, $\KR_N(K_{2})$ and a certain expression involving integers $D_{Q,s,r}$ that count BPS states on the other hand. For the examples mentioned it is straightforward to extract integers $D_{Q,s,r}$ from our results for $\KR_{2}(L)$ and $\KR_{3}(L)$ such that~\cite[Eq.\,(5.5)]{gsv0412243} is satisfied, and this then in turn allows to use~\cite[Eq.\,(5.5)]{gsv0412243} to propose an expression $\KR_{N}(L)$ for arbitrary~$N$.\footnote{We also note that the parameter~$\alpha$ left unexplained in~\cite{gsv0412243} turns out to be twice the linking number in our examples.}
\item 
An interesting question is whether Khovanov-Rozansky homology can distinguish mutant pairs of links. It is known that $\overline{\KR}_{N}$ for odd~$N$ is invariant under mutations~\cite{j1101.3302}, and that $\KR_{2}$ cannot resolve the simplest mutant pairs which have 11 crossings. Thus testing whether Khovanov-Rozansky homology can detect mutation is presently beyond the capabilities of our code when run on ordinary PCs. 
\end{itemize}

{\footnotesize 
\begin{longtable}{p{0.07\textwidth}|p{0.87\textwidth}} 
\caption{Unreduced Khovanov-Rozansky $\sln$ invariants for prime links up to 6 crossings}
\label{unreducedKR}
\\
link $L$ & $\KR_{N}(L)$ \\
\hline\hline
\endfirsthead
link $L$ & $\KR_{N}(L)$ \\
\hline\hline
\endhead
\hline\hline
\endfoot
$2_{1}^2\,\text{{\tiny (v1\&v2)}}$ 
\includegraphics[scale=0.07,angle=0]{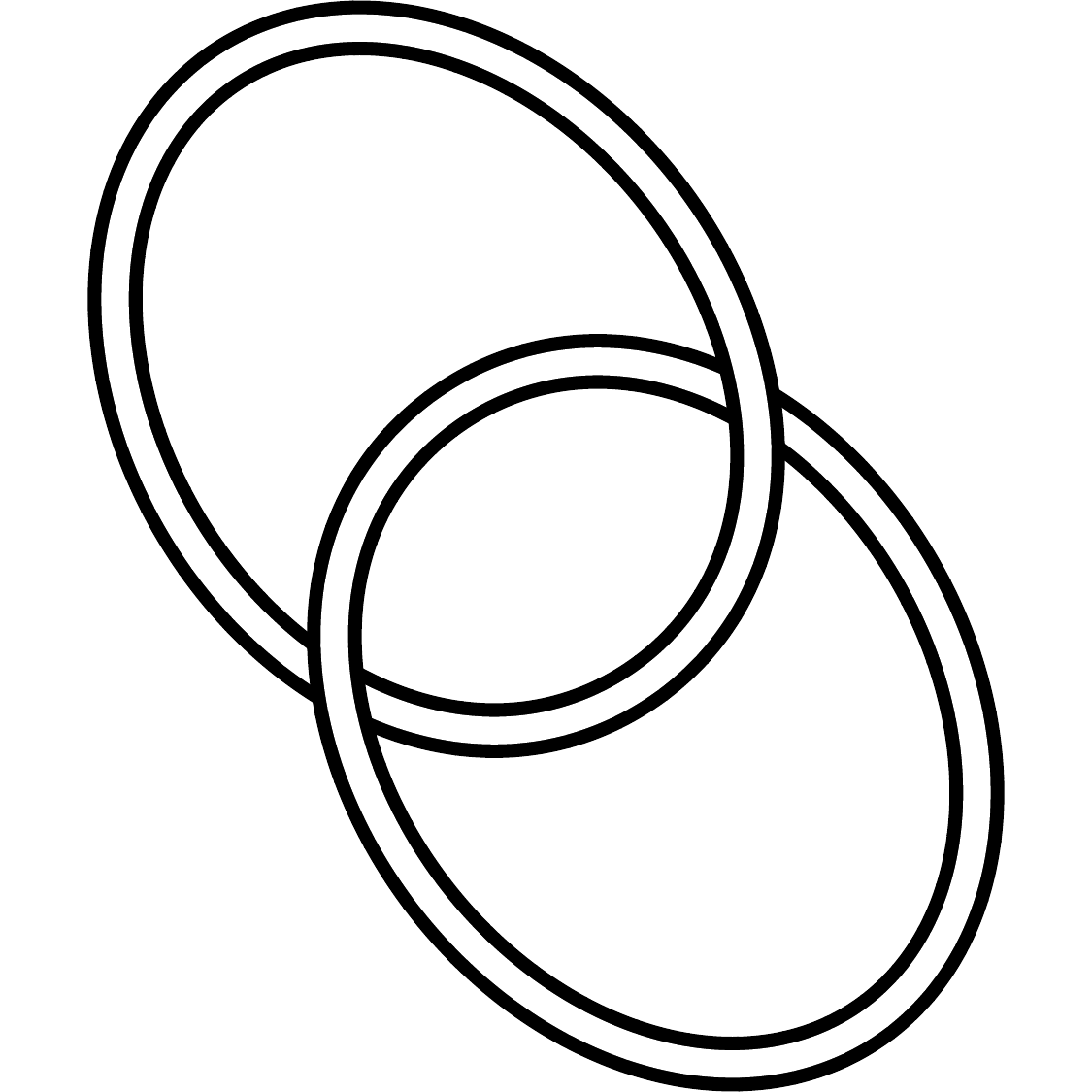} 
&
\newline 
$
\displaystyle{%
q^{N-1} \Big(\frac{q^N-q^{-N}}{q-q^{-1}}\Big) + q^{2N} \Big(\frac{q^N-q^{-N}}{q-q^{-1}}\Big)^2 t^2 - q^{N+1} \Big(\frac{q^N-q^{-N}}{q-q^{-1}}\Big) t^2
}
$
\newline\newline\newline
Checked up to $N=11$. Agrees with proposal of~\cite{gikv0705.1368}.
\\
\hline
$3_{1}$ 
\includegraphics[scale=0.07,angle=0]{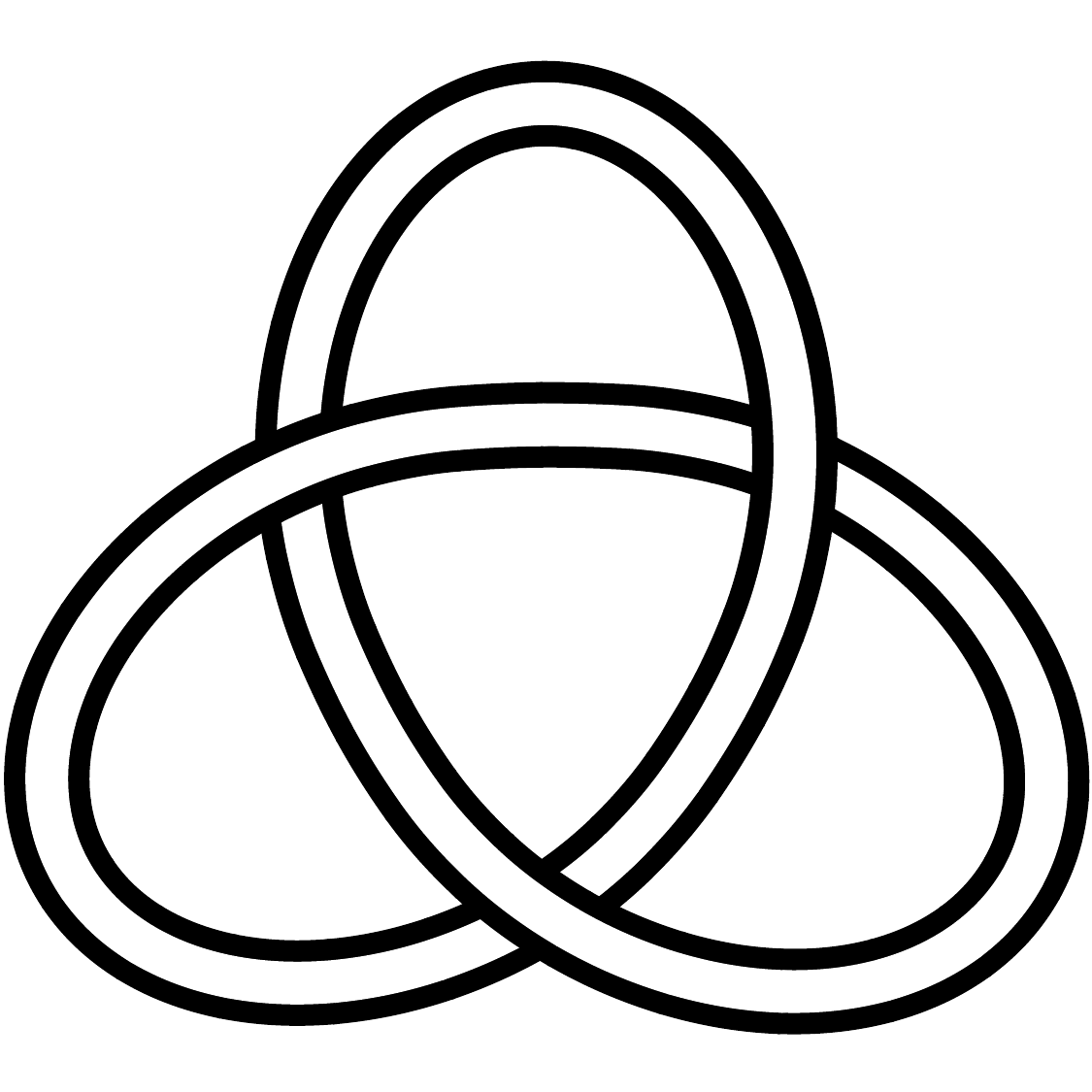} 
& 
\newline
$
\displaystyle{%
q^{2N-2} \Big( \frac{q^N-q^{-N}}{q-q^{-1}} + \frac{q^{N-1}-q^{-N+1}}{q-q^{-1}} q^{-1} (1 + q^{2N} t^{-1}) q^{4} t^{-2} \Big)
}
$
\newline\newline\newline
Checked up to $N=11$. Extends result of~\cite{r0508510} to $N=3$ and $N=4$. 
\\
\hline
$4_{1}$ 
\includegraphics[scale=0.07,angle=0]{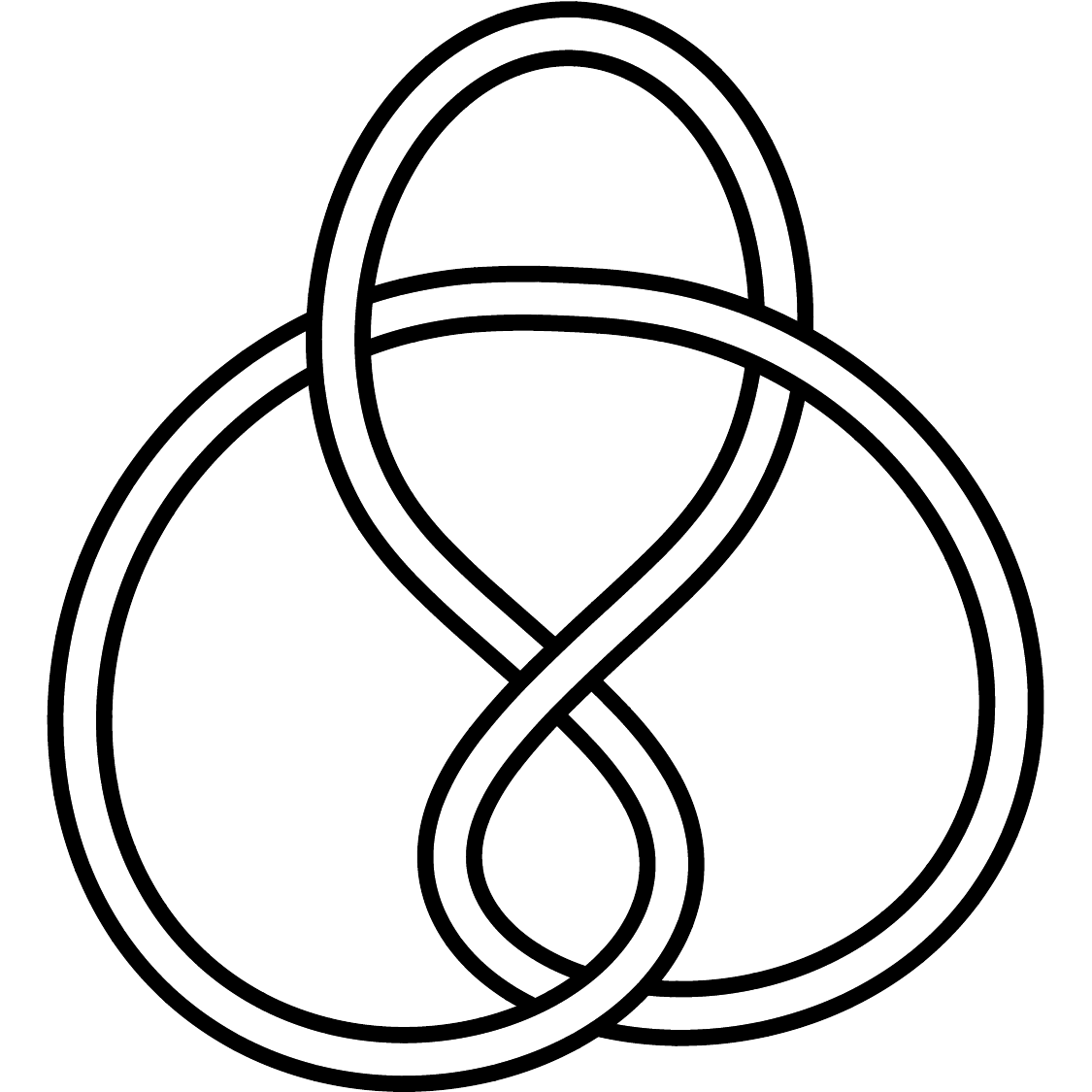} 
&
\newline 
$
\displaystyle{%
\frac{q^N-q^{-N}}{q-q^{-1}} + \frac{q^{N-1}-q^{-N+1}}{q-q^{-1}}  \Big( q^{2N+1} t^{-2} + q t^{-1} + q^{-1} t + q^{-2N-1} t^2 \Big)
}
$
\newline\newline\newline
Checked up to $N=5$. Extends result of~\cite{r0508510} to $N=3$ and $N=4$. 
\\
\hline
$4_{1}^2\,\text{{\tiny (v1)}}$ 
\includegraphics[scale=0.07,angle=0]{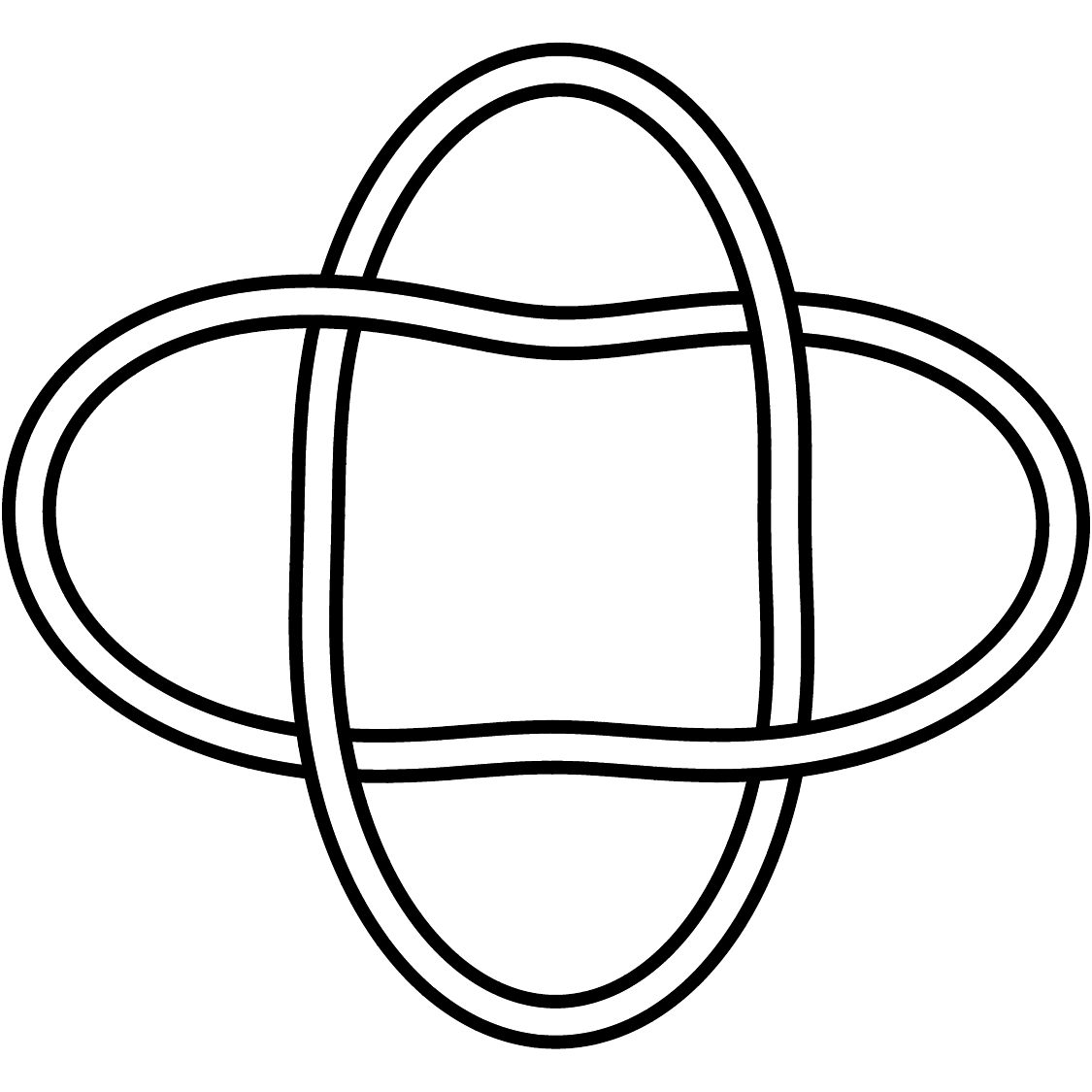} 
& 
\newline
$
\displaystyle{%
\frac{q^{-4 N}}{\left(q^2 - 1\right)^2} \left(q^{2-2 N} \left(q^{2 N}-1\right)^2 t^4+\left(q^2 - 1\right) \left(-q^2 t^2+t^4+q^{4 N} \left(q^2+t\right)-q^{2 N} (1+t) \left(q^2+(t-1) t^2\right)\right)\right) 
}
$
\newline\newline\newline
Checked up to $N=5$.
\\
\hline
$4_{1}^2\,\text{{\tiny (v2)}}$ 
\includegraphics[scale=0.07,angle=0]{link4_1_2.pdf} 
&
\newline 
$
\displaystyle{%
\frac{q^{-2-6 N}}{\left(q^2-1\right)^2}} \left(q^2 t^3 \left(q^2-q^4+t\right)+q^{4 N} \left(q^6 \left(q^2-1\right)+q^2 \left(q^2-1\right) t^2+t^4\right) \right.
$
\newline
$
\displaystyle{%
\left. \qquad\qquad\quad -q^{2 N} \left(q^8+t^4-q^4 t^2 (1+t)+q^2 t^3 (1+t)+q^6 \left(t^2-1\right)\right)\right)
}
$
\newline
Checked up to $N=5$.
\\
\hline
$5_{1}$ 
\includegraphics[scale=0.07,angle=0]{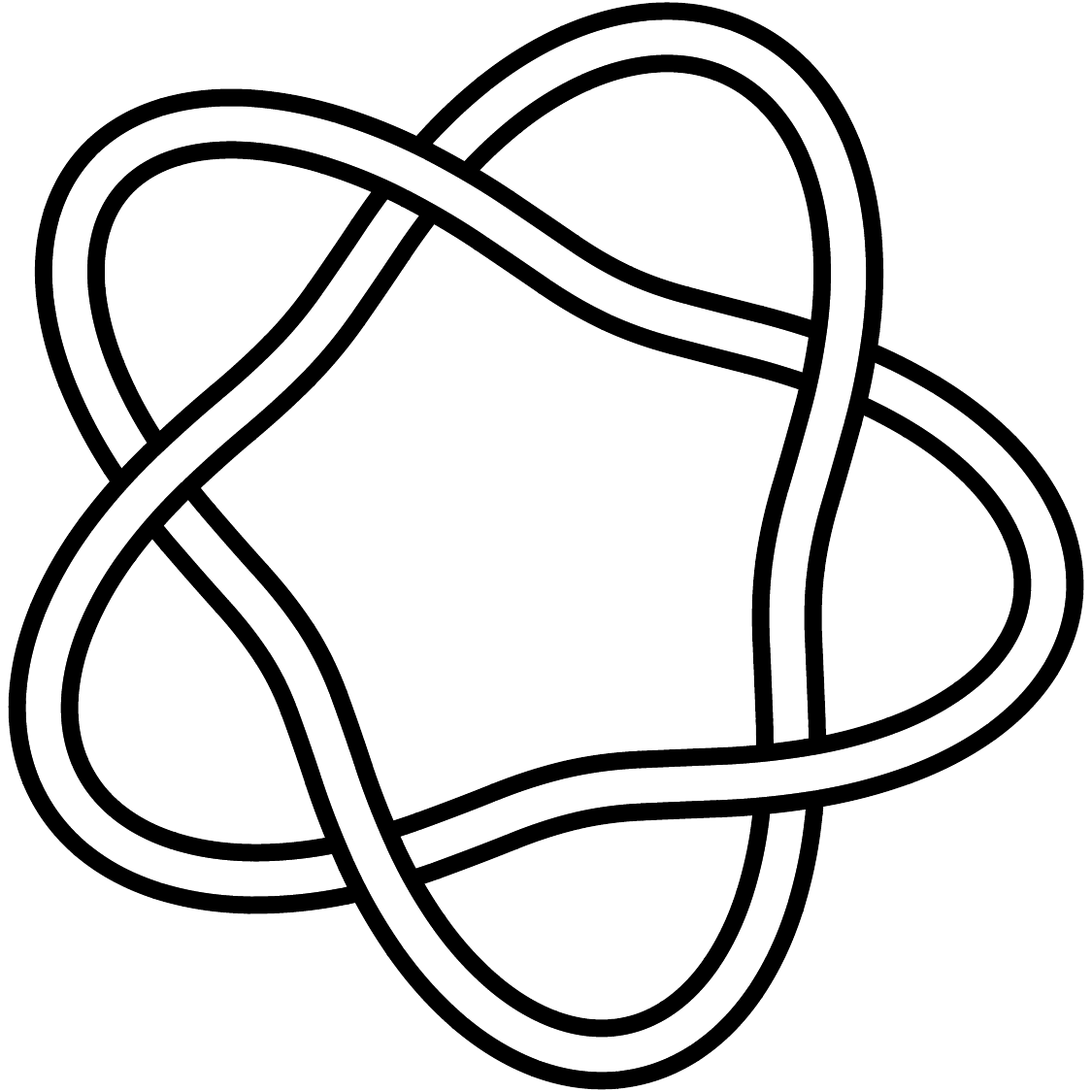} 
& 
\newline
$
\displaystyle{%
q^{4N-4} \Big( \frac{q^N-q^{-N}}{q-q^{-1}} + \frac{q^{N-1}-q^{-N+1}}{q-q^{-1}} q^{-1} (1 + q^{2N} t^{-1}) (q^{4} t^{-2} + q^{8} t^{-4}) \Big)
}
$
\newline\newline\newline
Checked up to $N=6$. Extends result of~\cite{r0508510} to $N=3$ and $N=4$. 
\\
\hline
$5_{2}$ 
\includegraphics[scale=0.07,angle=0]{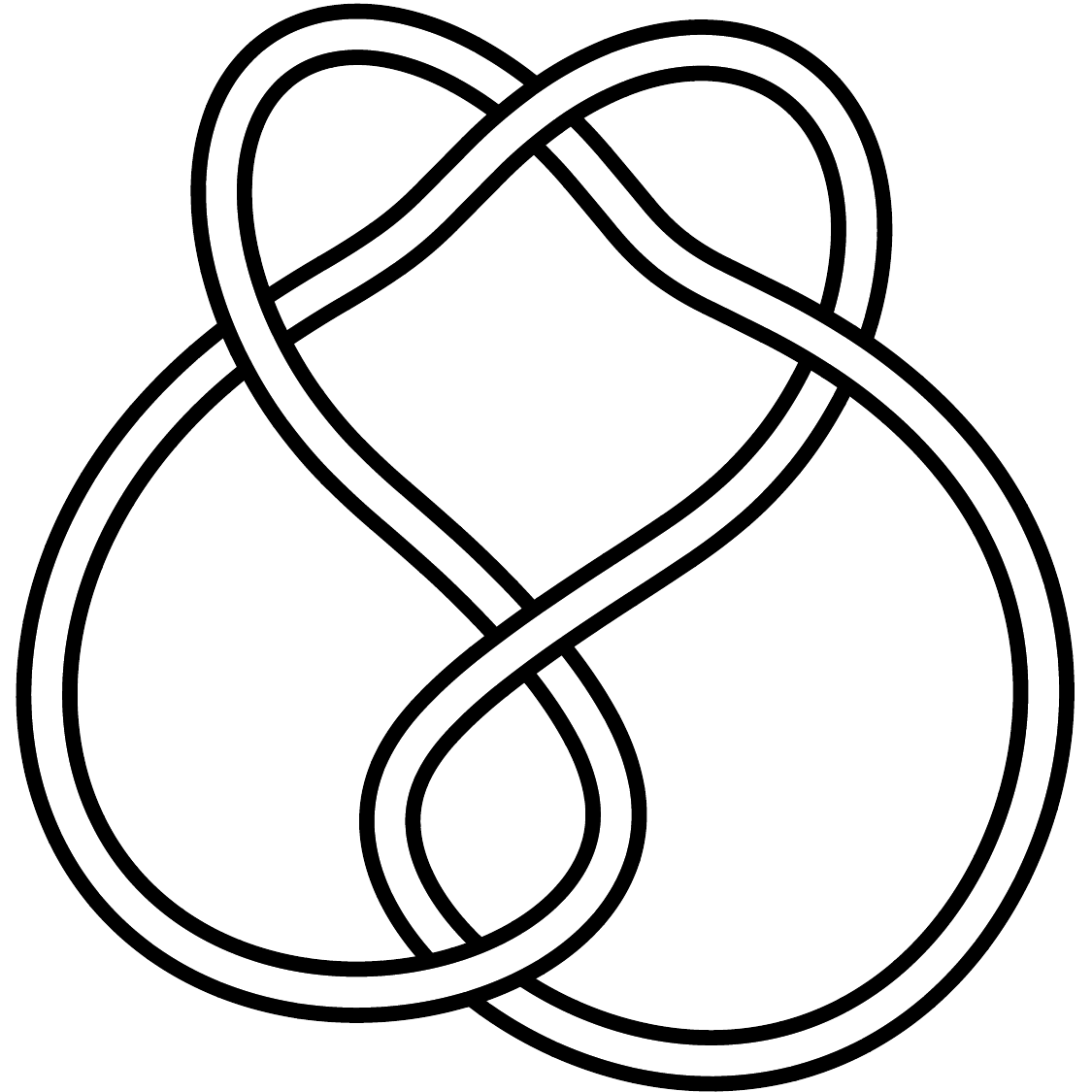} 
& 
$
\newline 
\displaystyle{%
\KR_{2} = \frac{1}{q^3}+\frac{1}{q}+\frac{t}{q^3}+\frac{t^2}{q^7}+\frac{t^2}{q^5}+\frac{t^3}{q^9}+\frac{t^4}{q^9}+\frac{t^5}{q^{13}}
}
$
\newline 
$
\displaystyle{%
\KR_{3} = \frac{1}{q^6}+\frac{1}{q^4}+\frac{1}{q^2}+\frac{t}{q^6}+\frac{t}{q^4}+\frac{t^2}{q^{12}}+\frac{t^2}{q^{10}}+\frac{t^2}{q^8}+\frac{t^2}{q^6}+\frac{t^3}{q^{14}}+\frac{t^3}{q^{12}}+\frac{t^4}{q^{14}}+\frac{t^4}{q^{12}}+\frac{t^5}{q^{20}}+\frac{t^5}{q^{18}}
}
$
\newline 
$
\displaystyle{%
\KR_{4} = \frac{1}{q^9}+\frac{1}{q^7}+\frac{1}{q^5}+\frac{1}{q^3}+\frac{t}{q^9}+\frac{t}{q^7}+\frac{t}{q^5}+\frac{t^2}{q^{17}}+\frac{t^2}{q^{15}}+\frac{t^2}{q^{13}}+\frac{t^2}{q^{11}}+\frac{t^2}{q^9}+\frac{t^2}{q^7}+\frac{t^3}{q^{19}}+\frac{t^3}{q^{17}}+\frac{t^3}{q^{15}}
}
$
\newline
$
\displaystyle{%
\qquad\qquad +\frac{t^4}{q^{19}}+\frac{t^4}{q^{17}}+\frac{t^4}{q^{15}}+\frac{t^5}{q^{27}}+\frac{t^5}{q^{25}}+\frac{t^5}{q^{23}}
}
$
\newline 
\\
\hline
$5_{1}^2\,\text{{\tiny (v1\&v2)}}$ 
\includegraphics[scale=0.07,angle=0]{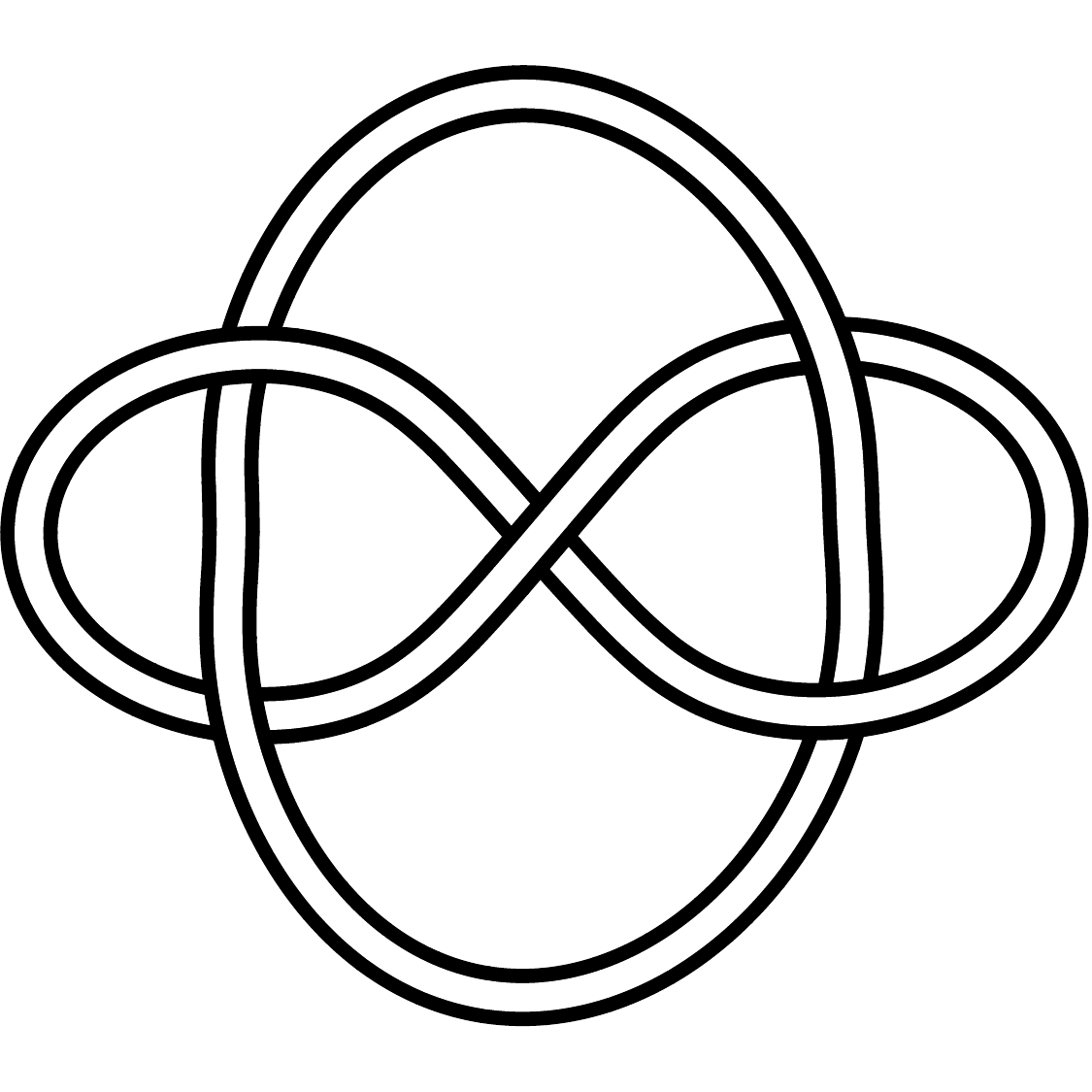} 
&
\newline 
$
\displaystyle{%
\frac{q^{-2-4 N}}{\left(q^2-1\right)^2} \left(q^{4+2 N} \left(q^{2 N}-1\right)^2+t^{-2}\left(\left(q^2-1\right) \left(q^{2 N}-q^2\right) \left(q^{2+4 N} (q-t) (q+t)+t^4 \left(q^2+t\right) \right.\right.\right.
}
$
\newline
$
\displaystyle{%
\left.\left.\left. \qquad\qquad\quad+q^{2 N} t \left(q^2+t\right) \left(q^2+t^2\right)\right)\right)\right)
}
$
\newline
Checked up to $N=4$.
\\
\hline
$6_{1}$ 
\includegraphics[scale=0.07,angle=0]{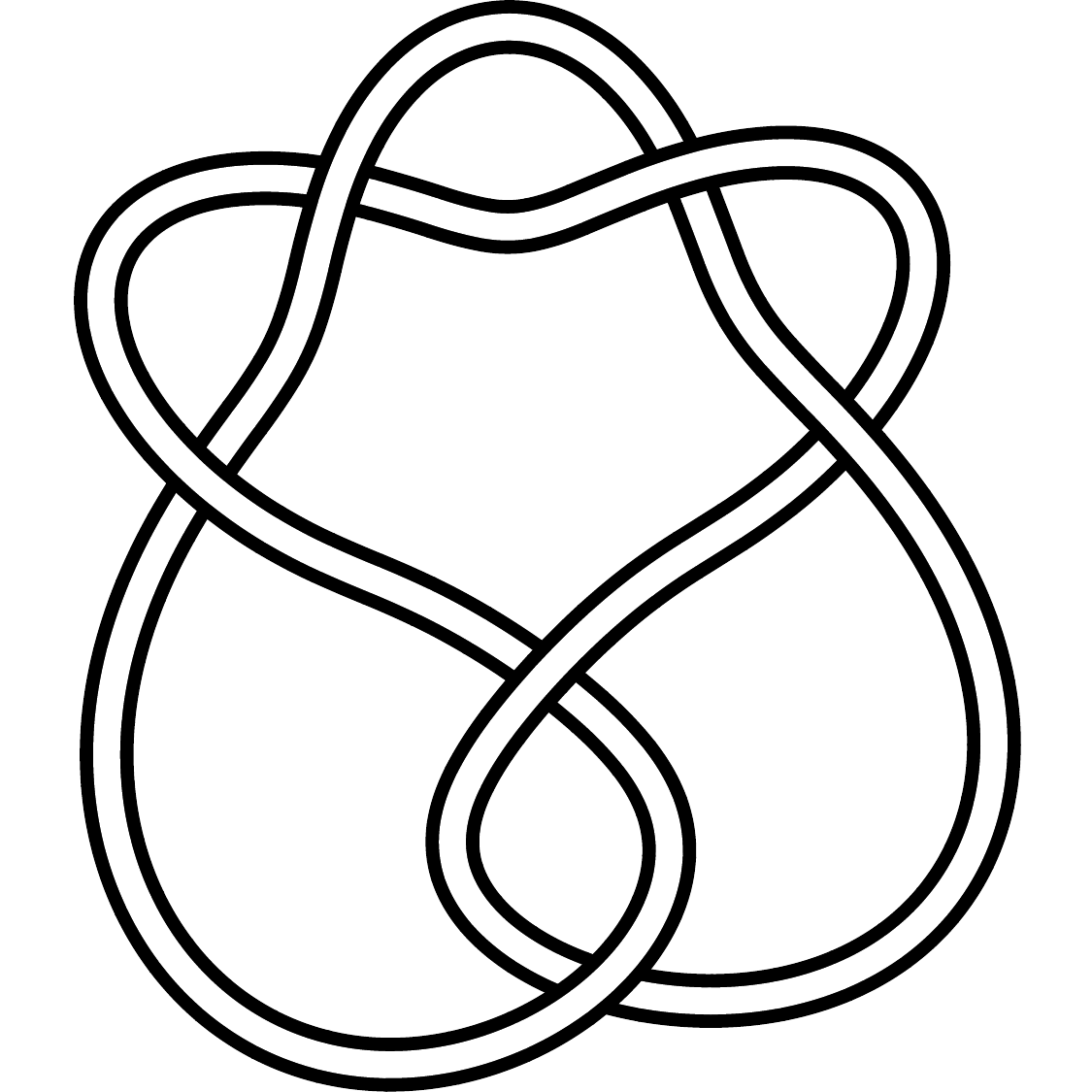} 
&
\newline 
$
\displaystyle{%
\KR_{2} = \frac{1}{q}+2 q+\frac{q^5}{t^2}+\frac{q}{t}+\frac{t}{q^3}+\frac{t}{q}+\frac{t^2}{q^5}+\frac{t^3}{q^5}+\frac{t^4}{q^9}
}
$
\newline 
$
\displaystyle{%
\KR_{3} = 2+\frac{1}{q^2}+2 q^2+\frac{q^6}{t^2}+\frac{q^8}{t^2}+\frac{1}{t}+\frac{q^2}{t}+t+\frac{t}{q^6}+\frac{t}{q^4}+\frac{t}{q^2}+\frac{t^2}{q^8}+\frac{t^2}{q^6}+\frac{t^3}{q^8}+\frac{t^3}{q^6}+\frac{t^4}{q^{14}}+\frac{t^4}{q^{12}}
}
$
\newline
\\
\hline
$6_{2}$ 
\includegraphics[scale=0.07,angle=0]{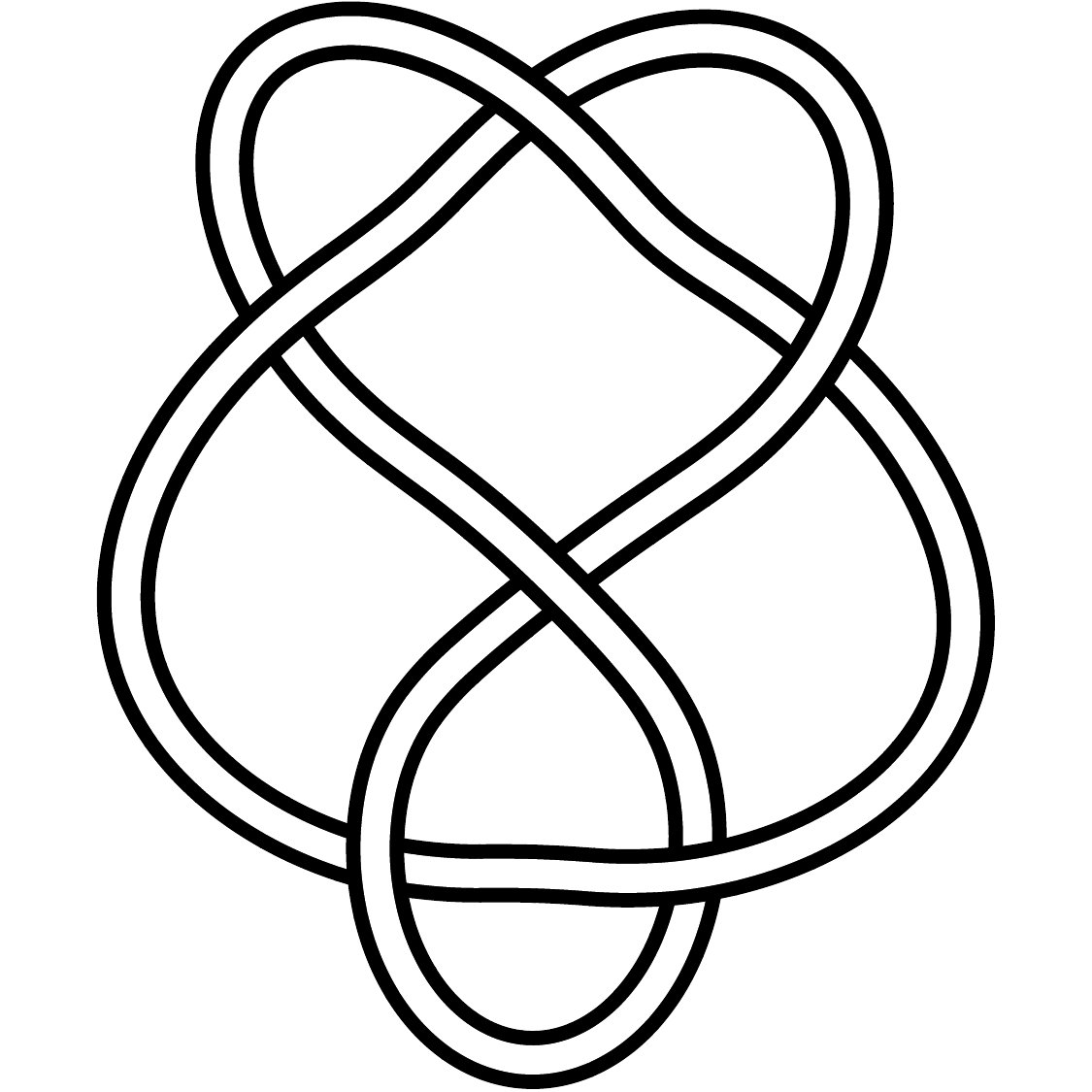} 
& 
\newline 
$
\displaystyle{%
\KR_{2} = \frac{1}{q^3}+\frac{2}{q}+\frac{q^3}{t^2}+\frac{1}{q t}+\frac{t}{q^5}+\frac{t}{q^3}+\frac{t^2}{q^7}+\frac{t^2}{q^5}+\frac{t^3}{q^9}+\frac{t^3}{q^7}+\frac{t^4}{q^{11}}
}
$
\newline 
$
\displaystyle{%
\KR_{3} = 1+\frac{1}{q^6}+\frac{1}{q^4}+\frac{2}{q^2}+\frac{q^2}{t^2}+\frac{q^4}{t^2}+\frac{1}{q^4 t}+\frac{1}{q^2 t}+\frac{t}{q^8}+\frac{2 t}{q^6}+\frac{t}{q^4}+\frac{t^2}{q^{12}}+\frac{t^2}{q^{10}}+\frac{t^2}{q^8}+\frac{t^2}{q^6}+\frac{t^3}{q^{14}}+\frac{t^3}{q^{12}}
}
$
\newline
$
\displaystyle{%
\qquad\qquad +\frac{t^3}{q^{10}}+\frac{t^3}{q^8}+\frac{t^4}{q^{16}}+\frac{t^4}{q^{14}}
}
$
\newline 
$
\displaystyle{%
\KR_{4} = \frac{1}{q^9}+\frac{1}{q^7}+\frac{1}{q^5}+\frac{2}{q^3}+\frac{1}{q}+q+\frac{q}{t^2}+\frac{q^3}{t^2}+\frac{q^5}{t^2}+\frac{1}{q^7 t}+\frac{1}{q^5 t}+\frac{1}{q^3 t}+\frac{t}{q^{11}}+\frac{2 t}{q^9}+\frac{2 t}{q^7}+\frac{t}{q^5}+\frac{t^2}{q^{17}}
}
$
\newline
$
\displaystyle{%
\qquad\qquad +\frac{t^2}{q^{15}}+\frac{t^2}{q^{13}}+\frac{t^2}{q^{11}}+\frac{t^2}{q^9}+\frac{t^2}{q^7}+\frac{t^3}{q^{19}}+\frac{t^3}{q^{17}}+\frac{t^3}{q^{15}}+\frac{t^3}{q^{13}}+\frac{t^3}{q^{11}}+\frac{t^3}{q^9}+\frac{t^4}{q^{21}}+\frac{t^4}{q^{19}}+\frac{t^4}{q^{17}}
}
$
\newline 
\\
\hline
$6_{3}$ 
\includegraphics[scale=0.07,angle=0]{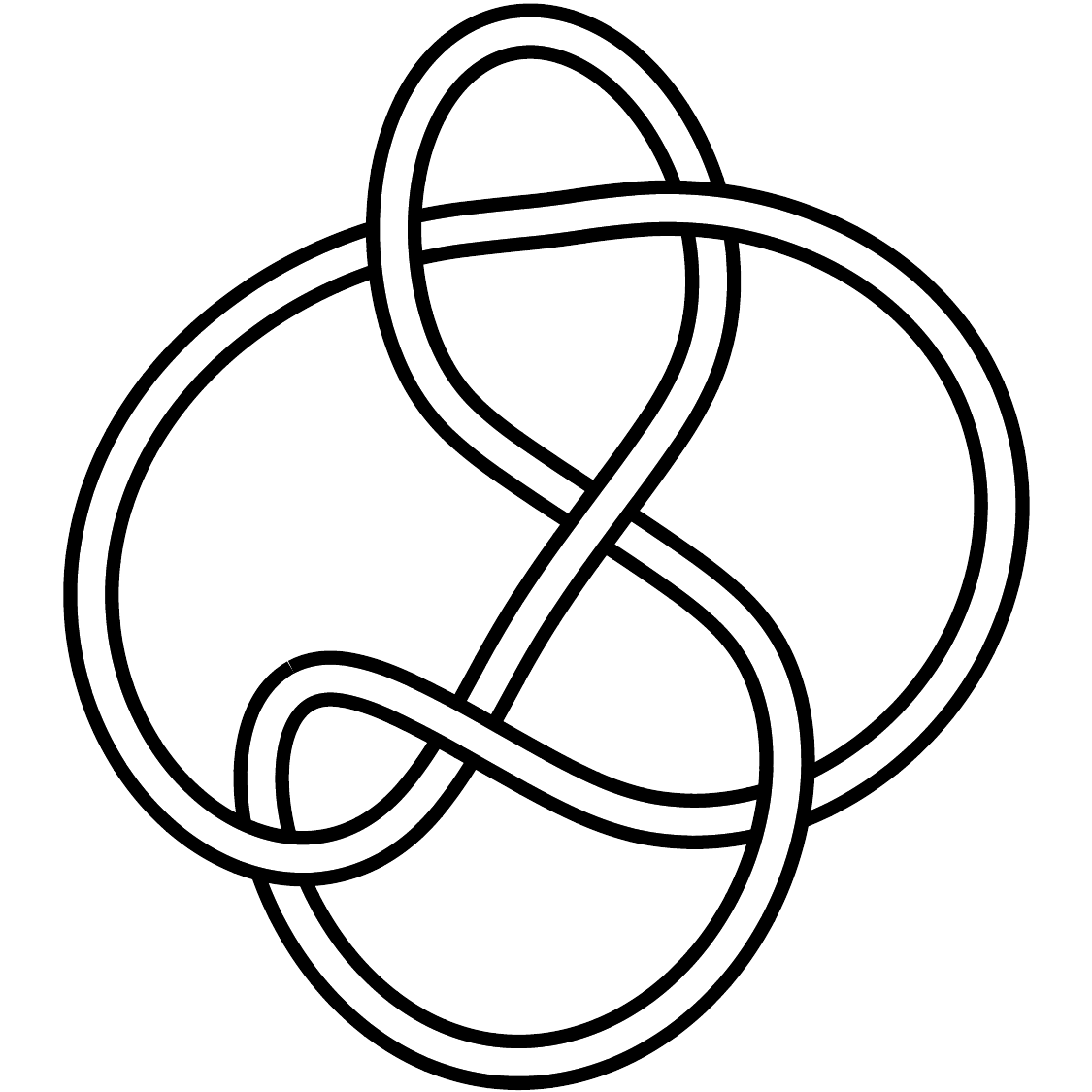} 
& 
\newline 
$
\displaystyle{%
\KR_{2} = \frac{2}{q}+2 q+\frac{q^7}{t^3}+\frac{q^3}{t^2}+\frac{q^5}{t^2}+\frac{q}{t}+\frac{q^3}{t}+\frac{t}{q^3}+\frac{t}{q}+\frac{t^2}{q^5}+\frac{t^2}{q^3}+\frac{t^3}{q^7}
}
$
\newline 
$
\displaystyle{%
\KR_{3} = 3+\frac{2}{q^2}+2 q^2+\frac{q^8}{t^3}+\frac{q^{10}}{t^3}+\frac{q^2}{t^2}+\frac{q^4}{t^2}+\frac{q^6}{t^2}+\frac{q^8}{t^2}+\frac{1}{t}+\frac{q^2}{t}+\frac{q^4}{t}+\frac{q^6}{t}+t+\frac{t}{q^6}+\frac{t}{q^4}+\frac{t}{q^2}+\frac{t^2}{q^8}
}
$
\newline
$
\displaystyle{%
\qquad\qquad +\frac{t^2}{q^6}+\frac{t^2}{q^4}+\frac{t^2}{q^2}+\frac{t^3}{q^{10}}+\frac{t^3}{q^8}
}
$
\newline 
\\
\hline
$6_{1}^2\,\text{{\tiny (v1)}}$ 
\includegraphics[scale=0.07,angle=0]{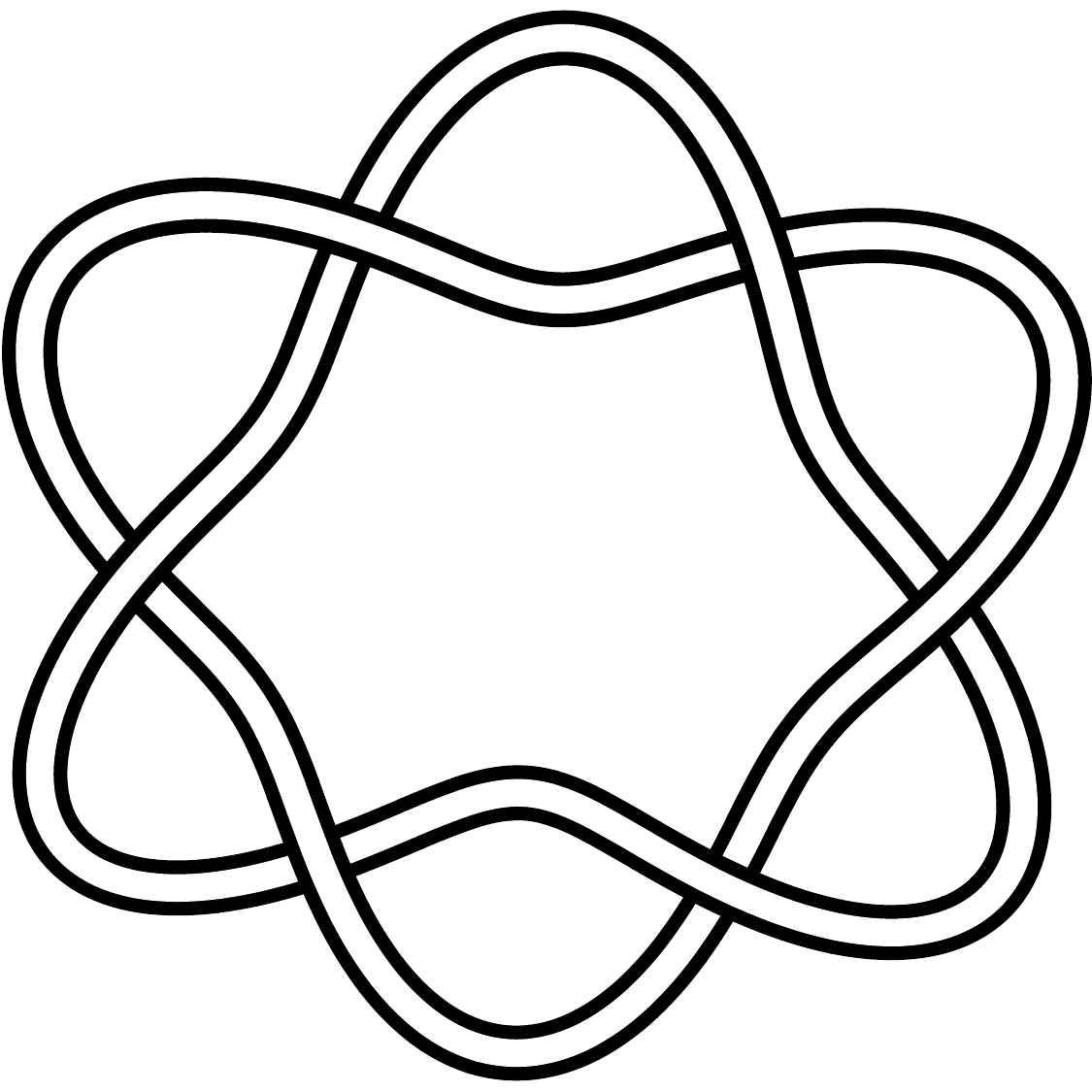} 
& 
\newline 
$
\displaystyle{%
\KR_{2} = \frac{1}{q^6}+\frac{1}{q^4}+\frac{t^2}{q^8}+\frac{t^3}{q^{12}}+\frac{t^4}{q^{12}}+\frac{t^5}{q^{16}}+\frac{t^6}{q^{18}}+\frac{t^6}{q^{16}}
}
$
\newline 
$
\displaystyle{%
\KR_{3} = \frac{1}{q^{12}}+\frac{1}{q^{10}}+\frac{1}{q^8}+\frac{t^2}{q^{14}}+\frac{t^2}{q^{12}}+\frac{t^3}{q^{20}}+\frac{t^3}{q^{18}}+\frac{t^4}{q^{18}}+\frac{t^4}{q^{16}}+\frac{t^5}{q^{24}}+\frac{t^5}{q^{22}}+\frac{t^6}{q^{26}}+\frac{2 t^6}{q^{24}}+\frac{2 t^6}{q^{22}}+\frac{t^6}{q^{20}}
}
$
\newline 
$
\displaystyle{%
\KR_{4} = \frac{1}{q^{18}}+\frac{1}{q^{16}}+\frac{1}{q^{14}}+\frac{1}{q^{12}}+\frac{t^2}{q^{20}}+\frac{t^2}{q^{18}}+\frac{t^2}{q^{16}}+\frac{t^3}{q^{28}}+\frac{t^3}{q^{26}}+\frac{t^3}{q^{24}}+\frac{t^4}{q^{24}}+\frac{t^4}{q^{22}}+\frac{t^4}{q^{20}}+\frac{t^5}{q^{32}}+\frac{t^5}{q^{30}}
}
$
\newline
$
\displaystyle{%
\qquad\qquad +\frac{t^5}{q^{28}}+\frac{t^6}{q^{34}}+\frac{2 t^6}{q^{32}}+\frac{3 t^6}{q^{30}}+\frac{3 t^6}{q^{28}}+\frac{2 t^6}{q^{26}}+\frac{t^6}{q^{24}}
}
$
\newline 
\\
\hline
$6_{1}^2\,\text{{\tiny (v2)}}$ 
\includegraphics[scale=0.07,angle=0]{link6_1_2.pdf} 
& 
\newline
$
\displaystyle{%
\KR_{2} = 1+\frac{1}{q^2}+\frac{t^2}{q^6}+\frac{t^3}{q^6}+\frac{t^4}{q^{10}}+\frac{t^6}{q^{14}}+\frac{t^6}{q^{12}}+\frac{t q^{12}}{q^{14}}
}
$
\newline 
$
\displaystyle{%
\KR_{3} = 1+\frac{1}{q^4}+\frac{1}{q^2}+\frac{t}{q^4}+\frac{t}{q^2}+\frac{t^2}{q^{10}}+\frac{t^2}{q^8}+\frac{t^3}{q^{10}}+\frac{t^3}{q^8}+\frac{t^4}{q^{16}}+\frac{t^4}{q^{14}}+\frac{t^6}{q^{22}}+\frac{2 t^6}{q^{20}}+\frac{2 t^6}{q^{18}}+\frac{t^6}{q^{16}}
}
$
\newline
\\
\hline
$6_{1}^3\,\text{{\tiny (v1)}}$ 
\includegraphics[scale=0.07,angle=0]{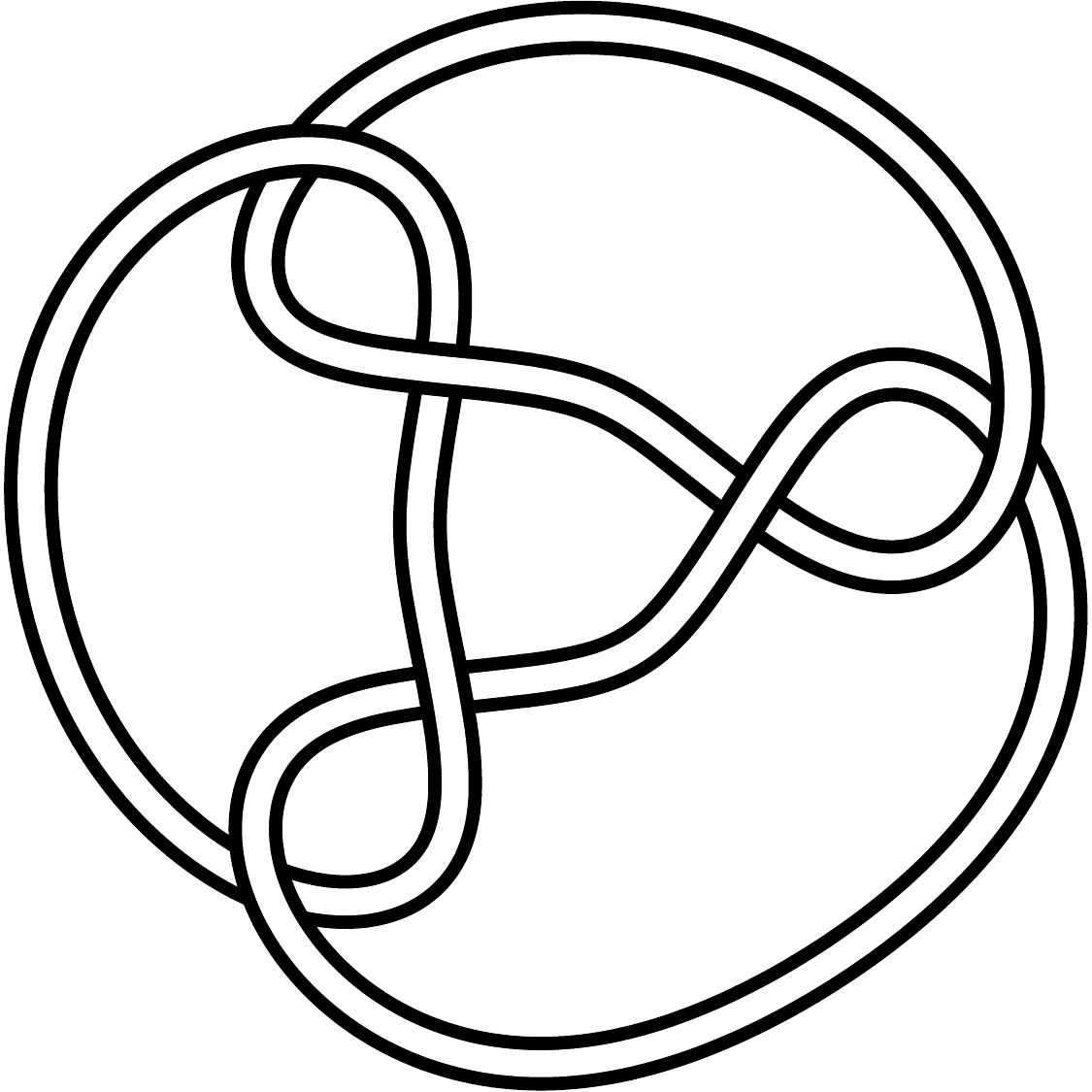} 
& 
\newline 
$
\displaystyle{%
\KR_{2} = \frac{1}{q^3}+\frac{1}{q}+\frac{2 t}{q^3}+\frac{2 t^2}{q^7}+\frac{t^2}{q^5}+\frac{t^3}{q^9}+\frac{3 t^4}{q^{11}}+\frac{3 t^4}{q^9}+\frac{t^5}{q^{11}}+\frac{t^6}{q^{15}}
}
$
\newline 
$
\displaystyle{%
\KR_{3} = \frac{1}{q^6}+\frac{1}{q^4}+\frac{1}{q^2}+\frac{2 t}{q^6}+\frac{2 t}{q^4}+\frac{2 t^2}{q^{12}}+\frac{2 t^2}{q^{10}}+\frac{t^2}{q^8}+\frac{t^2}{q^6}+\frac{t^3}{q^{14}}+\frac{t^3}{q^{12}}+\frac{3 t^4}{q^{18}}+\frac{6 t^4}{q^{16}}+\frac{6 t^4}{q^{14}}+\frac{3 t^4}{q^{12}}+\frac{t^5}{q^{16}}
}
$
\newline
$
\displaystyle{%
\qquad\qquad +\frac{t^5}{q^{14}}+\frac{t^6}{q^{24}}+\frac{3 t^6}{q^{22}}+\frac{3 t^6}{q^{20}}+\frac{t^6}{q^{18}}
}
$
\newline 
\\
\hline
$6_{1}^3\,\text{{\tiny (v2)}}$ 
\includegraphics[scale=0.07,angle=0]{link6_1_3.pdf} 
& 
\newline 
$
\displaystyle{%
\KR_{2} = 
3 q + 3 q^3 + \frac{q^9}{t^4} + \frac{q^11}{t^4} + \frac{2 q^9}{t^3} + \frac{2 q^5}{t^2} + \frac{q^7}{t^2} + \frac{q^3}{t} + q t + \frac{t^2}{q^3}
}
$
\newline 
$
\displaystyle{%
\KR_{3} = 2+5 q^2+5 q^4+3 q^6+\frac{q^{10}}{t^4}+\frac{2 q^{12}}{t^4}+\frac{2 q^{14}}{t^4}+\frac{q^{16}}{t^4}+\frac{2 q^{12}}{t^3}+\frac{2 q^{14}}{t^3}+\frac{q^4}{t^2}+\frac{4 q^6}{t^2}+\frac{4 q^8}{t^2}+\frac{2 q^{10}}{t^2}+\frac{q^{12}}{t^2}
}
$
\newline
$
\displaystyle{%
\qquad\qquad +\frac{q^4}{t}+\frac{q^6}{t}+q^2 t+q^4 t+\frac{t^2}{q^4}+\frac{t^2}{q^2}
}
$
\newline 
\\
\hline
$6_{2}^2$ 
\includegraphics[scale=0.07,angle=0]{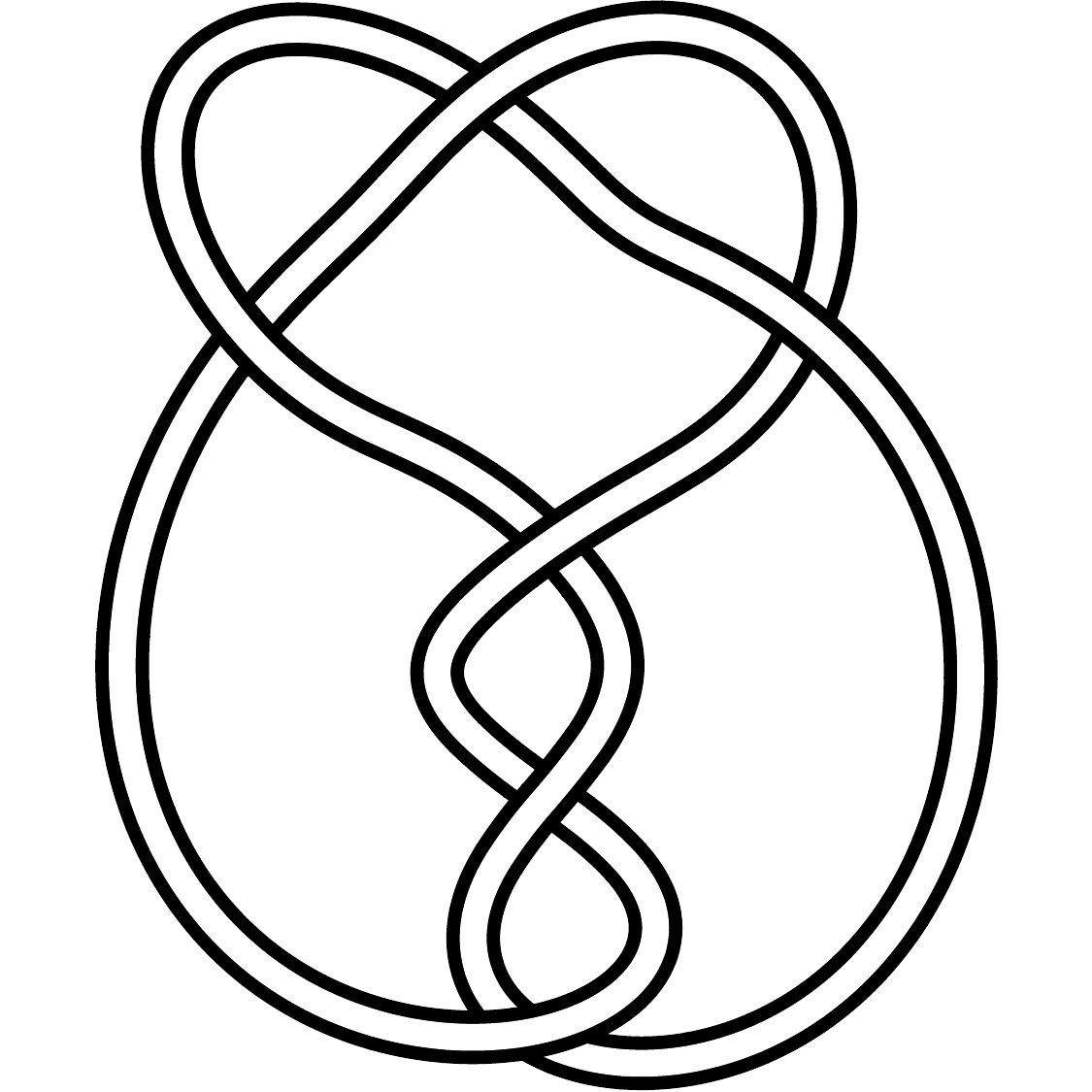} 
& 
\newline 
$
\displaystyle{%
\KR_{2} = q^2+q^4+\frac{q^{14}}{t^6}+\frac{q^{16}}{t^6}+\frac{q^{14}}{t^5}+\frac{q^{10}}{t^4}+\frac{q^{12}}{t^4}+\frac{q^8}{t^3}+\frac{q^{10}}{t^3}+\frac{q^6}{t^2}+\frac{q^8}{t^2}+\frac{q^4}{t}
}
$
\newline 
$
\displaystyle{%
\KR_{3} = q^4+q^6+q^8+\frac{q^{18}}{t^6}+\frac{2 q^{20}}{t^6}+\frac{2 q^{22}}{t^6}+\frac{q^{24}}{t^6}+\frac{q^{20}}{t^5}+\frac{q^{22}}{t^5}+\frac{q^{14}}{t^4}+\frac{2 q^{16}}{t^4}+\frac{q^{18}}{t^4}+\frac{q^{10}}{t^3}+\frac{q^{12}}{t^3}+\frac{q^{14}}{t^3}
}
$
\newline
$
\displaystyle{%
\qquad\qquad +\frac{q^{16}}{t^3}+\frac{q^8}{t^2}+\frac{q^{10}}{t^2}+\frac{q^{12}}{t^2}+\frac{q^{14}}{t^2}+\frac{q^6}{t}+\frac{q^8}{t}
}
$
\newline 
\\
\hline
$6_{2}^3\,\text{{\tiny (v1)}}$ 
\includegraphics[scale=0.07,angle=0]{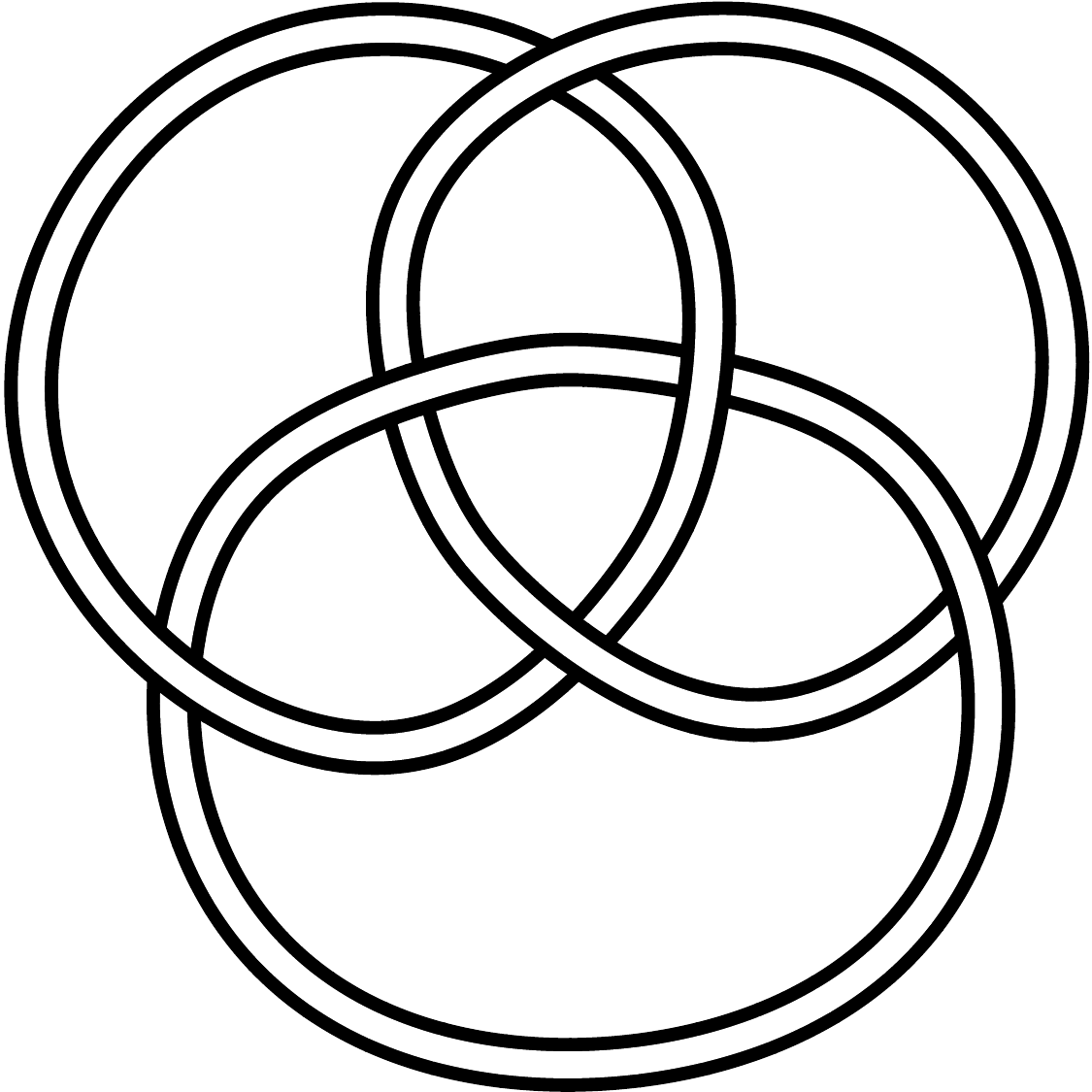} 
& 
\newline
$
\displaystyle{%
\KR_{2} = 
\frac{2}{q}+2 q+\frac{q^3}{t^2}+\frac{q^5}{t^2}+\frac{t^2}{q^5}+\frac{t^2}{q^3}
}
$
\newline 
$
\displaystyle{%
\KR_{3} = 
5+\frac{1}{q^4}+\frac{4}{q^2}+4 q^2+q^4+\frac{q^2}{t^2}+\frac{2 q^4}{t^2}+\frac{2 q^6}{t^2}+\frac{q^8}{t^2}+\frac{t^2}{q^8}+\frac{2 t^2}{q^6}+\frac{2 t^2}{q^4}+\frac{t^2}{q^2}
}
$
\newline
\\
\hline
$6_{2}^3\,\text{{\tiny (v2)}}$ 
\includegraphics[scale=0.07,angle=0]{link6_2_3.pdf} 
& 
\newline 
$
\displaystyle{%
\KR_{2} = \frac{4}{q}+4 q+\frac{q^7}{t^3}+\frac{q^3}{t^2}+\frac{2 q^5}{t^2}+\frac{2 q}{t}+\frac{2 t}{q}+\frac{2 t^2}{q^5}+\frac{t^2}{q^3}+\frac{t^3}{q^7}
}
$
\newline 
$
\displaystyle{%
\KR_{3} = 9+\frac{2}{q^4}+\frac{7}{q^2}+7 q^2+2 q^4+\frac{q^8}{t^3}+\frac{q^{10}}{t^3}+\frac{q^2}{t^2}+\frac{q^4}{t^2}+\frac{2 q^6}{t^2}+\frac{2 q^8}{t^2}+\frac{2}{t}+\frac{2 q^2}{t}+2 t+\frac{2 t}{q^2}+\frac{2 t^2}{q^8}+\frac{2 t^2}{q^6}
}
$
\newline
$
\displaystyle{%
\qquad\qquad +\frac{t^2}{q^4}+\frac{t^2}{q^2}+\frac{t^3}{q^{10}}+\frac{t^3}{q^8}
}
$
\newline 
\\
\hline
$6_{3}^2\,\text{{\tiny (v1)}}$ 
\includegraphics[scale=0.07,angle=0]{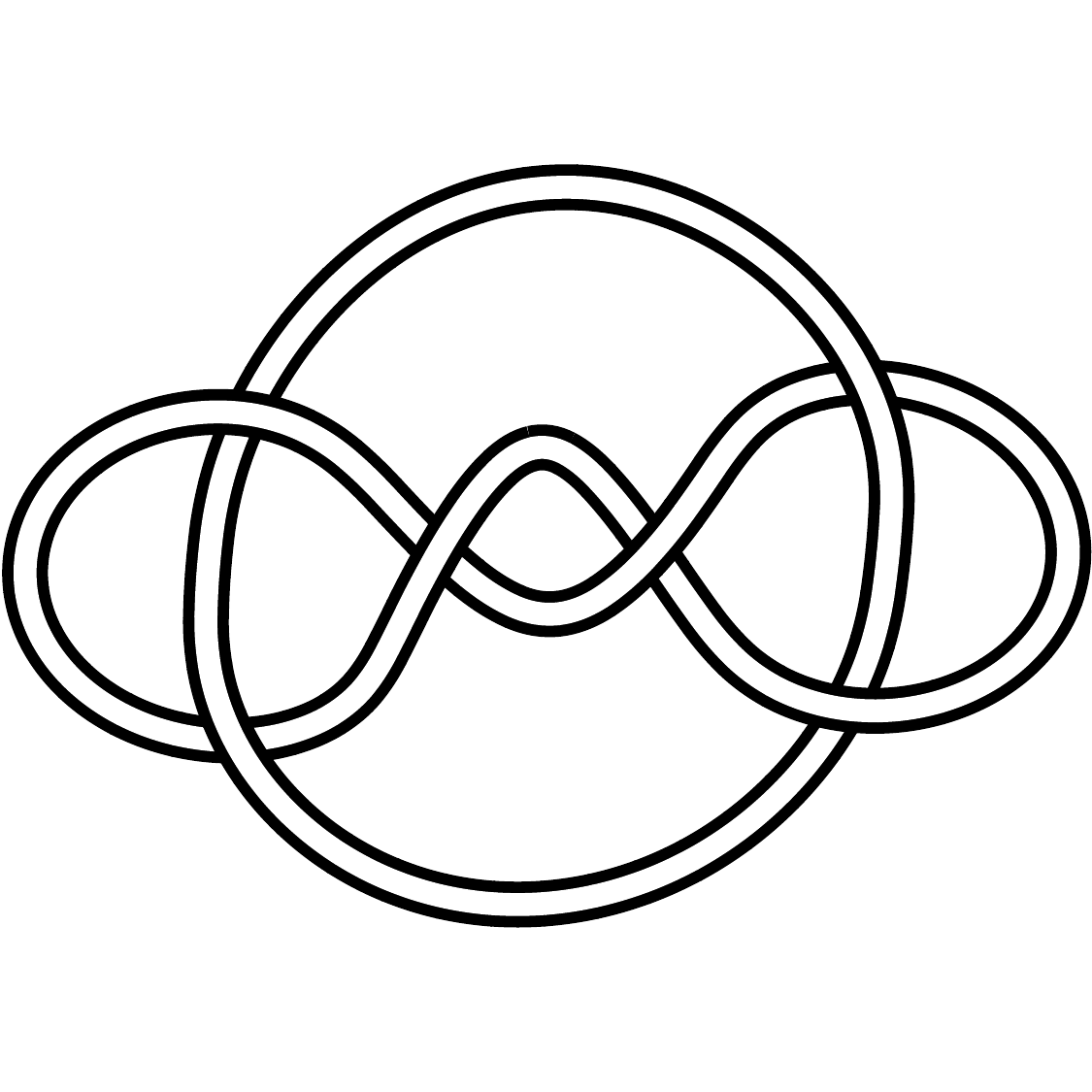} 
& 
\newline
$
\displaystyle{%
\KR_{2} = 
\frac{t^4}{q^{10}}+\frac{t^4}{q^8}+\frac{t^3}{q^8}+\frac{2 t^2}{q^6}+\frac{q^4}{t^2}+\frac{t^2}{q^4}+\frac{q^2}{t}+\frac{2 t}{q^2}+\frac{2}{q^2}+\frac{1}{t}+1
}
$
\newline 
$
\displaystyle{%
\KR_{3} = 
\frac{t^4}{q^{16}}+\frac{2 t^4}{q^{14}}+\frac{2 t^4}{q^{12}}+\frac{t^3}{q^{12}}+\frac{t^4}{q^{10}}+\frac{t^3}{q^{10}}+\frac{2 t^2}{q^{10}}+\frac{2 t^2}{q^8}+\frac{q^6}{t^2}+\frac{t^2}{q^6}+\frac{q^4}{t^2}+\frac{t^2}{q^4}+\frac{q^4}{t}+\frac{2 t}{q^4}+\frac{2}{q^4}+\frac{q^2}{t}
}
$
\newline
$
\displaystyle{%
\qquad\qquad +\frac{2 t}{q^2}+\frac{1}{q^2 t}+\frac{2}{q^2}+\frac{1}{t}+1
}
$
\newline
\\
\hline
$6_{3}^2\,\text{{\tiny (v2)}}$ 
\includegraphics[scale=0.07,angle=0]{link6_3_2.pdf} 
& 
\newline 
$
\displaystyle{%
\KR_{2} = q^2+q^4+\frac{q^{16}}{t^6}+\frac{q^{12}}{t^5}+\frac{q^{14}}{t^5}+\frac{2 q^{10}}{t^4}+\frac{q^{12}}{t^4}+\frac{2 q^{10}}{t^3}+\frac{2 q^6}{t^2}+\frac{q^8}{t^2}+\frac{q^4}{t}
}
$
\newline 
$
\displaystyle{%
\KR_{3} = q^4+q^6+q^8+\frac{q^{22}}{t^6}+\frac{q^{24}}{t^6}+\frac{q^{16}}{t^5}+\frac{q^{18}}{t^5}+\frac{q^{20}}{t^5}+\frac{q^{22}}{t^5}+\frac{q^{12}}{t^4}+\frac{3 q^{14}}{t^4}+\frac{3 q^{16}}{t^4}+\frac{q^{18}}{t^4}+\frac{2 q^{14}}{t^3}+\frac{2 q^{16}}{t^3}
}
$
\newline
$
\displaystyle{%
\qquad\qquad +\frac{2 q^8}{t^2}+\frac{2 q^{10}}{t^2}+\frac{q^{12}}{t^2}+\frac{q^{14}}{t^2}+\frac{q^6}{t}+\frac{q^8}{t}
}
$
\newline 
\\
\hline
$6_{3}^3\,\text{{\tiny (v1)}}$
\includegraphics[scale=0.07,angle=0]{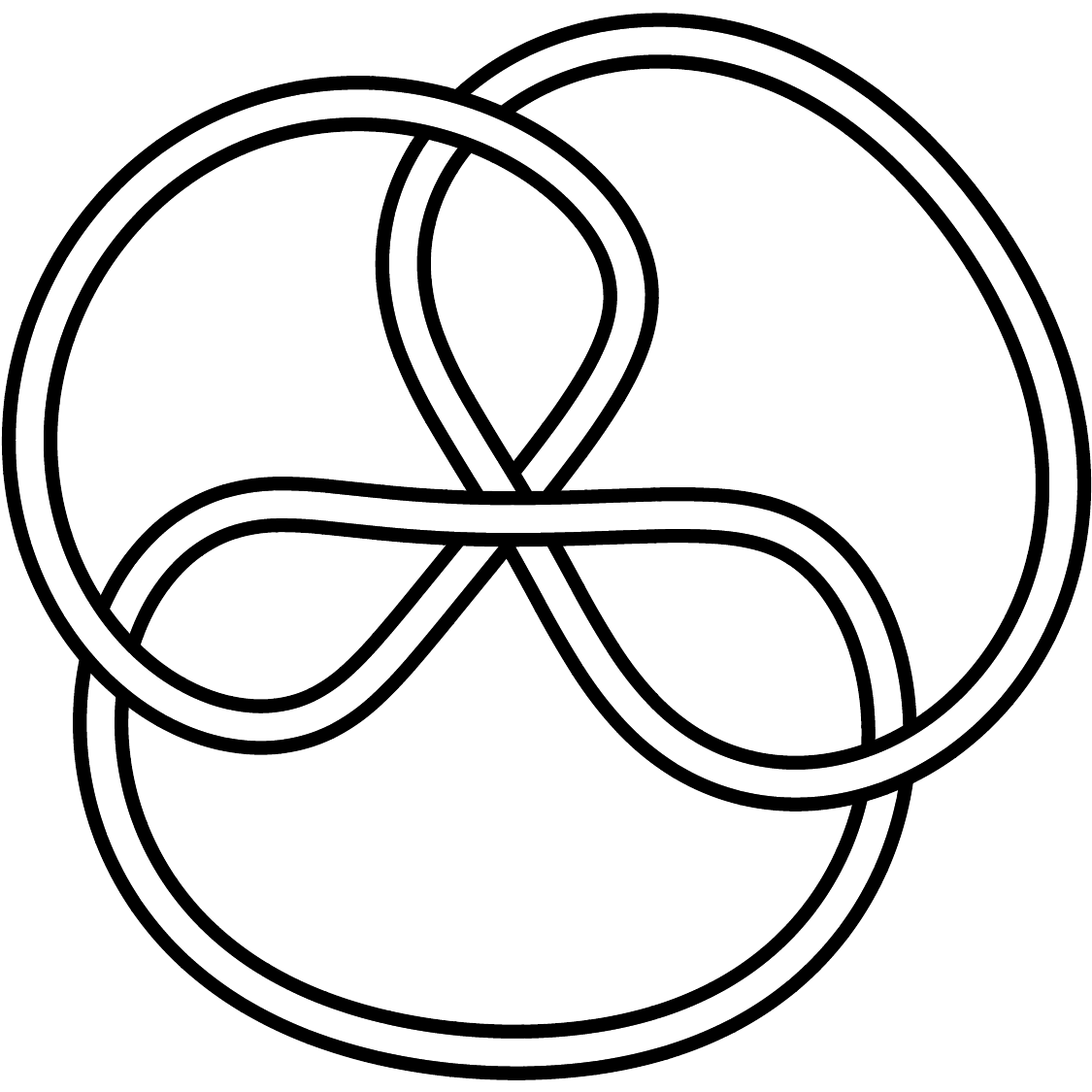} 
& 
\newline 
$
\displaystyle{%
\KR_{2} = \frac{1}{q^5}+\frac{1}{q^3}+\frac{t^2}{q^7}+\frac{t^3}{q^{11}}+\frac{2 t^4}{q^{13}}+\frac{3 t^4}{q^{11}}+\frac{t^4}{q^9}
}
$
\newline 
$
\displaystyle{%
\KR_{3} = \frac{1}{q^{10}}+\frac{1}{q^8}+\frac{1}{q^6}+\frac{t^2}{q^{12}}+\frac{t^2}{q^{10}}+\frac{t^3}{q^{18}}+\frac{t^3}{q^{16}}+\frac{2 t^4}{q^{20}}+\frac{5 t^4}{q^{18}}+\frac{6 t^4}{q^{16}}+\frac{4 t^4}{q^{14}}+\frac{t^4}{q^{12}}+\frac{t^6}{q^{24}}+\frac{2 t^6}{q^{22}}+\frac{2 t^6}{q^{20}}
}
$
\newline
$
\displaystyle{%
\qquad\qquad +\frac{t^6}{q^{18}}
}
$
\newline 
\\
\hline
$6_{3}^3\,\text{{\tiny (v2)}}$
\includegraphics[scale=0.07,angle=0]{link6_3_3.pdf} 
&
\newline 
$
\displaystyle{%
\KR_{2} = \frac{2}{q}+3 q+q^3+\frac{q^7}{t^4}+\frac{q^9}{t^4}+\frac{q^5}{t^2}+\frac{q}{t}
}
$
\newline 
$
\displaystyle{%
\KR_{3} = 4+\frac{2}{q^2}+5 q^2+3 q^4+q^6+\frac{q^8}{t^4}+\frac{2 q^{10}}{t^4}+\frac{2 q^{12}}{t^4}+\frac{q^{14}}{t^4}+\frac{q^4}{t^2}+\frac{3 q^6}{t^2}+\frac{3 q^8}{t^2}+\frac{q^{10}}{t^2}+\frac{1}{t}+\frac{q^2}{t}
}
$
\newline
\\
\hline
\end{longtable}
}

{\footnotesize 
\begin{longtable}{p{0.07\textwidth}|p{0.87\textwidth}} 
\caption{Reduced Khovanov-Rozansky $\sln$ invariants for prime links up to 6 crossings} \\
\label{reducedKR}
link $L$ & $\overline{\KR}_{N}(L)$ \\
\hline\hline
\endfirsthead
link $L$ & $\overline{\KR}_{N}(L)$ \\
\hline\hline
\endhead
\hline\hline
\endfoot
$2_{1}^2\,\text{{\tiny (v1\&v2)}}$ 
\includegraphics[scale=0.07,angle=0]{link2_1_2.pdf} 
& 
\newline
$
\displaystyle{%
q^{1-N} + \frac{q^{-N-1}-q^{1-3 N}}{q^2-1} t^2
}
$
\newline\newline\newline
Checked up to $N=18$. Extends result of~\cite{r0508510} to $N=3$ and $N=4$. 
\\
\hline
$3_{1}$ 
\includegraphics[scale=0.07,angle=0]{knot3_1.pdf} 
&
\newline 
$
\displaystyle{%
q^{4N} t^{-3} + q^{2 + 2N} t^{-2} + q^{-2 + 2N}
}
$
\newline\newline\newline
Checked up to $N=16$. Agrees with results of~\cite{r0508510, r0607544}. 
\\
\hline
$4_{1}$ 
\includegraphics[scale=0.07,angle=0]{knot4_1.pdf} 
& 
\newline
$
\displaystyle{%
1 + q^{2 N} t^{-2} + q^2 t^{-1} + t q^{-2} + q^{-2 N} t^2
}
$
\newline\newline\newline
Checked up to $N=7$. Agrees with results of~\cite{r0508510, r0607544}. 
\\
\hline
$4_{1}^2\,\text{{\tiny (v1)}}$ 
\includegraphics[scale=0.07,angle=0]{link4_1_2.pdf} 
& 
\newline
$
\displaystyle{%
q^{1-N}+q^{-N-1} t+q^{1-3 N} t^2-\frac{q^{1-5N} t^4}{q^2-1}+\frac{q^{-3 N-1} t^4}{q^2-1}
}
$
\newline\newline\newline
Checked up to $N=6$. Extends result of~\cite{r0508510} to $N=3$ and $N=4$. 
\\
\hline
$4_{1}^2\,\text{{\tiny (v2)}}$ 
\includegraphics[scale=0.07,angle=0]{link4_1_2.pdf} 
& 
\newline
$
\displaystyle{%
q^{3 N-3}-\frac{q^{3N+5}}{\left(q^2-1\right) t^4}+\frac{q^{5 N+3}}{\left(q^2-1\right) t^4}+\frac{q^{5 N-1}}{t^3}+\frac{q^{3 N+1}}{t^2}
}
$
\newline\newline\newline
Checked up to $N=8$. Extends result of~\cite{r0508510} to $N=3$ and $N=4$. 
\\
\hline
$5_{1}$ 
\includegraphics[scale=0.07,angle=0]{knot5_1.pdf} 
& 
\newline
$
\displaystyle{%
q^{4 - 4 N} + q^{-4 N} t^2 + q^{2 - 6 N} t^3 + q^{-4 - 4 N} t^4 +  q^{-2 - 6 N} t^5
}
$
\newline\newline\newline
Checked up to $N=7$. Agrees with results of~\cite{r0508510, r0607544}. 
\\
\hline
$5_{2}$ 
\includegraphics[scale=0.07,angle=0]{knot5_2.pdf} 
& 
\newline
$
\displaystyle{%
q^{2 - 2 N} + q^{-2 N} t + q^{2 - 4 N} t^2 + q^{-2 - 2 N} t^2 + q^{-4 N} t^3 + q^{-2 - 4 N} t^4 + q^{-6 N} t^5
}
$
\newline\newline\newline
Checked up to $N=4$. Agrees with results of~\cite{r0508510, r0607544}. 
\\
\hline
$5_{1}^2\,\text{{\tiny (v1\&v2)}}$ 
\includegraphics[scale=0.07,angle=0]{link5_1_2.pdf} 
& 
\newline
$
\displaystyle{%
q^{1-N}+\frac{q^{N-1} - q^{1-N}}{q^2-1}+\frac{q^{N+1}}{t^2}+\frac{q^{3-N}}{t}+q^{-N-1} t+q^{1-3N} t^2+q^{-N-3} t^2+q^{-3 N-1} t^3
}
$
\newline\newline\newline
Checked up to $N=4$. Extends result of~\cite{r0508510} to $N=3$ and $N=4$. 
\\
\hline
$6_{1}$ 
\includegraphics[scale=0.07,angle=0]{knot6_1.pdf} 
& 
\newline
$
\displaystyle{%
2+q^{2 N}t^{-2}+q^2 t^{-1}+ t q^{-2} + q^{2-2N} t+q^{-2 N} t^2+q^{-2N-2} t^3+q^{-4 N} t^4
}
$
\newline\newline\newline
Checked up to $N=2$. Agrees with results of~\cite{r0508510, r0607544}. 
\\
\hline
$6_{2}$ 
\includegraphics[scale=0.07,angle=0]{knot6_2.pdf} 
& 
\newline
$
\displaystyle{%
q^{-2} + q^{2-2 N} + q^2t^{-2} + q^{4-2 N} t^{-1} + 2 q^{-2 N} t+q^{2-4 N} t^2+q^{-2N-2} t^2+q^{-4 N} t^3+q^{-2N-4} t^3+q^{-4N-2} t^4
}
$
\newline\newline\newline
Checked up to $N=4$. Agrees with results of~\cite{r0508510, r0607544}. 
\\
\hline
$6_{3}$ 
\includegraphics[scale=0.07,angle=0]{knot6_3.pdf} 
& 
\newline
$
\displaystyle{%
3+q^{2N+2} t^{-3} + q^4t^{-2}+ q^{2 N}t^{-2}+ q^2 t^{-1}+ q^{2 N-2}t^{-1}+ t q^{-2}+q^{2-2 N} t+ t^2q^{-4}+q^{-2 N} t^2+q^{-2N-2} t^3
}
$
\newline\newline\newline
Checked up to $N=3$. Agrees with results of~\cite{r0508510, r0607544}. 
\\
\hline
$6_{1}^2\,\text{{\tiny (v1)}}$ 
\includegraphics[scale=0.07,angle=0]{link6_1_2.pdf} 
& 
\newline
$
\displaystyle{%
q^{5 N-5} \left(1+\frac{q^{2N+10}-q^{12}}{\left(q^2-1\right) t^6}+\frac{q^{2 N+6}}{t^5}+\frac{q^8}{t^4}+\frac{q^{2 N+2}}{t^3}+\frac{q^4}{t^2}\right)
}
$
\newline\newline\newline
Checked up to $N=5$. Extends result of~\cite{r0508510} to $N=3$ and $N=4$. 
\\
\hline
$6_{1}^2\,\text{{\tiny (v2)}}$ 
\includegraphics[scale=0.07,angle=0]{link6_1_2.pdf} 
& 
\newline
$
\displaystyle{%
q^{-7 N-1} \left(q^{6 N+2}+q^{6 N} t+q^{4 N+2} t^2+q^{4 N} t^3+q^{2 N+2} t^4+\frac{q^{2 N} -q^2 }{q^2-1} t^6\right)
}
$
\newline\newline\newline
Checked up to $N=3$. Extends result of~\cite{r0508510} to $N=3$.
\\
\hline
$6_{1}^3\,\text{{\tiny (v1)}}$ 
\includegraphics[scale=0.07,angle=0]{link6_1_3.pdf} 
& 
\newline
$
\displaystyle{%
\overline{\KR}_{2} = \frac{1}{q^2}+\frac{2 t}{q^4}+\frac{3 t^2}{q^6}+\frac{t^3}{q^8}+\frac{3 t^4}{q^{10}}+\frac{t^5}{q^{12}}+\frac{t^6}{q^{14}}
}
$
\newline 
$
\displaystyle{%
\overline{\KR}_{3} = \frac{1}{q^4}+\frac{2 t}{q^6}+\frac{2 t^2}{q^{10}}+\frac{t^2}{q^8}+\frac{t^3}{q^{12}}+\frac{3 t^4}{q^{16}}+\frac{3 t^4}{q^{14}}+\frac{t^5}{q^{16}}+\frac{t^6}{q^{22}}+\frac{2 t^6}{q^{20}}
}
$
\newline
\\
\hline
$6_{1}^3\,\text{{\tiny (v2)}}$ 
\includegraphics[scale=0.07,angle=0]{link6_1_3.pdf} 
& 
\newline
$
\displaystyle{%
\overline{\KR}_{2} = 3 q^2+\frac{q^{10}}{t^4}+\frac{2 q^8}{t^3}+\frac{3 q^6}{t^2}+\frac{q^4}{t}+t+\frac{t^2}{q^2}
}
$
\newline 
$
\displaystyle{%
\overline{\KR}_{3} = 2 q^2+3 q^4+\frac{q^{12}}{t^4}+\frac{q^{14}}{t^4}+\frac{2 q^{12}}{t^3}+\frac{q^6}{t^2}+\frac{3 q^8}{t^2}+\frac{q^{10}}{t^2}+\frac{q^6}{t}+q^2 t+\frac{t^2}{q^2}
}
$
\newline
\\
\hline
$6_{2}^2$ 
\includegraphics[scale=0.07,angle=0]{link6_2_2.pdf} 
& 
\newline
$
\displaystyle{%
q^{3 N-3}+\frac{q^{7 N+3}-q^{5 N+5}}{\left(q^2-1\right) t^6}+\frac{q^{7 N-1}}{t^5}+\frac{2 q^{5 N+1}}{t^4}+\frac{q^{3 N+3}}{t^3}+\frac{q^{5 N-1}}{t^3}+\frac{q^{3 N+1}}{t^2}+\frac{q^{5 N-3}}{t^2}+\frac{q^{3 N-1}}{t}
}
$
\newline\newline\newline
Checked up to $N=4$. Extends result of~\cite{r0508510} to $N=3$ and $N=4$.
\\
\hline
$6_{2}^3\,\text{{\tiny (v1)}}$ 
\includegraphics[scale=0.07,angle=0]{link6_2_3.pdf} 
& 
\newline
$
\displaystyle{%
\overline{\KR}_{2} = 2+\frac{q^4}{t^2}+\frac{t^2}{q^4}
}
$
\newline 
$
\displaystyle{%
\overline{\KR}_{3} = 
3+\frac{1}{q^2}+q^2+\frac{q^4}{t^2}+\frac{q^6}{t^2}+\frac{t^2}{q^6}+\frac{t^2}{q^4}
}
$
\newline
\\
\hline
$6_{2}^3\,\text{{\tiny (v2)}}$ 
\includegraphics[scale=0.07,angle=0]{link6_2_3.pdf} 
& 
\newline 
$
\displaystyle{%
\overline{\KR}_{2} = 4+\frac{q^6}{t^3}+\frac{3 q^4}{t^2}+\frac{2 q^2}{t}+\frac{2 t}{q^2}+\frac{3 t^2}{q^4}+\frac{t^3}{q^6}
}
$
\newline 
$
\displaystyle{%
\overline{\KR}_{3} = 5+\frac{2}{q^2}+2 q^2+\frac{q^8}{t^3}+\frac{q^4}{t^2}+\frac{2 q^6}{t^2}+\frac{2 q^2}{t}+\frac{2 t}{q^2}+\frac{2 t^2}{q^6}+\frac{t^2}{q^4}+\frac{t^3}{q^8}
}
$
\newline
$
\displaystyle{%
\overline{\KR}_{4} = 6+\frac{2}{q^4}+\frac{3}{q^2}+3 q^2+2 q^4+\frac{q^{10}}{t^3}+\frac{q^4}{t^2}+\frac{2 q^8}{t^2}+\frac{2 q^2}{t}+\frac{2 t}{q^2}+\frac{2 t^2}{q^8}+\frac{t^2}{q^4}+\frac{t^3}{q^{10}}
}
$
\newline 
\\
\hline
$6_{3}^2\,\text{{\tiny (v1)}}$ 
\includegraphics[scale=0.07,angle=0]{link6_3_2.pdf} 
& 
\newline
$
\displaystyle{%
\overline{\KR}_{2} = 
\frac{t^4}{q^9}+\frac{t^3}{q^7}+\frac{3 t^2}{q^5}+\frac{q^3}{t^2}+\frac{2 t}{q^3}+\frac{2 q}{t}+\frac{2}{q}
}
$
\newline 
$
\displaystyle{%
\overline{\KR}_{3} = 
\frac{t^4}{q^{14}}+\frac{t^4}{q^{12}}+\frac{t^3}{q^{10}}+\frac{2 t^2}{q^8}+\frac{t^2}{q^6}+\frac{q^4}{t^2}+\frac{2 t}{q^4}+\frac{q^2}{t}+\frac{2}{q^2}+\frac{1}{t}
}
$
\newline\newline
Extends result of~\cite{r0508510} to $N=3$. 
\\
\hline
$6_{3}^2\,\text{{\tiny (v2)}}$ 
\includegraphics[scale=0.07,angle=0]{link6_3_2.pdf} 
& 
\newline
$
\displaystyle{%
q^{3-3 N}+q^{1-3 N} t+q^{3-5 N} t^2+2 q^{-3 N-1} t^2+2 q^{1-5 N} t^3+q^{-5 N-1} t^4-\frac{q^{-5 N-1} t^4}{q^2-1}+\frac{q^{-3 N-3} t^4}{q^2-1}+q^{1-7 N} t^5
}
$
\newline
$
\displaystyle{%
\qquad\qquad +q^{-5 N-3} t^5+q^{-7 N-1} t^6
}
$
\newline\newline
Checked up to $N=3$. Extends result of~\cite{r0508510} to $N=3$. 
\\
\hline
$6_{3}^3\,\text{{\tiny (v1)}}$
\includegraphics[scale=0.07,angle=0]{link6_3_3.pdf} 
& 
\newline
$
\displaystyle{%
\overline{\KR}_{2} = \frac{1}{q^4}+\frac{t^2}{q^8}+\frac{t^3}{q^{10}}+\frac{2 t^4}{q^{12}}+\frac{t^4}{q^{10}}
}
$
\newline 
$
\displaystyle{%
\overline{\KR}_{3} = \frac{1}{q^8}+\frac{t^2}{q^{12}}+\frac{t^3}{q^{16}}+\frac{2 t^4}{q^{18}}+\frac{3 t^4}{q^{16}}+\frac{t^4}{q^{14}}+\frac{t^6}{q^{22}}+\frac{t^6}{q^{20}}
}
$
\newline
\\
\hline
$6_{3}^3\,\text{{\tiny (v2)}}$
\includegraphics[scale=0.07,angle=0]{link6_3_3.pdf} 
& 
\newline
$
\displaystyle{%
\overline{\KR}_{2} = 2+q^2+\frac{q^8}{t^4}+\frac{q^4}{t^2}+\frac{q^2}{t}
}
$
\newline 
$
\displaystyle{%
\overline{\KR}_{3} = 2+2 q^2+q^4+\frac{q^{10}}{t^4}+\frac{q^{12}}{t^4}+\frac{2 q^6}{t^2}+\frac{q^8}{t^2}+\frac{q^2}{t}
}
$
\newline
\\
\hline
\end{longtable}
}

\appendix

\section{The construction}\label{section:constrkr}

The construction as originally given in \cite{kr0401268} has many beautiful properties but can strike the newcomer as looking unmotivated. Why these matrix factorisations, and why these maps? Things becomes clearer as one reads on \cite{k0510265,RouquierMexico} but in order to make this apparent from the outset we are going to take as our starting point a simple adjunction involving symmetric polynomials. 

\subsection{Symmetric polynomials and adjunction}

Consider the inclusion
\[
\varphi: A = \QQ[x_1 + x_2, x_1 x_2] \lto \QQ[x_1,x_2] = S
\]
of symmetric polynomials in all polynomials. We give the variables $x_i$ degree $2$. As with any degree zero ring morphism, $\varphi$ determines adjunctions between the categories of graded modules and degree zero morphisms: 
\[
\xymatrix@C+3pc{
\Grmodd(A) \ar@/^1pc/[r]^-{\varphi^*} \ar@/_1pc/[r]_-{\varphi^!} & \Grmodd(S) \ar[l]^-{\varphi_*}
}, \qquad
\xymatrix{
\varphi^* \ar@{-|}[r] & \varphi_* \ar@{-|}[r] & \varphi^!\,.
}
\]
Here $\varphi_*$ denotes restriction of scalars, $\varphi^* = S \otimes_A (-)$ is extension of scalars and $\varphi^! = \Hom_A(S, -)$. Any polynomial $f \in S$ can be written uniquely as $f_1 + f_2(x_1-x_2)$ for symmetric polynomials $f_i$, so $S$ is a free $A$-module on the basis $\{1, x_1 - x_2\}$ and we write $1^*, (x_1-x_2)^*$ for the corresponding dual basis of $\Hom_A(S,A)$ as an $A$-module. The important point is that there is an isomorphism of graded $S$-modules (here $\{ m \}$ denotes shifting the grading \textsl{up} by $m$)
\[
\psi: S \lto \Hom_A(S,A)\{2\} \, , \qquad \psi(1) = 2(x_1-x_2)^* \, , \qquad \psi(x_1-x_2) = 2 \cdot 1^* \, .
\]
Obviously we are free to choose the normalisation: the factors of $2$ will be justified later. From this we deduce for any graded $A$-module $N$ a natural isomorphism of graded $S$-modules
\[
\varphi^!(N) = \Hom_A(S,N) \cong \Hom_A(S,A) \otimes_A N \stackrel{\psi}{\cong} S \otimes_A N \{-2\} = \varphi^*(N)\{-2\} \, ,
\]
so that, up to a grading shift, the left and right adjoints of $\varphi_*$ are naturally isomorphic. One way to look at this is that the map $\psi$ has provided us with an adjunction $\xymatrix@1{\varphi_*\, \ar@{-|}[r] & \,\varphi^*\{-2\}}$. We say that the functors $\varphi_*, \varphi^*$ are \textsl{biadjoint}. By a process of stabilisation (discussed in the next subsection) this biadjunction will give rise to a biadjunction on the level of matrix factorisations, which is the starting point for the Khovanov-Rozansky construction.

The units and counits of these adjunctions will give rise to the morphisms $\chi_0, \chi_1$ between $\Xbul, \Xcirc$ mentioned above. The units are described for a graded $A$-module $N$ and graded $S$-module $M$ by
\begin{align*}
\eta: N & \lto \varphi_* \varphi^*(N) \, , \qquad & n & \longmapsto 1 \otimes n \, ,\\
\eta': M & \lto \varphi^* \varphi_*(M)\{-2\} \, , \qquad & m & \longmapsto \frac{1}{2}(x_1-x_2) \otimes m + 1 \otimes \frac{1}{2}(x_1 - x_2) \cdot m
\end{align*}
and the counits are given by
\begin{align*}
\varepsilon: \varphi^* \varphi_*(M) & \lto M \, , \qquad & b \otimes m & \longmapsto b \cdot m \, ,\\
\varepsilon': \varphi_*\varphi^*(N)\{-2\} & \lto N \, , \qquad & b \otimes n & \longmapsto \psi(b)(1) \cdot n \, .
\end{align*}
To make the transition from modules to matrix factorisations, we present the biadjunction in terms of graded bimodules and bimodule maps. It is clear that $\varphi^*$ is given by tensoring with the $S$-$A$-bimodule ${}_S S_A$, and likewise $\varphi_*$ is represented by the $A$-$S$-bimodule ${}_A S_S$. Hence the composite $\varphi^* \varphi_*$ is represented by ${}_S S \otimes_A S {}_S$ and $\varphi_* \varphi^*$ by ${}_A S {}_A$. Let us abuse notation and reuse $\eta, \eta', \varepsilon, \varepsilon'$ for the morphisms between these bimodules which represent the units
\begin{align*}
\eta: A & \lto S \, , \qquad & a & \longmapsto \varphi(a)\,,\\
\eta': S & \lto S \otimes_A S\{-2\} \, , \qquad & b &\longmapsto \frac{1}{2}(x_1 - x_2) \otimes b + 1 \otimes \frac{1}{2}(x_1 - x_2)b
\end{align*}
and counits
\begin{align*}
& \varepsilon: S \otimes_A S \lto S \, , \qquad & b \otimes b' & \longmapsto bb'\,,\\
& \varepsilon': S\{-2\} \lto A \, , \qquad & b & \longmapsto \psi(b)(1) \, . 
\end{align*}

\begin{lemma}\label{lemma:bimodules_as_algebras} Viewing a $\QQ[x_1,x_2,y_1,y_2]$-module as an $S$-bimodule by identifying the action of $y_1, y_2$ with the right action of $x_1, x_2$, respectively, there are isomorphisms of graded $S$-bimodules
\begin{align*}
S \otimes_A S & \cong \QQ[x_1,x_2,y_1,y_2]/(y_1 + y_2 - x_1 - x_2, y_1 y_2 - x_1 x_2) \, ,\\
S & \cong \QQ[x_1,x_2,y_1,y_2]/(y_1 - x_1, y_2 - x_2)\,.
\end{align*}
\end{lemma}

Moreover it is clear that these isomorphisms can be chosen to make the diagrams
\begin{equation}
\label{eq:bimodalg1}
\xymatrix{
S \otimes_A S \ar[d]_{\varepsilon} \ar[r]^-{\cong} & \QQ[x_1,x_2,y_1,y_2]/(y_1 + y_2 - x_1 - x_2, y_1 y_2 - x_1 x_2) \ar[d]^{\can}\\
S \ar[r]_-{\cong} & \QQ[x_1,x_2,y_1,y_2]/(y_1 - x_1, y_2 - x_2)
}
\end{equation}
and
\begin{equation}
\label{eq:bimodalg2}
\xymatrix{
S \ar[r]^-{\cong}\ar[d]_{\eta'} & \QQ[x_1,x_2,y_1,y_2]/(y_1 - x_1, y_2 - x_2) \ar[d]^{\frac{1}{2}(x_1 + y_1 - x_2 - y_2)}\\
S \otimes_A S\{-2\} \ar[r]_-{\cong} & \QQ[x_1,x_2,y_1,y_2]/(y_1 + y_2 - x_1 - x_2, y_1 y_2 - x_1 x_2)\{-2\}
}
\end{equation}
commute. In the first diagram the vertical map on the right comes about because $y_1 + y_2 - x_1 - x_2$ and $y_1 y_2 - x_1 x_2$ belong to the ideal generated by the $y_i - x_i$. We note that the right-hand vertical map in (\ref{eq:bimodalg2}) is the determinant of the matrix which writes the regular sequence $(y_1 + y_2 - x_1 - x_2, y_1 y_2 - x_1 x_2)$ in terms of $(y_1 - x_1, y_2 - x_2)$, and if one thinks in terms of Koszul complexes it is clear that (\ref{eq:bimodalg1}) is in some sense dual to (\ref{eq:bimodalg2}). This will be made explicit at the end of Section \ref{section:stabilisation}.

Now we make the jump to matrix factorisations. If $V$ is a homogeneous symmetric polynomial of even degree then the functors $\varphi_*$ and $\varphi^*$ lift to a biadjoint pair on the triangulated categories of finite-rank graded matrix factorisations
\[
\xymatrix@C+2pc{
\hmf(A, V) \ar@<1ex>[r]^{\varphi^*} & \hmf(S, V) \ar@<1ex>[l]^{\varphi_*} \, .
}
\]
Let us fix for example the polynomial $V = x_1^{N+1} + x_2^{N+1}$. These triangulated categories are algebraic, and the general theory of \cite[Section 8.1]{RouquierMexico} (see also \cite{rickard}) applies to produce an interesting autoequivalence of $\K^b( \hmf(S, V) )$. Namely, tensoring with the complexes of graded bimodules
\begin{equation}\label{eq:bimod_cpx1}
\xymatrix{
0 \ar[r] & S \otimes_A S \ar[r]^-{\varepsilon} & \underline{S} \ar[r] & 0 \, ,
}
\end{equation}
and
\begin{equation}\label{eq:bimod_cpx2}
\xymatrix{
0 \ar[r] & \underline{S} \ar[r]^-{\eta'} & S \otimes_A S\{-2\} \ar[r] & 0
}
\end{equation}
determines endofunctors on $\K^b( \hmf(S, V) )$ which are mutually inverse equivalences. In fact these equivalences are the ones associated by Khovanov and Rozansky to over- and under-crossings, so the whole theory can be built out of (\ref{eq:bimod_cpx1}) and (\ref{eq:bimod_cpx2}). But there is an additional step before we can phrase the construction in the original terms: writing $T = S/(V)$ there is an equivalence
\[
\hmf(S, V) \cong \qderu{b}{\textup{grmod}\, T}/\K^b(\textup{grproj}\, T)\,.
\]
The intuition is that a graded matrix factorisation of $V$ remembers the asymptotic part of a graded free resolution of a graded $T$-module. The functors defined on $\hmf(S,V)$ in terms of tensoring with the bimodules $S \otimes_A S$ and $S$ will only depend on the asymptotic part of the resolutions of these bimodules over $T$ (see Appendix \ref{subsection:kernelfunctors}). It is therefore more economical to replace the bimodules in (\ref{eq:bimod_cpx1}) and (\ref{eq:bimod_cpx2}) by the matrix factorisations which remember this asymptotic data, and the bimodule morphisms by the morphisms of matrix factorisations coming from the asymptotic part of the lifted morphism to the graded free resolutions.

This process is known as \textsl{stabilisation}, and will take place in the next subsection. We will stabilise $S \otimes_A S$ to a matrix factorisation $\Xbul$ and $S$ to $\Xcirc$. The morphisms $\varepsilon$ and $\eta'$ will stabilise to morphisms $\chi_1: \Xbul \lto \Xcirc$ and $\chi_0: \Xcirc \lto \Xbul\{-2\}$, respectively, and with these stabilisations in hand we will be ready to define the $\sln$ link homology.

\begin{remark} The bimodule $S \otimes_A S$ is the simplest kind of \textsl{Soergel bimodule}. The connection between the triply-graded $\sln$ link homology and Soergel bimodules was uncovered in \cite{k0510265} and similar ideas give the connection to the bigraded $\sln$ link homology considered here; see~\cite[Section~2.3]{w0610650} where the theory is presented in terms of stabilisation of Soergel bimodules. The relation to Soergel bimodules is also explained very clearly in \cite{b1105.0702}.
\end{remark}

\subsection{Stabilisation}\label{section:stabilisation}

We will now present the explicit matrix factorisations and morphisms which stabilise the adjunction given in the previous subsection. We begin with a general definition: let $R$ be a graded ring, $W \in R_{2c}$ a homogeneous element of even degree and $M$ a finitely generated graded $R/(W)$-module. The \textsl{stabilisation} of $M$ with respect to $W$ is a finite-rank graded matrix factorisation $X_M$ of $W$ over $R$ together with a morphism of linear factorisations $\pi: X_M \lto M$ with the property that
\[
\Hom_R(Y, X_M) \lto \Hom_R(Y, M), \qquad f \mapsto \pi \circ f
\]
is a quasi-isomorphism for every finite-rank graded matrix factorisation $Y$ of $W$.

Clearly the stabilisation, if it exists, is unique up to homotopy equivalence. Take $S = \QQ[x_1,x_2]$ as before and $R = \QQ[x_1,x_2,y_1,y_2]$. The reader wanting to compare with \cite[Section 6]{kr0401268} should make the change of variables $x_1,x_2,y_1,y_2 \longleftrightarrow x_4,x_3,x_1,x_2$. Fix an integer $N > 0$ and define
\begin{align*}
W &= y_1^{N+1} + y_2^{N+1} - x_1^{N+1} - x_2^{N+1}\,,\\
t_i &= y_i - x_i\,,\\
s_1 &= y_1 + y_2 - x_1 - x_2\,,\\
s_2 &= y_1 y_2 - x_1 x_2\,.
\end{align*}
As in Lemma \ref{lemma:bimodules_as_algebras} we identify $S \otimes_A S$ and $S$ with the graded $R$-modules $R/(\bs{s})$ and $R/(\bs{t})$. Both are killed by $W$ and therefore give $R/(W)$-modules. The stabilisations of these $R/(W)$-modules exist and are described using cyclic Koszul complexes. Note that
\begin{equation}\label{eq:cons1}
W = w_1 t_1 + w_2 t_2 \, , \qquad w_i = \frac{y_i^{N+1} - x_i^{N+1}}{y_i - x_i} \, ,
\end{equation}
and due to the symmetry of $W$ there are polynomials $u_1,u_2$ such that
\begin{equation}\label{eq:cons2}
W = u_1 \cdot ( y_1 + y_2 - x_1 - x_2 ) + u_2 \cdot ( y_1 y_2 - x_1 x_2 ) \, .
\end{equation}
Then in the notation of Section \ref{prelim:cyclic_koszul}, we define graded matrix factorisations of $W$ over $R$ by
\begin{align}
\Xbul &= \{ (u_1, u_2)\, , (y_1 + y_2 - x_1 - x_2, y_1 y_2 - x_1 x_2) \} \, , \nonumber \\
\Xcirc &= \{ (w_1, w_2) \, , (t_1, t_2) \}\,. \label{XbulXcirc}
\end{align}
Of course there is some indeterminacy in the choice of the $u_i$'s, but up to homotopy equivalence the factorisation $\Xbul$ does not depend on the choice of $u_i$'s (this is a standard fact, see for example \cite[Lemma 3.10]{r0607544}, and will also be clear from the description of $\Xbul$ as the solution of a universal problem). Moreover, we can and do assume that the polynomials $u_i$ are invariant under interchanging the $x_{i}$'s with the $y_{i}$'s, that is, we arrange that $u_i(y_1,y_2,x_1,x_2) = u_i(x_1,x_2,y_1,y_2)$. We will manage to avoid discussing the explicit form of these polynomials, making use only of (\ref{eq:cons2}) and the symmetry just mentioned.

Recall the model of the Koszul complex from Section \ref{prelim:cyclic_koszul}: we set $F = R \theta_1 \oplus R \theta_2$, where the $\theta_i$'s are formal symbols of $\Ztwo$-degree $-1$. Then both $\Xbul$ and $\Xcirc$ have underlying $\Ztwo$-graded module $\bigwedge F$ and the differentials are given by
\begin{align*}
\Xbul & = ( \bigwedge F, \beta_+ + \beta_{-}) \, , & \beta_+ & = \Big(\sum_i s_i \theta_i^*\Big) \neg\, (-) \, , & \beta_{-} & = \Big(\sum_i u_i \theta_i\Big) \wedge (-) \, , \\
\Xcirc & = ( \bigwedge F, \delta_+ + \delta_{-}) \, , & \delta_+ & = \Big(\sum_i t_i \theta_i^*\Big) \neg\, (-) \, , & \delta_{-} & = \Big(\sum_i w_i \theta_i\Big) \wedge (-) \, .
\end{align*}
In the $\mathds{Z}$-grading on $\bigwedge F$ the differentials with subscripts $\pm$ increase/decrease degree, so it is clear that there are morphisms of linear factorisations (viewing modules as factorisations of zero)
\begin{align*}
& \pi_{\bullet}: \Xbul = ( \bigwedge F, \beta_+ + \beta_{-}) \lto {\bigwedge}^0 F/\Im(\beta_+^0) = R/(\boldsymbol{s}) \, ,\\
& \pi_{\circ}: \Xcirc = ( \bigwedge F, \delta_+ + \delta_{-}) \lto {\bigwedge}^0 F/\Im(\delta_+^0) = R/(\boldsymbol{t}) \, .
\end{align*}

\begin{proposition} The morphisms $\pi_{\bullet}$ and $\pi_{\circ}$ are the stabilisations of $R/(\bs{s})$ and $R/(\bs{t})$, respectively.
\end{proposition}

That $\pi_{\circ}$ stabilises the diagonal is a theorem of Dyckerhoff \cite{d0904.4713} and the same argument settles related cases like $\pi_{\bullet}$. We give a new proof in  Appendix \ref{section:morphismpert} using different methods; methods that will help us stabilise morphisms later. What Dyckerhoff actually shows (and we recall in Appendix \ref{subsection:kernelfunctors}) is that the functors determined by these matrix factorisations agree with the functors determined by the bimodules, that is, both squares implicit in the diagram 
\[
\xymatrix@C+3pc@R+1pc{
\hmf(\QQ[x_1,x_2], x_1^{N+1} + x_2^{N+1})\ar[d]_-{\inc}\ar@<1ex>[r]^-{(-) \otimes R/(\bs{s})}
\ar@<-1ex>[r]_-{(-) \otimes R/(\bs{t})}
& \hmf(\QQ[y_1,y_2], y_1^{N+1} + y_2^{N+1})\ar[d]^-{\inc}\\
\HMF(\QQ[x_1,x_2], x_1^{N+1} + x_2^{N+1})\ar@<1ex>[r]^-{(-) \otimes \Xbul}
\ar@<-1ex>[r]_-{(-) \otimes \Xcirc}
& \HMF(\QQ[y_1,y_2], y_1^{N+1} + y_2^{N+1})
}
\]
commute up to natural isomorphism; see also \cite{b1105.0702}. Here $(-) \otimes R/(\bs{t})$ is the identity functor (up to relabelling variables) and $(-) \otimes R/(\bs{s}) = \varphi^* \varphi_*$ is the result of restricting to symmetric polynomials and coming back up. Having obtained a description of these two functors on $\hmf(S, x_1^{N+1} + x_2^{N+1})$ in terms of matrix factorisations, it remains to stabilise the complexes in (\ref{eq:bimod_cpx1}) and (\ref{eq:bimod_cpx2}), that is, we need to stabilise the maps $\varepsilon$ and $\eta'$. 

The first observation is that, by the proposition just stated, the map
\[
\Hom(1,\pi_{\circ}): \Hom_R(\Xbul, \Xcirc) \lto \Hom_R(\Xbul, R/(\boldsymbol{t}))
\]
is a quasi-isomorphism. By definition the morphism $\chi_1: \Xbul \lto \Xcirc$ stabilising $\varepsilon$ is the cohomology class of $\Hom_R(\Xbul, \Xcirc)$ corresponding under the above quasi-isomorphism to the composite
\begin{equation}\label{eq:cons4}
\xymatrix{
\Xbul \ar[r]^-{\pi_{\bullet}} & R/(\boldsymbol{s}) \ar[r]^-{\can} & R/(\boldsymbol{t})
}
\end{equation}
where the second map comes about because of the inclusion of ideals $(\boldsymbol{s}) \subseteq (\boldsymbol{t})$. That is, $\chi_1$ will be the unique (up to homotopy) morphism making the following diagram commute:
\begin{equation}\label{eq:cons5}
\xymatrix@C+2pc{
\Xbul \ar[d]_{\chi_1} \ar[r]^{\pi_{\bullet}} & R/(\boldsymbol{s}) \ar[d]^{\can}\\
\Xcirc \ar[r]_{\pi_{\circ}} & R/(\boldsymbol{t}) \, . 
}
\end{equation}

We know abstractly that $\chi_1$ exists. With patience one can construct by hand an explicit matrix for such a map, and this appears to be the approach of \cite{kr0401268} where the matrices are presented without explanation; see also~\cite[Section 2]{kr0505056} and \cite[Example 2.3.7]{b1105.0702}. In~\cite[Section 7.5]{w0907.0695} a method is developed for stabilising module maps of the form $R/(\bs{r}) \lto R/(\bs{r}')$ where $\bs{r},\bs{r}'$ are regular sequences and $(\bs{r}) \subseteq (\bs{r}')$, and this can obviously be used to give an explicit matrix for $\chi_1$.

We are going to give yet another derivation, based on the observation that the same perturbation techniques of \cite{dm1102.2957} that go into the construction of the idempotent in Section \ref{compilewebs} also provide an explicit homotopy inverse of $\Hom(1,\pi_{\circ})$, and therefore a matrix for $\chi_1$. The techniques can be used to lift any morphism $Y \lto R/(\bs{t})$ to a morphism $Y \lto \Xcirc$, and with other applications in mind we feel it is worth including the argument here. The statement is given by the next lemma (but we have relegated the details to Appendix~\ref{section:morphismpert}), for which we need to introduce the polynomial 
\[
\gamma = \frac{\partial u_1}{\partial y_1} - \frac{\partial u_1}{\partial y_2} - \frac{1}{2} \frac{\partial u_2}{\partial y_2}(x_2 + y_2) + \frac{1}{2} \frac{\partial u_2}{\partial y_1}(x_1 + y_1) \, .
\]
\begin{lemma}\label{lemma:consfirst} The partial differential equation 
\begin{equation}\label{eq:consfirst1}
2 z + \frac{\partial z}{\partial y_1} (y_1 - x_1) + \frac{\partial z}{\partial y_2}(y_2 - x_2) = \gamma
\end{equation}
in an unknown $z$ has a unique polynomial solution $\Omega_2(\gamma)$, and the map $\chi_1: \Xbul \lto \Xcirc$ defined by
\begin{align*}
\chi_1(1) &= 1 + \Omega_2( \gamma ) \cdot \theta_1 \theta_2,\\
\chi_1(\theta_1) &= \theta_1 + \theta_2,\\
\chi_1(\theta_2) &= \frac{1}{2}(x_2 + y_2)\cdot \theta_1 + \frac{1}{2}(x_1 + y_1)\cdot \theta_2,\\
\chi_1(\theta_1\theta_2) &= \frac{1}{2}(x_1 + y_1 - x_2 - y_2) \cdot \theta_1\theta_2
\end{align*}
is a morphism of graded matrix factorisations, making the diagram (\ref{eq:cons5}) commute up to homotopy.
\end{lemma}

The only obstacle to having an explicit formula for $\chi_1$ is determining the solution $\Omega_2(\gamma)$ of the differential equation (\ref{eq:consfirst1}). Fortunately the solution has already been provided by Khovanov and Rozansky: up to a sign, it is the polynomial $a_2 = \frac{1}{2}u_2 + (u_1 + y_1 u_2 - w_2)/(x_1-y_1)$ of \cite[Section~6]{kr0401268} (where we set the parameter~$\lambda$ of \cite{kr0401268} equal to~$\frac{1}{2}$).

\begin{lemma} $\Omega_2(\gamma) = -a_2$.
\end{lemma}
\begin{proof}
Setting $y_1 = x_1$ in (\ref{eq:cons2}) shows that $u_1 + y_1 u_2 - w_2$ is divisible by $x_1 - y_1$. We must show it is a solution of (\ref{eq:consfirst1}). All we know of the polynomials $u_i$ is that they satisfy (\ref{eq:cons2}), and indeed one shows that $-a_2$ is a solution of (\ref{eq:consfirst1}) by direct substitution using the $y_2$-derivative of (\ref{eq:cons2}) to obtain an expression for $u_1 + y_1 u_2$.
\end{proof}

The $\chi_1$ given above agrees with the definition given in \cite{kr0401268} for $\lambda = \frac{1}{2}$, up to a sign which arises because they write their differentials and maps in the basis $\{1, \theta_2 \theta_1 = - \theta_1 \theta_2, \theta_1, \theta_2\}$.

\medskip

Next we discuss the morphism $\chi_0$ stabilising $\eta'$. Recall that $\eta': S \lto S \otimes_A S\{-2\}$ corresponds to the map $R/(\bs{t}) \lto R/(\bs{s})\{-2\}$ given by multiplication with $\frac{1}{2}(x_1 + y_1 - x_2 - y_2)$. So by definition $\chi_0$ is the unique (up to homotopy) morphism making the following diagram commute:
\begin{equation}\label{eq:cons50}
\xymatrix@C+2pc{
\Xcirc \ar[d]_{\chi_0} \ar[r]^{\pi_{\circ}} & R/(\boldsymbol{t}) \ar[d]^-{\frac{1}{2}(x_1+y_1-x_2-y_2)}\\
\Xbul\{-2\} \ar[r]_{\pi_{\bullet}} & R/(\boldsymbol{s})\{-2\} \, .
}
\end{equation}
Rather than use perturbation theory, which is much more complicated in this case, we obtain $\chi_0$ by dualising $\chi_1$ to get a morphism $(\chi_1)\mdual: \Xcirc\mdual \lto \Xbul\mdual$. Using the standard properties of cyclic Koszul complexes collected in Appendix \ref{appendix:graded_mfs}, we can rewrite the dual of a cyclic Koszul complex as a cyclic Koszul complex; up to signs, grading shifts, and a change of variables $\Xcirc$ and $\Xbul$ are self-dual, so we get a morphism $\Xcirc \lto \Xbul$. This will turn out to be the desired map $\chi_0$.

More carefully, we define $\chi_0'$ to be the morphism making the diagram
\begin{equation}\label{eq:chi1dualischi0}
\xymatrix@C+3pc{
\{ \bs{w}, \bs{t} \}\mdual = \Xcirc\mdual \ar[r]^-{\chi_1\mdual} & \Xbul\mdual = \{ \bs{u}, \bs{s} \}\mdual \ar[d]^-{(\ref{lemma:cyclickos1})}\\
\{ -\bs{t}, \bs{w} \} \ar[u]^-{(\ref{lemma:cyclickos1})} & \{ -\bs{s}, \bs{u} \} \ar[d]^-{(\ref{lemma:cyclickos2})}\\
\{ -\bs{w}, \bs{t} \}\{ 2N - 2 \} \ar[u]^-{(\ref{lemma:cyclickos2})} & \{ -\bs{u}, \bs{s} \} \{ 2N - 4 \} \ar[d]^-{(\ref{lemma:cyclickos3})}\\
\{ \bs{w}, -\bs{t} \}\{ 2N - 2 \} \ar@{-->}[r]_-{\chi_0'} \ar[u]^-{(\ref{lemma:cyclickos3})} & \{ \bs{u}, -\bs{s} \}\{ 2N - 4 \}
}
\end{equation}
commute, where each vertical map is an isomorphism labelled by the corresponding lemma in Appendix~\ref{appendix:graded_mfs}. The matrix factorisations $\{ \bs{w}, -\bs{t} \}$ and $\{ \bs{u}, -\bs{s} \}$ are what we would assign to the diagrams~\eqref{MFcrossings} with the orientation reversed. Switching the $x$'s and $y$'s will therefore give us a map between $\Xbul$ and $\Xcirc$.

Let $\phi: R \lto R$ be the $\nQ$-automorphism with $\phi(x_i) = y_i$ and $\phi(y_i) = x_i$. It is clear that $\phi$ sends $t_i$ to $-t_i$, $s_i$ to $-s_i$ and $W$ to $-W$, so that by pulling back along $\phi$ (equivalently, applying $\phi$ to the entries of all matrices) we get a morphism of matrix factorisations of $W$: 
\[
\xymatrix{
\Xcirc = \{ \phi\bs{w}, -\phi \bs{t} \} = \phi^*\{ \bs{w}, -\bs{t} \} \ar[rr]^-{\phi^* \chi'_0} & & \phi^*\{ \bs{u}, -\bs{s} \}\{-2\} = \{ \phi \bs{u}, -\phi \bs{s} \}\{-2\} = \Xbul\{-2\} \, ,
}
\]
and we define $\chi_0$ to be this map $\phi^* \chi'_0$. Here we use that we have chosen the $u_i$ such that $\phi(u_i) = u_i$. If one is careful with signs, it is straightforward to deduce from the above the explicit form of $\chi_0$:
\begin{align}
\chi_0(1) &= \frac{1}{2}(x_1 + y_1 - x_2 - y_2) \cdot 1 + a_2 \cdot \theta_1 \theta_2 \, ,\nonumber \\
\chi_0(\theta_1) &= -\theta_2 + \frac{1}{2}(x_1 + y_1) \cdot \theta_1\, ,\nonumber \\
\chi_0(\theta_2) &= \theta_2 - \frac{1}{2}(x_2 + y_2) \cdot \theta_1\, , \nonumber \\
\chi_0(\theta_1\theta_2) &= \theta_1\theta_2\,, \label{eq:chi0defn}
\end{align}
where $a_2 = \frac{1}{2}u_2 + (u_1 + y_1 u_2 - w_2)/(x_1-y_1)$ is as before, and we have used the fact that $\phi(a_2) = -a_2$. Again, keeping in mind the change of basis discussed above, this agrees with the explicit matrices for $\chi_0$ given in \cite[Section~6]{kr0401268} with their parameter~$\mu$ set to~$\frac{1}{2}$. So far we have just produced a morphism, but it is clear from the explicit form given above that this morphism actually stabilises the map $\eta'$, completing our stabilisation of the short complexes (\ref{eq:bimod_cpx1}) and (\ref{eq:bimod_cpx2}).

\begin{lemma} The map $\chi_0: \Xcirc \lto \Xbul\{-2\}$ of (\ref{eq:chi0defn}) is a morphism of graded matrix factorisations making (\ref{eq:cons50}) commute.
\end{lemma}

\section{Properties of graded matrix factorisations}\label{appendix:graded_mfs}

Let $R = \bigoplus_{i \ge 0} R_i$ be a graded ring. If $X$ is a finite-rank graded matrix factorisation of $W \in R_{2c}$ over $R$ then the \textsl{dual factorisation} is $X\mdual = \Hom_R(X,R)$, which factorises $-W$. More explicitly, the dual factorisation is the pair $(-(d^1_X)^*, (d^0_X)^*)$.

\begin{remark}\label{remark:dual_koszul_cyclic} If $a,b \in R$ are homogeneous with $\deg(a) + \deg(b) = 2c$ then it is clear that there is an isomorphism $\{ b, a \}\mdual \cong \{-a, b\}$, and if we think of these cyclic Koszul factorisations as exterior algebras on a symbol $\theta$, this isomorphism sends the basis element $1^*$ to $1$ and $\theta^*$ to $\theta$.
\end{remark}

We will need the following basic facts relating the graded tensor and Hom. The proofs are easy (the standard isomorphisms for graded modules commute with the differentials) and we omit them.

\begin{lemma}\label{lemma:homdual} Given $W, W' \in R_{2c}$ and a finite-rank graded matrix factorisation~$X$ of~$W$ and a graded linear factorisation~$Y$ of~$W'$, there is a natural isomorphism
\[
\xi: X\mdual \otimes Y \lto \Hom_R(X,Y)
\]
of graded linear factorisations of $W' - W$, defined for $\Ztwo$-homogeneous elements $\nu \in X\mdual, y \in Y$ and $x \in X$ by $\xi( \nu \otimes y )(x) = (-1)^{|y||\nu|} \nu(x) \cdot y$.
\end{lemma}

\begin{lemma} Given $X,Y,Z$ which are graded linear factorisations of $W, W', W'' \in R_{2c}$, respectively, there is a natural isomorphism of graded linear factorisations of $W'' - W' - W$
\[
\Hom_{\mathrm{gr}}(X \otimes Y, Z) \lto \Hom_{\mathrm{gr}}(X, \Hom_{\mathrm{gr}}(Y,Z)) \, .
\]
\end{lemma}

In particular if $X, Y$ are finite-rank graded matrix factorisations we have
\begin{align}
(X \otimes Y)\mdual &= \Hom_R(X \otimes Y, R) 
\cong \Hom_R(X, \Hom_R(Y,R)) \nonumber \\
&\cong \Hom_R(X,R) \otimes \Hom_R(Y,R) = X\mdual \otimes Y\mdual \, . \label{eq:dual_tensor}
\end{align}
We also note that the isomorphism $X\langle n \rangle \otimes Y \cong (X \otimes Y)\langle n \rangle$ involves no signs, but the isomorphism $X \otimes (Y \langle n \rangle) \cong (X \otimes Y)\langle n \rangle$ involves a sign: $x \otimes y \longmapsto (-1)^{n|x|} x \otimes y$ for $\Ztwo$-homogeneous $x$. Grading shifts $\{ m \}$ can be pulled out of either component of a tensor product.
\\

In the following let $R = \bigoplus_{i \ge 0} R_i$ be a graded ring, and $\bs{a}, \bs{b}$ sequences of $n$ homogeneous elements in $R$ with $\deg(a_i) + \deg(b_i) = 2c$. If $\bs{a} = (a_1,\ldots,a_n)$ then $- \bs{a}$ denotes $(-a_1, \ldots, -a_n)$. In describing maps between cyclic Koszul complexes we use the symbols $\theta_i$ introduced in Section \ref{prelim:cyclic_koszul}.

\begin{lemma}\label{lemma:cyclickos1} There is a canonical isomorphism of graded matrix factorisations $\{ \bs{a}, \bs{b} \}\mdual \cong \{ -\bs{b}, \bs{a} \}$.
\end{lemma}
\begin{proof}
Clear from Remark \ref{remark:dual_koszul_cyclic}.
\end{proof}

\begin{lemma}\label{lemma:cyclickos2} There is a canonical isomorphism of graded matrix factorisations
\begin{equation}
\label{abba}
\{ - \bs{b}, \bs{a} \} \cong \{ -\bs{a}, \bs{b} \}\langle n \rangle \big\{ \sum_i \deg(a_i) - nc \big\} \, .
\end{equation}
\end{lemma}
\begin{proof}
Let us begin with $a,b$ homogeneous such that $\deg(a) + \deg(b) = 2c$. Then $\{ b, a \}\langle 1 \rangle$ is
\[
\xymatrix{
R\{ \deg(a) - c \} \ar[r]^-{-a} & R \ar[r]^-{-b} & R\{ \deg(a) - c \}
}
\]
and shifting the grading by $\{ c - \deg(a) \}$ we have $\{ b, a \}\langle 1 \rangle\{ c - \deg(a) \}$ is equal to $\{ -a, -b \}$. Thus
\begin{align*}
\{ - \bs{b}, \bs{a} \} &\cong \{ -b_1, a_1 \} \otimes \ldots \otimes \{ -b_n, a_n \}\\
&\cong \{ -a_1, b_1 \}\langle 1 \rangle\{ \deg(a_1) - c \} \otimes \ldots \otimes \{ -a_n, b_n \}\langle 1 \rangle\{ \deg(a_n) - c \}\\
&\cong \{ -\bs{a}, \bs{b} \}\langle n \rangle\big\{ \sum_i \deg(a_i) - nc \big\} \, .
\end{align*}
\end{proof}

\begin{lemma}\label{lemma:cyclickos3} There is an isomorphism $\{ \bs{a}, \bs{b} \} \cong \{ -\bs{a}, - \bs{b} \}$ sending $\theta_i$ to $- \theta_i$.
\end{lemma}

We end with a discussion of cohomology. Let $R$ be a ring and $C = \bigoplus_{i \in \mathds{Z}} C^i$ a $\mathds{Z}$-graded $R$-module with $R$-linear maps $d_{\pm}: C \lto C$ of degree $\pm 1$, satisfying
$
(d_+)^2 = (d_{-})^2 = 0$ 
and 
$
 d_{+}d_{-} + d_{-}d_{+} = 0
$. 
Let $C_{\Ztwo}$ denote the $\Ztwo$-folding, which is a $\Ztwo$-graded complex with differential $d_{\text{tot}} = d_+ + d_{-}$. If we replace $d_{+}$ by $-d_{+}$ then we have another complex $(C_{\Ztwo}, -d_{+} + d_{-})$. We claim that these complexes have the same cohomology. Given a $\Ztwo$-graded complex $(X,d)$ we denote by $(X,d)_{-}$ the complex
\[
\xymatrix{%
X^0 \ar[r]^-{-d^0} & X^1 \ar[r]^-{d^1} & X^0\,.
}
\]

\begin{lemma}\label{lemma:cyclickoszulsign} There is a canonical isomorphism of $\Ztwo$-graded complexes
\[
T: (C_{\Ztwo}, d_{+} + d_{-}) \lto (C_{\Ztwo}, -d_{+} + d_{-})_{-}
\]
defined by $T|_{C^i} = (-1)^{\frac{1}{2}i(i+1)}$. This induces an isomorphism $H(C_{\Ztwo}, d_{+} + d_{-}) \lto H(C_{\Ztwo}, -d_{+} + d_{-})$.
\end{lemma}

In particular for sequences $\bs{a}, \bs{b}$ we have
\begin{equation}
\label{Habba}
H(\{ \bs{a}, \bs{b} \}) \cong H(\{\bs{a}, -\bs{b}\}) \, .
\end{equation}
Obviously if $R$ is a graded ring and each $C^i$ is a graded module such that $d_{\pm}$ have degree $c$ then $d_{\text{tot}}$ has bidegree $(1,c)$ on $C_{\Ztwo}$, the cohomology is naturally bigraded, and (\ref{Habba}) preserves the bigrading.

\section{The relation between reduced and unreduced homology}
\label{relationReducedUnreduced}

As was discussed in Section~\ref{compilewebs}, to a state graph $\Gamma$ with $m$ edges Khovanov and Rozansky assign a $\Ztwo$-graded complex $\krc(\Gamma)$ over the polynomial ring $R = \QQ[\boldsymbol{x}]$ in the edge variables $\boldsymbol{x} = \{x_1,\ldots,x_m\}$, defined by tensoring together local matrix factorisations assigned to the singular crossings and smoothings of the state graph. The definition depends on a choice of integer $N > 0$, which we fix throughout and omit from the notation, so that all our homologies are $\sln$ homologies.

The Khovanov-Rozansky complex $\krc(D)$ of a planar diagram $D$ of a link $L$ is defined by tensoring together two-term complexes build from these $\krc(\Gamma)$'s, as $\Gamma$ varies over all resolutions of $D$, and recall that the unreduced and reduced Khovanov-Rozansky homology are defined respectively by
\begin{align*}
H(L) &= H\big( H(\krc(D),\diffm), \diffh \big) \, ,\qquad
\overline{H}(L, K) = H\big( H( \krc(D) \otimes_{R} R/(x_i), \diffm), \diffh \big) \, ,
\end{align*}
where $K$ denotes a component of $L$, and $i$ is a label assigned in the planar diagram $D$ to an edge lying on this component. In this appendix we make some remarks about the relationship between reduced and unreduced homology. Let us fix the label $i$ in the following.

After taking the $\diffm$-cohomology everything becomes finite-dimensional, and taking the cohomology with respect to the differential $\diffh$ is straightforward; the whole problem with computing these invariants is computing $H(\krc(D), \diffm)$. In this article the emphasis is on computations, and in this context there is a big difference between taking cohomology before and after setting $x_i = 0$. Since $H(\krc(D), \diffm)$ is a direct sum of modules $H(\krc(\Gamma), \diffm)$ for various state graphs $\Gamma$ let us speak only of state graphs, in which case we are making the distinction between
\[
H(\krc(\Gamma) \otimes_R R/(x_i), \diffm ) \qquad \text{and} \qquad H( \krc(\Gamma), \diffm ) \otimes_R R/(x_i) \, .
\]
In the second case one computes the essential ingredient $H( \krc(\Gamma), \diffm )$ in the \textsl{unreduced} homology, and then throws away information; in the first case one deals from the beginning with one less variable, so the computation is significantly easier. For this reason, the first definition (which is the original one of \cite{kr0401268}) is the one our code computes; but since the second definition is the one studied in \cite{r0607544} and other references, we want to explain the relationship between the two. Let us set
\begin{align*}
H(\Gamma) &= H( \krc(\Gamma), \diffm ) \, ,\qquad
\redh(\Gamma) = H( \krc(\Gamma) \otimes_R R/(x_i), d ) \, .
\end{align*}
These are both finite-dimensional $(\mathds{Z} \times \Ztwo)$-graded vector spaces. It is shown in \cite{kr0401268} that $H(\Gamma)$ is concentrated in only one $\Ztwo$-degree, namely the degree of the parity $p = p(\Gamma)$ defined by erasing all four-valent vertices in $\Gamma$ and counting the number of circles in the resulting graph. It follows that the link homology $H(L)$, which is \textsl{a priori} $(\ZZ \times \ZZ \times \Ztwo)$-graded, is actually only $(\ZZ \times \ZZ)$-graded.

On the other hand $\redh(\Gamma)$ has nonzero contributions in both $\Ztwo$-degrees, so the reduced link homology $\overline{H}(L,K)$, as defined in \cite{kr0401268}, is honestly $(\ZZ \times \ZZ \times \Ztwo)$-graded. We write $\redh^j(\Gamma)$, respectively $\redh^j(L,K)$, for the $\Ztwo$-component in degree $j \in \Ztwo$. The point of this appendix is to check that $\redh^0(L,K)$ and $\redh^1(L,K)$ only differ by a grading shift.

\begin{lemma}\label{lemma:redvsunred} There is an isomorphism of $(\ZZ \times \ZZ)$-graded $\mathds{Q}$-vector spaces 
\[
\redh^{p+1}(L,K) \cong \redh^{p}(L,K)\{N-1\} \, .
\]
\end{lemma}

This will follow if we can show that there is an isomorphism $\redh^{p+1}(\Gamma) \cong \redh^p(\Gamma)\{N-1\}$ natural with respect to the $\chi$ maps. It is easy to see that the two $\Ztwo$-components of $\redh(\Gamma)$ have the same dimension: it is the gradings that we need to compare. From the short exact sequence
\[
\xymatrix{%
0 \ar[r] & C(\Gamma) \ar[r]^-{x_{i}} & C(\Gamma)\{-2\} \ar[r] & C(\Gamma) \otimes_R R/(x_{i})\{-2\} \ar[r] & 0
}
\]
we deduce a long exact sequence in cohomology, of graded $R$-modules:
\begin{equation}\label{eq:redvsunred_seq}
\xymatrix{%
0 \ar[r] & \redh^{p+1}(\Gamma)\{N-1\} \ar[r] & H(\Gamma) \ar[r]^-{x_{i}} & H(\Gamma)\{-2\} \ar[r] & \redh^p(\Gamma)\{-2\} \ar[r] & 0 \, .
}
\end{equation}
It is clear from this exact sequence that the dimensions of the odd and even degrees of $\redh(\Gamma)$ agree (this is also observed in \cite[Proposition 3.12]{r0607544}). Moreover, assembling the $H(\Gamma)$'s to form the link homology, we deduce that
\[
\redh^p(L, K) \cong H(H(C(D),d) \otimes_R R/(x_i), d_{\chi})\,.
\]
As has already been mentioned, this $(\ZZ \times \ZZ)$-graded $\QQ$-vector space is adopted as the \textsl{definition} of the reduced Khovanov-Rozansky homology in \cite{r0607544,w0610650}. In light of Lemma \ref{lemma:redvsunred} this is reasonable, as the other $\Ztwo$-degree of $\redh(L,K)$ contains no new information, but as far as we know the lemma has not appeared before in the literature (although it may be known to the experts). To relate the $\ZZ$-gradings of $\redh^p(\Gamma)$ and $\redh^{p+1}(\Gamma)$ we make use of the following basic property of $H(\Gamma)$.

\textit{Convention.} Since we ultimately only care about the homologies $H(L)$ and $\overline{H}(L,K)$, we are free to choose $D$ to be the closure of a braid with $i$ the label on one of the edges in the closure, and for the rest of this section we make this assumption. In particular, $\Gamma$ is a braid graph.

\begin{lemma}\label{lemma:directsumalgebras} $H(\Gamma)$ is concentrated in degree $p(\Gamma)$, and as a graded $\QQ[x_i]$-module it is isomorphic to a direct sum of copies of $\QQ[x_i]/(x_i^N)$ shifted in the $\mathds{Z}$-grading, that is, for some integers $b_j$ we have
\[
H(\Gamma) \cong \bigoplus_{j} \QQ[x_i]/(x_i^N)\{b_j\} \, .
\]
\end{lemma}

Further, one can prove that the integers $b_j$ have a symmetry: $H(\Gamma)$ is always self-dual as a graded vector space, but we will not need this. Taking the lemma as a given:

\begin{proof}[Proof of Lemma \ref{lemma:redvsunred}] We claim that $\redh^{p+1}(\Gamma) \cong \redh^p(\Gamma)\{N-1\}$ as graded vector spaces. In light of the exact sequence (\ref{eq:redvsunred_seq}), to understand $\overline{H}(\Gamma)$ it suffices to understand the action of $x_i$ on $H(\Gamma)$. But with Lemma \ref{lemma:directsumalgebras} in hand this is trivial: it reduces to understanding the action of $x_i$ on $\QQ[x_i]/(x_i^N)$, and we deduce that
\begin{align*}
\redh^{p+1}(\Gamma)\{N-1\} &\cong \bigoplus_j \QQ \cdot x_i^{N-1} \{ b_j \} = \bigoplus_j \QQ\{b_j + 2N - 2 \} \, ,\qquad 
\redh^p(\Gamma) \cong \bigoplus_j \QQ \{ b_j \}\,.
\end{align*}
From this we deduce that $\redh^{p+1}(\Gamma) \cong \bigoplus_j \QQ\{b_j + N - 1 \} \cong \redh^p(\Gamma)\{N-1\}$, as claimed.

Now suppose that $\Gamma, \Gamma'$ are state graphs of $D$ for which there is a morphism $\chi: C(\Gamma) \lto C(\Gamma')$. We need to prove that the diagram
\[
\xymatrix@C+1pc{
\redh^{p+1}(\Gamma) \ar[d]_{\redh^{p+1}\chi}\ar[r]^-{\cong} & \redh^p(\Gamma)\{N-1\}\ar[d]^{\redh^p \chi}\\
\redh^{p+1}(\Gamma') \ar[r]_-{\cong} & \redh^p(\Gamma')\{N-1\}
}
\]
commutes. Decompose $H(\Gamma), H(\Gamma')$ as in Lemma \ref{lemma:directsumalgebras} and let $\chi': \QQ[x_i]/(x_i^N) \lto \QQ[x_i]/(x_i^N)$ be one of the components of $\chi$. The corresponding component of $\redh^{p+1} \chi$ is the restriction of $\chi'$ to the ideal generated by $x_i^{N-1}$, and the component of $\redh^p \chi$ is the map induced by $\chi$ on the quotients by the ideal $(x_i)$. But since $\chi'$ is homogeneous it sends $1$ to $\lambda x_i^a$ for some $\lambda \in \QQ$ and $a \in \{0,\ldots, N - 1\}$, and it is easy to see that if $a > 0$ then these components of $\redh^{p+1} \chi$ and $\redh^p \chi$ are both zero. On the other hand if $a = 0$ then the components are both multiplication by $\lambda$, and hence the diagram commutes. Since the differentials in $(H^p(C(D) \otimes_R R/(x_i),d), d_\chi)$ are built out of the $\chi$'s, we may conclude from this that $\redh^{p+1}(L,K) \cong \redh^{p}(L,K)\{N-1\}$, completing the proof.
\end{proof}

\begin{proof}[Proof of Lemma \ref{lemma:directsumalgebras}]
The argument is similar to that of \cite[Lemma 5.8]{r0607544}. We prove the claim by induction on the complexity of braid graphs $\Gamma$ together with a chosen edge $i$ appearing in the braid closure, using the induction scheme of~\cite{w0508064} and the MOY relations~\eqref{MOYcatDecomp1}--\eqref{MOYcatDecomp4} proved in \cite[Section 6]{kr0401268}. Let us denote by $\Gamma_{\text{I}}, \Gamma_{\text{II}}, \Gamma_{\text{III}}$ the graphs which are the arguments of~$C$ on the left of \eqref{MOYcatDecomp1}, \eqref{MOYcatDecomp2}, \eqref{MOYcatDecomp4}, respectively. 

By \cite{w0508064}, if $\Gamma$ is the closure of an open braid graph $\Gamma_{\text{open}}$ then $\Gamma$ always contains a region of the form $\Gamma_{\text{I}}$, or $\Gamma_{\text{open}}$ contains a region $\Gamma_{\text{II}}$ or $\Gamma_{\text{III}}$. This means that using only the MOY relations we can go from any braid graph to a family of unlinked circles, and in this base case $H(\Gamma)$ is a tensor product over $\mathds{Q}$ of grading shifted copies of $\QQ[x_j]/(x_j^N)\langle 1 \rangle$ for various $j$. By hypothesis $i$ is among these indices, so both claims are clear in the base case.

The only MOY relation which changes the parity is relation~\eqref{MOYcatDecomp1}, which changes the parity by one and also introduces a suspension, so it is now clear by induction that $H(\Gamma)$ is always concentrated in degree $p(\Gamma)$, and we need to check the other claim.

Given a braid graph $\Gamma$ suppose that we can find a region of type $\Gamma_{\text{II}}$ or $\Gamma_{\text{III}}$ in $\Gamma_{\text{open}}$. Then the corresponding direct sum decomposition has $x_i$ as an external variable, so the equivalences of the decompositions are $x_i$-linear. By the inductive hypothesis the cohomology of all the other braid graphs involved in the decomposition are direct sums of grading shifted copies of $\QQ[x_i]/(x_i^N)$ so the same is true of $C(\Gamma)$.

It remains to treat the case where we can only find a region of type $\Gamma_{\text{I}}$, the catch being that $x_i$ may not be an external variable to the decomposition: $i$ may be the loop contracted away by the MOY relation. But in this case we may assume that, apart from some circles which we may ignore, the edge $i$ is the rightmost edge in the braid graph. 

In this situation we consider the dual braid graph $\Gamma^{\lor}$ where we reverse the orientation of all the edges. Note that for some sequences of polynomials $\boldsymbol{a},\boldsymbol{b}$ and an integer $p$ we have $C(\Gamma) = \{ \boldsymbol{a}, \boldsymbol{b} \}\{p\}$ and therefore $C(\Gamma^{\lor}) = \{ \bs{a}, - \bs{b} \}\{p\}$. But by~\eqref{Habba} $\{ \bs{a}, -\bs{b} \}$ has the same cohomology as $\{ \bs{a}, \bs{b} \}$, so the cohomology of $\Gamma$ and $\Gamma^{\lor}$ agree. It therefore suffices to prove the claim for $\Gamma^{\lor}$. This follows by the above argument, since $i$ is the leftmost edge and can therefore never be involved in a $\Gamma_{\text{I}}$-relation.
\end{proof}

\section{Stabilisation and perturbation}\label{section:morphismpert}

In this appendix we use perturbation techniques to study stabilisation. For simplicity graded matrix factorisations will not feature here: this means that we work throughout with matrix factorisations in the usual sense. One nonstandard piece of terminology is that for us a \textsl{linear factorisation} of~$W \in R$ is an arbitrary $\Ztwo$-graded $R$-module with an odd differential squaring to multiplication by $W$. One can tensor and Hom such things in the obvious way; see for example \cite[Section 2]{dm1102.2957}.

Let $R$ be a noetherian ring, $\ba = (a_1,\ldots,a_n)$ and $\bb = (b_1, \ldots, b_n)$ sequences in $R$ with~$\bs{b}$ regular, and set $W = \sum_i a_i b_i$. We define the matrix factorisation $\{ \ba, \bb \}$ of $W$ as in Section \ref{prelim:cyclic_koszul} so $F = \bigoplus_i R \theta_i$ with $|\theta_i| = -1$ and $\bigwedge F$ has two differentials $\delta_+$ and $\delta_{-}$ such that $\{ \ba, \bb \} = (\bigwedge F, \delta_+ + \delta_{-})$. If we use only the differential $\delta_+$ then this is just the Koszul complex on the $b_i$, so there is a morphism of complexes from the Koszul complex to its cohomology
\begin{equation}\label{eq:app_stab0}
\pi: ( \bigwedge F, \delta_+ ) \lto R/(\bb)\,.
\end{equation}
In fact $\pi$ is also a morphism of linear factorisations $\{ \ba, \bb \} \lto R/(\bb)$. The task we give ourselves here is to show that under some mild hypotheses this map is universal, in the following sense: for any finite-rank matrix factorisation $Y$ of $W$ we claim that the map
\begin{equation}\label{eq:app_stab1}
\Hom(1, \pi): \Hom_R(Y, \{ \ba, \bb \}) \lto \Hom_R(Y, R/(\bb))
\end{equation}
is a quasi-isomorphism. That is, $\{ \ba, \bb \}$ \textsl{stabilises} $R/(\bb)$. This is a standard fact but the twist here is that we produce an explicit homotopy inverse using the language of deformation retracts and perturbation. Rather than repeat the basic facts about perturbation here, we direct the reader to \cite[Section 5]{dm1102.2957} and \cite{c0403266} for a full discussion; from now on we use the notation of \textsl{loc.\,cit.}

Here are our hypotheses: suppose that $R$ is an $S$-algebra for some ring $S$ and that (\ref{eq:app_stab0}) is a homotopy equivalence of $\ZZ$-graded complexes over $S$. More precisely, suppose that we can find an $S$-linear morphism of complexes $\sigma: R/(\bb) \lto (\bigwedge F, \delta_+)$ and $S$-linear homotopy $h$ on $\bigwedge F$ such that
\begin{equation}\label{eq:homotopycond}
\delta_+ h + h \delta_+ = 1 - \sigma \pi\,, \qquad h^2 = 0\,, \qquad h \sigma = 0\,.
\end{equation}
We will see later how to write down very explicit homotopies $h$ using connections. In any case, we deduce the claim about universality of $\pi$ from the following more general statement.

\begin{proposition}\label{prop:pertapptechnical} If $X$ is a finite-rank matrix factorisation of $V \in R$ over $R$, and the sum $V + W$ belongs to the image of the structure morphism $S \lto R$, then there is a deformation retract datum of linear factorisations of $V+ W$ over $S$, 
\begin{equation}\label{eq:pre_perturb_htm3}
\xymatrix@C+2pc{
(X \otimes R/(\bb), d_X \otimes 1) \ar@<-0.8ex>[r]_-{\sigma_\infty} & (X \otimes \bigwedge F, 1 \otimes \delta_+ + \mu), \ar@<-0.8ex>[l]_-{1 \otimes \pi}
} \quad h_\infty\,,
\end{equation}
where $\mu = d_X \otimes 1 + 1 \otimes \delta_-$ and $\sigma_\infty = \sum_{m \ge 0} (-1)^m (h \mu)^m \sigma$.
\end{proposition}
\begin{proof}
First let us note that the statement makes sense: the differential $d_X \otimes 1$ on $X \otimes R/(\bb)$ squares to multiplication by $V + W$, since $d_X^2 = V \cdot 1_X$ and $W$ acts as zero on $R/(\bb)$. We fix a homogeneous $R$-basis of $X$ and use it to extend $h, \sigma$ (with Koszul signs in the former case) to $S$-linear maps on and between $X \otimes \bigwedge F$ and $X \otimes R/(\bb)$. Then the conditions in (\ref{eq:homotopycond}) amount to the statement that we have a deformation retract datum of linear factorisations \textsl{of zero} over $S$, 
\[
\xymatrix@C+2pc{
(X \otimes R/(\bb), 0) \ar@<-0.8ex>[r]_{1 \otimes \sigma} & (X \otimes \bigwedge F, 1 \otimes \delta_+), \ar@<-0.8ex>[l]_{\pi}
} \quad - 1 \otimes h \, .
\]
Consider the perturbation $\mu$ of the differential on $X \otimes \bigwedge F$. The hypotheses of \cite[Proposition 6.1]{dm1102.2957} are easily checked, and the conclusion is that the desired linear factorisation exists.
\end{proof}

\begin{remark}\label{remark:splitexactsequenceskos}
Suppose that $R$ and $R/(\bb)$ are projective $S$-modules. Then the sequence of $R$-modules
\begin{equation}\label{eq:koszulexactseq}
\xymatrix{
0 \ar[r] & {\bigwedge}^n F \ar[r] & \cdots \ar[r] & {\bigwedge}^0 F = R \ar[r]^-{\pi} & R/(\bb) \ar[r] & 0
}
\end{equation}
is not just exact over $S$, but is even \textsl{split} exact. From the numerous short split exact sequences we construct a map $\sigma$ and homotopy $h$ satisfying (\ref{eq:homotopycond}).
\end{remark}

As has already been mentioned, the next corollary (with a different proof) is due to Dyckerhoff \cite{d0904.4713}. There are related results in the literature on link homology, see \cite[Section 3.2]{r0607544} and \cite[Section 1.2]{w0610650}.

\begin{corollary} If $R$ contains a field then $\pi: \{\bs{a},\bs{b}\} \lto R/(\bb)$ is a stabilisation. That is, for any finite-rank matrix factorisation $Y$ of $W$ the map (\ref{eq:app_stab1}) is a quasi-isomorphism.
\end{corollary}
\begin{proof}
Let $S$ be a field contained in $R$. Then in light of the previous remark we can find $\sigma, h$ satisfying our hypotheses, and Proposition \ref{prop:pertapptechnical} applied to $X = Y\mdual, V = -W$ shows that the bottom row of the following commutative diagram is a homotopy equivalence over $S$: 
\begin{equation}\label{eq:cons3}
\xymatrix@C+2pc{
\Hom_R(Y, \{\ba, \bb\}) \ar[r]^-{\Hom(1,\pi)} & \Hom_R(Y, R/(\bb))\\
Y^{\lor} \otimes \{ \ba, \bb \} \ar[u]^{\cong} \ar[r]_-{1 \otimes \pi} & Y^{\lor} \otimes R/(\bb) \, . \ar@/^2pc/[l]^-{\sigma_\infty} \ar[u]_{\cong}
}
\end{equation}
But then the top row is also a homotopy equivalence over $S$, hence a quasi-isomorphism over $R$.
\end{proof}

Here our interest splits two ways. In Section \ref{subsection:kernelfunctors} we clarify a point about kernel functors and in Section \ref{subsection:liftingmorphisms} we use the formula for $\sigma_\infty$ given in Proposition \ref{prop:pertapptechnical} to lift morphisms $Y \lto R/(\bb)$ to morphisms $Y \lto \{ \ba, \bb \}$.

\subsection{Kernel functors}\label{subsection:kernelfunctors}

We want to interpret $\{ \bs{a}, \bs{b} \}$ and $R/(\bb)$ as functors, and $\pi$ as a natural isomorphism. To this end, suppose that $k$ is a ring, $S,T$ are $k$-algebras and that our ring $R$ above is equal to $T \otimes_k S$. We view $R$-modules as $T$-$S$-bimodules, and make the following hypotheses:
\begin{itemize}
\item $W = V - U$ for elements $U \in T$ and $V \in S$.
\item $R$ is free as an $S$-module, and $R/(\bb)$ is free of finite rank on both sides, i.\,e.~as a right $S$-module and as a left $T$-module.
\end{itemize}
It follows from Remark \ref{remark:splitexactsequenceskos} that the hypotheses of Proposition \ref{prop:pertapptechnical} are satisfied; thus for any finite-rank matrix factorisation $X$ of $U$ over $T$ the map
\[
\pi: X \otimes \{ \bs{a}, \bs{b} \} \lto X \otimes R/(\bb)
\]
is a homotopy equivalence over $S$. Hence the diagram of functors
\[
\xymatrix@C+2.5pc{
\hmf(T, U) \ar[d]_-{\inc}\ar[r]^-{(-) \otimes R/(\bb)} & \hmf(S,V) \ar[d]^-{\inc}\\
\HMF(T, U) \ar[r]_-{(-) \otimes \{ \bs{a}, \bs{b} \}} & \HMF(S,V)
}
\]
commutes up to the natural isomorphism.

\subsection{Lifting morphisms}\label{subsection:liftingmorphisms}

The techniques work more generally, but for simplicity we specialise to the situation of Section~\ref{section:stabilisation} where $S = \QQ[x_1, x_2], R = \QQ[x_1,x_2,y_1,y_2]$ and $W = y_1^{N+1} + y_2^{N+1} - x_1^{N+1} - x_2^{N+1}$. The cyclic Koszul complex of interest is $\Xcirc = \{ \bs{w}, \bs{t} \}$, where $t_i = y_i - x_i$ and the sequence $\bs{w} = (w_1,w_2)$ is defined in~\eqref{eq:cons1}. This means that we take $\bs{a}=\bs{w}$ and $\bb = \bs{t}$ in the above. 

Since we are interested in morphisms $\Xbul \lto \Xcirc$, in addition we take $Y = \Xbul$. We can write $R$ as $S[t_1,t_2]$ so there is an $S$-linear standard (in the sense of \cite[Section 8.1]{dm1102.2957}) flat connection
\begin{gather*}
\nabla: R \lto R \otimes_{S[\boldsymbol{t}]} \Omega^1_{S[\boldsymbol{t}]/S} \, ,\qquad
\nabla( f(x_1,x_2,y_1,y_2) ) = \frac{\partial f}{\partial y_1} \ud t_1 + \frac{\partial f}{\partial y_2} \ud t_2 \, .
\end{gather*}
This induces an $S$-linear map, also denoted $\nabla$, on $R \otimes_{S[\boldsymbol{t}]} \Omega_{S[\boldsymbol{t}]/S}$ where $\Omega = \bigwedge^\bullet \Omega^1$. We identify this free $\mathds{Z}$-graded $R$-module with $\bigwedge F$ via $\ud t_i = \theta_i$. It is easily checked that for $\theta = \theta_{i_1} \ldots \theta_{i_p}$ with distinct indices $i_j$ and $p > 0, f \in R$,
\[
(\nabla \delta_+ + \delta_+ \nabla)(f \cdot \theta) = (p + \delta_{+}\nabla)(f) \cdot \theta \, ,
\]
and moreover the $S$-linear operator $p \cdot 1_R + \delta_{+} \nabla$ on $R$ is invertible. It will be convenient to have notation for the inverse operator on $R$.

\begin{definition}\label{defn:omega} For $p > 0$ we write $\Omega_p = (p \cdot 1_R + \delta_{+} \nabla)^{-1}$. So by definition $\Omega_p(f)$ is the unique polynomial solution of the partial differential equation in an unknown $z$,
\[
p z + \frac{\partial z}{\partial y_1} (y_1 - x_1) + \frac{\partial z}{\partial y_2}(y_2 - x_2) = f \, .
\]
\end{definition}

\begin{example}\label{example:compomega} For a monomial $f = x_1^{c_1} x_2^{c_2} t_1^{d_1} t_2^{d_2}$ we have by direct computation $\Omega_p(f) = \frac{1}{p + d_1 + d_2} f$. Hence in particular
\begin{itemize}
\item[(i)] If $f \in k[x_1,x_2]$ then $\Omega_p(f) = \frac{1}{p} f$,
\item[(ii)] $\Omega_p(y_i) = \Omega_p(x_i + t_i) = \frac{1}{p(p+1)} x_i + \frac{1}{p+1} y_i$.
\end{itemize}
\end{example}

One defines an $S$-linear map
\[
H = (\nabla \delta_+ + \delta_+ \nabla)^{-1} \nabla
\]
of degree $-1$ on the module $\bigwedge F$. Each application of $H$ involves differentiation with respect to the $y_i$, and then solving a differential equation. Let $\sigma: S = R/(\boldsymbol{t}) \lto \bigwedge F$ be the inclusion of $\QQ[x_1,x_2]$ into $R$. One checks (see \cite[Section 8.1]{dm1102.2957}) that $H \delta_+ + \delta_+ H = 1 - \sigma \pi$, so that $\sigma$ is exhibited by $H$ as the $S$-linear homotopy inverse to the morphism $\pi: (\bigwedge F, \delta_+) \lto (R/(\boldsymbol{t}),0)$.

As explained above, we tensor with $Y\mdual$ and perturb in the other differentials. The first step is to extend $H$ and $\sigma$ to $S$-linear maps
\[
\xymatrix@C+2pc{
Y\mdual \otimes \bigwedge F \ar@(dl,dr)[]_{H} & Y\mdual \otimes R/(\boldsymbol{t}) \ar[l]_{\sigma}
}
\]
using the homogeneous basis $1^*, \theta_1^*, \theta_2^*, (\theta_1\theta_2)^*$ for $Y\mdual$. This involves Koszul signs as we move~$H$ and~$\sigma$ past homogeneous elements of $Y\mdual$, for example, $H(\theta_1^* \otimes f \theta_2) = - \theta_1^* \otimes H(f \theta_2)$. 

Next we consider the $R$-linear map
\[
\mu = d_{Y\mdual} \otimes 1 + 1 \otimes \delta_{-}
\]
on $Y\mdual \otimes \bigwedge F$, viewed as a perturbation of the differential on $(Y\mdual \otimes \bigwedge F, 1 \otimes \delta_{+})$. The total differential $\mu + 1 \otimes \delta_{+}$ is the usual differential on the tensor product, and the perturbation lemma tells us that the $S$-linear map $\sigma_\infty = \sum_{m \ge 0} (-1)^m (H \mu)^m \sigma$ is an $S$-homotopy inverse to $1 \otimes \pi$. Composing with the vertical isomorphisms in (\ref{eq:cons3}) we have the desired inverse of $\Hom(1,\pi)$. The morphism (\ref{eq:cons4}) corresponds to the tensor $1^* \otimes 1 \in Y\mdual \otimes R/(\boldsymbol{t})$, so to find $\chi_1$ we are left with the task of evaluating
\[
\sigma_\infty(1^* \otimes 1) = \sigma(1^* \otimes 1) - H \mu \sigma(1^* \otimes 1) + H \mu H \mu(1^* \otimes 1) \, .
\]

\begin{proof}[Proof of Lemma \ref{lemma:consfirst}]
The map corresponding to $\sigma_\infty(1^* \otimes 1)$ makes (\ref{eq:cons5}) commute, so we need to evaluate this tensor and show that it has the explicit form given in the statement of the lemma. The notation becomes a little involved: throughout recall that we are producing tensors in $Y\mdual \otimes \bigwedge F$, an $R$-basis of which is given by tensors $\theta_1^* \otimes 1, 1^* \otimes \theta_1\theta_2$, etc. Clearly $\sigma(1^* \otimes 1)$ is just $1^* \otimes 1$, and
\begin{align*}
H \mu \sigma(1^* \otimes 1) &= H( d_{Y\mdual} + \delta_{-})(1^* \otimes 1)\\
&= H( -\theta_1^* \otimes s_1 - \theta_2^* \otimes s_2 + 1^* \otimes w_1 \theta_1 + 1^* \otimes w_2 \theta_2 )\\
&= \theta_1^* \otimes H(s_1) + \theta_2^* \otimes H(s_2) + 1^* \otimes H(w_1 \theta_1) + 1^* \otimes H(w_2 \theta_2) \, .
\end{align*}
By easy computations, for example
\begin{align*}
H(s_1) &= \Omega_1\Big( \frac{\partial s_1}{\partial y_1} \Big) \theta_1 + \Omega_1\Big( \frac{\partial s_1}{\partial y_2} \Big) \theta_2 = \Omega_1(1) \theta_1 + \Omega_1(1) \theta_2 = \theta_1 + \theta_2 \, ,
\end{align*}
and
\begin{align*}
H(w_1 \theta_1) &= (\nabla \delta_{+} + \delta_{+} \nabla)^{-1} \nabla( w_1 \theta_1 ) = (\nabla \delta_{+} + \delta_{+} \nabla)^{-1}\Big( \frac{\partial w_1}{\partial y_2} \theta_2 \theta_1 \Big) = 0
\end{align*}
we find that
$$
H \mu \sigma(1^* \otimes 1) = \theta_1^* \otimes (\theta_1 + \theta_2) + \frac{1}{2}\theta_2^* \otimes ( (x_2 + y_2) \theta_1 + (x_1 + y_1) \theta_2) \, .
$$
Applying $H\mu$ again we obtain after some work an expression for $H\mu H\mu\sigma(1^* \otimes 1)$, which we combine with the previous equation to find
\begin{align*}
\sigma_\infty(1^* \otimes 1) &= 1^* \otimes 1 + 1^* \otimes \Omega_2\Big( \frac{\partial u_1}{\partial y_1} - \frac{\partial u_1}{\partial y_2} - \frac{1}{2} \frac{\partial u_2}{\partial y_2}(x_2 + y_2) + \frac{1}{2} \frac{\partial u_2}{\partial y_1}(x_1 + y_1) \Big)\theta_1 \theta_2  \\
& \qquad+  \frac{1}{2}(\theta_1 \theta_2)^* \otimes (x_1 + y_1 - x_2 - y_2) \theta_1 \theta_2  \\
& \qquad- \theta_1^* \otimes (\theta_1 + \theta_2) - \frac{1}{2} \theta_2^* \otimes ( (x_2+y_2)\theta_1 + (x_1+y_1)\theta_2 ) \, . 
\end{align*}
The image under the canonical isomorphism $\xi: \Xbul\mdual \otimes \Xcirc \lto \Hom_R(\Xbul, \Xcirc)$ is our morphism $\chi_1 := \xi \sigma_\infty(1^* \otimes 1)$, which has the desired form (the signs get corrected by Koszul signs in $\xi$).
\end{proof}

\newcommand{\etalchar}[1]{$^{#1}$}
\providecommand{\href}[2]{#2}

\end{document}